\documentclass[hidelinks,onefignum,onetabnum]{siamart190516}

\usepackage[ntheorem]{empheq}

\usepackage{lipsum}
\usepackage{amsfonts}
\usepackage{graphicx}
\usepackage{epstopdf}
\usepackage{algorithmic}
\ifpdf
  \DeclareGraphicsExtensions{.eps,.pdf,.png,.jpg}
\else
  \DeclareGraphicsExtensions{.eps}
\fi

\newsiamremark{remark}{Remark}
\newsiamremark{hypothesis}{Hypothesis}
\crefname{hypothesis}{Hypothesis}{Hypotheses}
\newsiamthm{claim}{Claim}

\headers{Inverse scattering by a locally rough interface}{L. Li, J. Yang, B. Zhang, and H. Zhang}

\title{Direct imaging methods for reconstructing  a locally rough interface from phaseless total-field data or phased far-field data}

\author{Long Li\thanks{Academy of Mathematics and Systems Science, Chinese Academy of Sciences, Beijing 100190,
China ({\tt longli@amss.ac.cn})}
\and
{Jiansheng} Yang\thanks{LMAM, School of Mathematical Sciences, Peking University, Beijing 100871,
China ({\tt jsyang@pku.edu.cn})}
\and
Bo Zhang\thanks{LSEC and Academy of Mathematics and Systems Science, Chinese Academy of
Sciences, Beijing, 100190, China and School of Mathematical Sciences, University of Chinese
Academy of Sciences, Beijing 100049, China ({\tt b.zhang@amt.ac.cn})}
\and
Haiwen Zhang\thanks{Corresponding author. Academy of Mathematics and Systems Science, Chinese Academy of Sciences, Beijing 100190, China ({\tt zhanghaiwen@amss.ac.cn})}}

\usepackage{amsopn}

\usepackage{latexsym,amssymb,mathrsfs,amssymb}
\usepackage{subfigure}
\usepackage{enumerate}
\usepackage{cases}
\usepackage{empheq}

\newcommand{\R}{{\mathbb R}}
\newcommand{\II}{\mathbb I}
\newcommand{\RR}{{\mathcal{R}}}
\newcommand{\T}{{\mathcal T}}

\newcommand{\C}{{\mathbb C}}
\newcommand{\s}{{\mathbb S}}

\newcommand{\Sp}{{\mathbb S}}

\newcommand{\ds}{\displaystyle}
\newcommand{\no}{\nonumber}
\newcommand{\be}{\begin{eqnarray}}
\newcommand{\ben}{\begin{eqnarray*}}
\newcommand{\en}{\end{eqnarray}}
\newcommand{\enn}{\end{eqnarray*}}
\newcommand{\ba}{\backslash}
\newcommand{\pa}{\partial}

\newcommand{\ov}{\overline}
\newcommand{\se}{\setminus}

\newcommand{\Grad}{{\rm Grad\,}}

\newcommand{\Ima}{{\rm Im\,}}

\newcommand{\G}{\Gamma}

\newcommand{\la}{\lambda}

\newcommand{\Rt}{{\rm Re}}
\newcommand{\wid}{\widetilde}

\makeatletter

\newcommand{\Rmnum}[1]{\expandafter\@slowromancap\romannumeral #1@}
\makeatother

\ifpdf
\hypersetup{
  pdftitle={Direct imaging methods for reconstructing  a locally rough interface from phaseless total-field data or phased far-field data},
  pdfauthor={Long Li, JianSheng Yang, Bo Zhang and Haiwen Zhang}
}
\fi

\begin{document}

\maketitle

\begin{abstract}
This paper is concerned with the problem of inverse scattering of time-harmonic acoustic plane waves
by a two-layered medium with a locally rough interface in 2D. A direct imaging method is proposed to
reconstruct the locally rough interface from the phaseless total-field data measured on the upper half
of the circle with a large radius  at a fixed frequency or from the phased far-field data measured
on the upper half of the unit circle at a fixed frequency.
The presence of the locally rough interface poses challenges in the theoretical analysis of
the imaging methods. To address these challenges, a technically involved asymptotic analysis is provided
for the relevant oscillatory integrals involved in the imaging methods, based mainly on the techniques
and results in our recent work [L. Li, J. Yang, B. Zhang and H. Zhang, arXiv:2208.00456] on
the uniform far-field asymptotics of the scattered field for acoustic scattering in a two-layered medium.
Finally, extensive numerical experiments are conducted to demonstrate the feasibility and robustness
of our imaging algorithms.
\end{abstract}



\section{Introduction}

In this paper, we consider the problem of inverse scattering of time-harmonic acoustic plane waves in
a two-layered medium with a locally rough interface in 2D. The background two-layered medium
is composed of two unbounded media with different physical properties. The interface between
the two media is considered to be a local perturbation with a finite height from a planar surface
over a finite interval.
Such problems occur in a broad spectrum of science and engineering, such as remote sensing,
ocean acoustics, geophysical exploration and nondestructive testing.

Many numerical algorithms have been proposed for recovering impenetrable or penetrable locally rough
surfaces from the scattered-field data or far-field data. In \cite{BGL_11}, a continuation approach
using a series of wave frequencies was proposed for reconstructing locally rough surfaces with Dirichlet
boundary conditions. Newton iteration methods with multiple wave frequencies were developed
in \cite{QZZ19,ZZ_13} for recovering locally rough surfaces with Dirichlet or Neumann boundary conditions.
In \cite{jgz_17}, a Kirsch-Kress method was developed for reconstructing penetrable locally rough surfaces.
Further, linear sampling methods for recovering sound-soft or penetrable locally rough surfaces were
proposed in \cite{DJY_17,LYZ2021,LYZ2022}. Recently, a reverse time migration method was proposed
in \cite{LY2023I} for reconstructing sound-soft, sound-hard or penetrable locally rough surfaces
from incident point sources. This method has also been extended to simultaneously recover penetrable
locally rough surfaces and buried obstacles in \cite{LY2023II}.
Moreover, there are also some numerical studies concerning inverse scattering by an unbounded rough
surface (i.e., the case when the surface is a nonlocal perturbation of an infinite plane);
see \cite{BL_13, BL_14, BP_10,ChenS,Lzzd,MS_921,MS_922,RJ_91,Z_20}.

In many practical applications, obtaining the phase information of the wave fields is much harder than
acquiring the intensity (or the modulus) information of the wave fields.
Thus it is often desirable to study inverse scattering with phaseless data.
Some work has been made to develop numerical algorithms for recovering locally rough surfaces
or unbounded rough surfaces from phaseless data. In \cite{Bao_13}, an efficient continuation
method using a series of wave frequencies was developed to reconstruct the shapes of periodic diffraction
profiles from phaseless near-field data. A recursive Newton iteration algorithm with multiple wave
frequencies was proposed in \cite{Bao_16} to recover the shapes of multi-scale rough surfaces from
phaseless near-field data. By using superpositions of two plane waves with different directions as
the incident fields, a recursive Newton iteration algorithm in frequencies
was developed in \cite{ZZ17} to determine the shape and location of locally rough surfaces from
phaseless far-field data. Recently, a direct imaging method was proposed in \cite{xzz19} to recover
locally rough surfaces from phaseless total-field data corresponding to incident plane waves
at a fixed frequency. Further, an iterated marching method based on the parabolic integral equation was
developed in \cite{ChenOS} to recover unbounded rough surfaces from phaseless single frequency data
at grazing angles. It is worth mentioning that all the above work only considered the case of impenetrable
rough surfaces, and few work is available for numerically recovering penetrable rough surfaces with
phaseless data.
For more works on the mathematical and numerical studies (including uniqueness and inversion algorithms)
of relevant inverse scattering problems with phaseless data, we refer
to \cite{Chen1,JLZ19b,JLZ19c,IK11,KM1,KM2,NO2,xzz,G1} and the references therein.

In this paper, we develop two non-iterative numerical methods for our inverse problem of recovering
the locally rough interface from the measurement data corresponding to incident plane waves
at a fixed frequency.
Precisely, we propose two direct imaging methods to reconstruct the locally rough interface from phaseless
total-field data measured on the upper half of the circle with a large radius $R$, based on the imaging
function $I_P(z,R)$ with $z\in \R^2$ (see formula (\ref{eq34}) below), and
from phased far-field data measured on the upper half of the unit circle, based on the imaging
function $I_F(z)$ with $z\in \R^2$ (see formula (\ref{c:8}) below).
The work in this paper is a non-trivial extension of the work \cite{xzz19} from the case of sound-soft
locally rough surfaces to the case of penetrable locally rough surfaces.
In fact, due to the presence of the two-layered background medium, the reflected wave and the scattered
wave for the scattering problem considered in this paper are much more complicated than those for
the scattering problem considered in \cite{xzz19}
(cf. \cite[formula (1.4) and Lemma 2.1]{xzz19}, formula \eqref{eq3} and Lemma \ref{Le:my}).
This then leads to difficulties in the theoretical analysis of the proposed direct imaging methods.
To overcome these difficulties, we provide a technically involved asymptotic analysis for the relevant
oscillatory integrals. It is worth mentioning that our recent work \cite{LYZZ3} on
the uniform far-field asymptotics of the scattered wave for the two-layered medium scattering
problem provides a theoretical foundation for the proposed methods.
From the theoretical analysis, it is expected that both $I_P(z,R)$ with sufficiently large $R$
and $I_F(z)$ will take a large value when the sampling point $z$ is on the locally rough interface
and decay as $z$ moves away from the locally rough interface. Based on these properties, a direct
imaging algorithm with phaseless total-field data and a direct imaging algorithm with phased far-field
data are given to recover the locally rough interface (see Algorithm \ref{A1} and Algorithm \ref{A2} below).
A main feature of our algorithms is that only inner products are needed to compute the imaging functions
and thus they are very cheap in computation.
Finally, numerical examples are carried out to show that our imaging methods can provide an
accurate and reliable reconstruction of the locally rough interface even for the case
of multiple-scale profiles and that our imaging methods are very robust to noises.
To the best of our knowledge, the present paper is the first attempt to develop a non-iterative method
with phaseless total-field data and a non-iterative method with phased
far-field data for recovering locally rough interfaces.

The remaining part of the paper is organized as follows. In Section \ref{sec3},
we introduce the forward and inverse scattering problems under consideration.
In Section \ref{sec:3}, we propose the direct imaging method with phaseless total-field data
and the direct imaging method with phased far-field data for the considered inverse
scattering problems. The theoretical analysis of these methods is also given in
Section \ref{sec:3}. Numerical experiments are conducted in Section \ref{num} to
illustrate the performance of our imaging methods. Finally,
some concluding remarks are given in Section \ref{con}.

\section{The forward and inverse scattering problems}\label{sec3}

In this section, we present the considered forward and inverse scattering problems
in a two-layered medium with a locally rough interface.
We restrict our attention to the two-dimensional case by assuming that the
local perturbation is invariant in the $x_3$ direction.
First, we introduce some notations which will be used throughout the paper.
Let $\Gamma:=\{(x_1,x_2): x_2=h_\Gamma(x_1), x_1\in \R \}$ represent a locally rough surface,
where $h_\Gamma\in C^2(\R)$ has a compact support in $\R$.
Let $\Gamma_p:=\{(x_1,x_2): x_2=h_{\Gamma}(x_1), x_1\in {\rm Supp}(h_{\Gamma})\}$ denote the
local perturbation of $\Gamma$.
Let $\Omega_{\pm}:=\{(x_1,x_2):x_2\gtrless h_\Gamma(x_1),x_1\in \mathbb{R}\}$ denote the homogenous
media above and below $\Gamma$, respectively.
Let $k_\pm={\omega}/{c_\pm}>0$ be two different wave numbers in $\Omega_\pm$, respectively,
with $\omega$ being the wave frequency and $c_\pm$ being the wave speeds in the homogenous media
$\Omega_{\pm}$, respectively.
Define $n:={k_{-}}/{k_{+}}$. Let $\Sp^1_\pm:=\{x = (x_1, x_2)\in\mathbb{R}^2 :|x|=1,x_2 \gtrless 0\}$
denote the upper part and lower part of the unit circle, respectively.
Let $B_R:=\big\{x\in \mathbb{R}^2: \vert x\vert < R \big\}$ be a disk with radius $R>0$.
We will always assume that $R>0$ is large enough so
that the local perturbation ${\Gamma_p} \subset  {B_R}$.
Define $\partial B^+_R:=\partial B_R\cap\Omega_+$.
For any $x\in\mathbb{R}^2$, let $x=(x_1,x_2)$ and $x'=(x_1,-x_2)$.
For any $x\in\mathbb{R}^2$ with $\vert x \vert\neq0$,
let $\hat{x}=x/\vert x \vert=(\cos\theta_{\hat{x}},\sin\theta_{\hat{x}})$ with the angle $\theta_{\hat{x}}\in \left[\left.0,2\pi\right)\right.$.
For any positive integer $\ell$, let $H^{\ell}_{loc}(\mathbb{R}^2)$
be the space of all functions $\phi:\mathbb{R}^2\rightarrow \mathbb{C}$ such that $\phi\in H^{\ell}(B)$
for all open balls $B\subset\mathbb{R}^2$.
For any $t\in\mathbb{R}$ and $a>0$, let $\mathcal S(t,a):= \mathcal S_{1}(t-a) \mathcal S_{2}(t+a)$, where
$\mathcal S_{1}(s)$ and $\mathcal S_{2}(s)$ with $s\in\mathbb{R}$ are defined by
\begin{align*}
\mathcal S_{1}(s):=\left\{
\begin{aligned}
  &\sqrt{|s|},&&s>0,\\
  &-i\sqrt{|s|},&&s\leq0,
\end{aligned}
\right.
\qquad
\mathcal S_{2}(s):=\left\{
\begin{aligned}
  &\sqrt{|s|},&&s>0,\\
  &i\sqrt{|s|},&&s\leq0.
\end{aligned}
\right.
\end{align*}
It can be seen that for any $t\in\mathbb{R}$ and $a>0$,
\begin{align}\label{eq1}
\mathcal{S}(t,a) =\left\{
\begin{aligned} &-i\sqrt{a^2-t^2}&&\textrm{if}~  a^{-1}|t|\leq 1, \\
  &\sqrt{t^2-a^2}&&\textrm{if}~ a^{-1}|t|> 1.
\end{aligned}\right.
\end{align}
We note that for any fixed $a>0$, the function $\mathcal S(\cdot,a)$
can be continued analytically to the complex plane slit along two half-lines
$\{z\in \C: \Rt(z)=a,\; \Ima(z)\ge 0\}$ and $\{z\in \C: \Rt(z)=-a,\;\Ima(z)\le 0\}$
(see \cite[Section 2]{LYZZ3} for more details of the function $\mathcal S(\cdot,\cdot)$).

Consider the time-harmonic ($e^{-\omega t}$ time dependence) incident acoustic
plane wave $u^{i}(x,d):=e^{ik_{+}x\cdot d}$ propagating in the direction $d
=(\cos \theta_d, \sin \theta_d)
\in\Sp^1_-$ with $\theta_d \in(\pi, 2\pi)$.
Then the total field $u^{tot}(x,d)=u^{0}(x,d)+u^{s}(x,d)$ is the sum of the
reference wave $u^0(x,d)$ and the scattered field $u^s(x,d)$.
The reference wave $u^0(x,d)$ is generated by the incident field $u^i(x,d)$ and the two-layered
medium, and is given by (see, e.g., (2.13a) and (2.13b) in \cite{Car17} or Section 4 in \cite{LYZZ3})
\begin{align*}
u^0(x,d):=\begin{cases}
\ds u^i(x,d)+ u^r(x,d),\quad & x\in \R^2_{+},\\
\ds u^t(x,d), \quad & x\in  \R^2_{-},\\
\end{cases}
\end{align*}
where $\R^2_{+}:=\{{(x_1,x_2)}\in\R^2:x_2>0\}$
and $\R^2_{-}:=\{{(x_1,x_2)}\in\R^2:x_2<0\}$ denote
the upper and lower half-spaces, respectively,
and the reflected wave $u^r(x,d)$ and transmitted wave $u^t(x,d)$ are given by
\be\label{eq3}
u^r(x,d):=\mathcal{R}(\pi+\theta_d)e^{ik_{+}x\cdot d^r},\quad
u^t(x,d):=\mathcal{T}(\pi+\theta_d)
e^{ik_{-}x\cdot d^t}.
\en
Here,
$d^{r}:=(\cos\theta_d,-\sin\theta_d)$ is the reflection direction in $\mathbb{S}^1_+$ and
$d^t$ is given by
\ben
d^t:=n^{-1}(\cos\theta_d,-i\mathcal{S}(\cos\theta_d,n)).
\enn
In particular, we can see from (\ref{eq1}) that if $n^{-1}\vert \cos\theta_d\vert \leq1$, then $d^t=(\cos \theta^t_{d},\sin{\theta^t_{d}})$
is the transmission direction in $\mathbb{S}^1_-$ with $\theta^t_{d}\in[\pi,2\pi]$
satisfying $\cos\theta^t_{d}=n^{-1}\cos\theta_d$.
Further,
$\RR(\pi+\theta_d)$ and $\T(\pi+\theta_d)$ in \eqref{eq3} are called
the reflection and transmission coefficients, respectively,
with
$\mathcal{R}$ and $\mathcal{T}$ defined by
\be\label{eq:0}
\mathcal R(\theta):=\frac{{i\sin\theta+\mathcal S(\cos\theta,n)}}{{i\sin\theta-\mathcal S{(\cos\theta,n)}}},
\quad \T(\theta):={\RR(\theta)}+1\quad\textrm{for}~ \theta\in\mathbb{R}.
\en
It is easily seen that for any $d\in\mathbb{S}^1_-$, the reference wave $u^0(x,d)\in H^1_{loc}(\R^2)$
and $u^0(x,d)$ satisfies the Helmholtz equations by the unperturbed two-layered medium
together with the transmission condition on $\Gamma_0:=\{(x_1,0):x_1\in\R\}$, that is,
\begin{align}
\qquad\qquad\ds\Delta u^{0} +{k}^2_\pm u^{0}=0&&& \text{in}\quad\R^2_\pm,\qquad\qquad \label{eq:0.1}\\
\ds[u^{0}]=0,\;\;\left[{\partial u^{0}}/{\partial\nu}\right]=0&& &\text{on}\quad\Gamma_0,\label{eq:0.2}
\end{align}
where $\nu$ denotes the unit normal on $\Gamma_0$ pointing into $\R^2_+$ and
$[\cdot]$ denotes the jump across the interface $\Gamma_0$.
Moreover, the total
field $u^{tot}(x,d)$ and the scattered field $u^s(x,d)$
satisfy the following scattering problem in the two-layered medium with the
locally rough interface $\Gamma$
\begin{align} \label{e:1.1}
\qquad\qquad\ds\Delta u^{{tot}} +{k}^2_\pm u^{tot}=0&&& \text{in}\quad\Omega_\pm,\\
\label{eq5}
\ds[u^{tot}]=0,\;\;\left[{\partial u^{tot}}/{\partial\nu}\right]=0&&&\text{on}\quad\Gamma,\\
\label{eq2}
{\ds\lim_{|x|\rightarrow+\infty}\sqrt{|x|}\left(\frac{\partial u^s}{\partial |x|}-ik_\pm u^s\right)=0}&&& \textrm{uniformly~for~all}~\hat{x} \in \mathbb{S}^1_\pm,
\end{align}
where $\nu$ denotes the unit normal on $\Gamma$ pointing into $\Omega_+$,
$[\cdot]$ denotes the jump across the interface $\Gamma$,
(\ref{e:1.1}) is the Helmholtz equation
and (\ref{eq2}) is the well-known Sommerfeld radiation condition.
See Figure \ref{sca} for the problem geometry.

\begin{figure}
 \centering
 \includegraphics[width=0.6\textwidth]{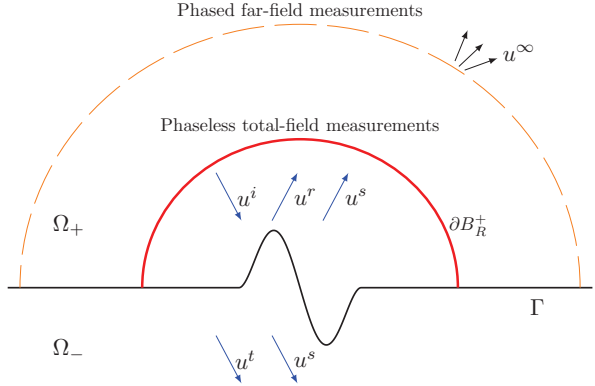}
 \caption{Direct and inverse scattering problems in a
 two-layered medium with the locally rough interface $\Gamma$.}\label{sca}
\end{figure}

The following theorem presents
the well-posedness of the scattering problem (\ref{e:1.1})--(\ref{eq2}), which is a
direct consequence of Theorem 2.5 in \cite{BHY}.
Throughout the paper, we assume that the total field $u^{tot}(x,d)$ and the scattered
field $u^s(x,d)$ are given in the sense of Theorem \ref{thm2}.
See also \cite{YLZ22} for the well-posedness of the two-layered medium scattering problem.

\begin{theorem}[see Theorem 2.5 in \cite{BHY}]\label{thm2}
For any $d\in\mathbb{S}^1_-$,
there exists a unique solution $u^s(x,d)\in H^1_{loc}(\mathbb{R}^2)$ such that
$u^{tot}(x,d):=u^0(x,d)+u^s(x,d)\in H^1_{loc}(\mathbb{R}^2)$, and the total field $u^{tot}(x,d)$ and the scattered field $u^s(x,d)$ solve the scattering problem (\ref{e:1.1})--(\ref{eq2}).
\end{theorem}

Moreover, we proved in \cite{LYZZ3} that the scattered wave $u^s(x,d)$
has the following asymptotic behavior: for any $d\in\mathbb{S}^1_-$,
\begin{align}\label{eq4}
u^s(x,d) = \frac{e^{ik_{+}|x|}}{\sqrt{|x|}}u^{\infty}(\hat x,d) + u^s_{Res}(x,d), \qquad
|x|\rightarrow +\infty,\qquad x\in\Omega_+,
\end{align}
with the residual term $u^s_{Res}(x,d)$ satisfying
$u^s_{Res}(x,d)=O({{|x|}}^{-3/4})$ as $|x|\rightarrow+\infty$ uniformly for all angles $\theta_{\hat{x}}\in(0,\pi)$,
where
$u^{\infty}(\hat{x},d)$ is called the far-field pattern of the scattered field $u^s(x,d)$
and is given by
\begin{align}\label{eq29}
u^\infty(\hat x,d)= \int_{\partial{{B_R}}}\left[{\frac{\partial G^\infty(\hat{x},y)}{\partial\nu(y)}} u^s(y,d)- \frac{\partial u^s(y,d)}{\partial\nu(y)} G^{\infty}(\hat{x},y) \right]ds(y),\quad \hat x \in \mathbb S^1_+.
\end{align}
Here, $G^{\infty}(\hat{x},y)$ is defined as follows: for any $\hat x=(\cos \theta_{\hat x}, \sin \theta_{\hat x}) \in \Sp^1_+$ and $y=(y_1,y_2)\in \R^2_+\cup \R^2_-$,
\begin{align*}
G^{\infty}(\hat x,y):= \frac{e^{i\frac{\pi}4}}{\sqrt{8\pi k_{+}}}
\left\{\begin{aligned}
&e^{-i k_{+}{\hat x} \cdot y} + \mathcal R(\theta_{\hat x}) e^{-ik_{+}{\hat x}\cdot y{'}},
&&\hat{x}\in\mathbb{S}^1_+,~y\in\mathbb{R}^2_+,\\
&\mathcal T(\theta_{\hat x})e^{-ik_{+}(y_1\cos\theta_{\hat x}+iy_2\mathcal S{(\cos\theta_{\hat x},n)})},
&&\hat{x}\in\mathbb{S}^1_+,~y\in\mathbb{R}^2_-.\\
\end{aligned}\right.
\end{align*}
In \cite{BL16,Car17}, it was proved that for any $d\in\mathbb{S}^1_-$,
the residual term $u^s_{Res}(x,d)$ in (\ref{eq4}) satisfies $u^s_{Res}(x,d)=O({{|x|}}^{-3/2})$ as $|x|\rightarrow+\infty$ for all angles
$\theta_{\hat{x}}\in(0,\pi)$ except possibly
for certain critical angles (we note that there is no critical angle in $(0,\pi)$ in the case $k_+<k_-$
and that there are only two critical angles $\theta_c$ and $\pi-\theta_c$ in $(0,\pi)$ with $\theta_c:=\arccos(k_-/k_+)\in(0,\pi/2)$ in the case $k_+>k_-$).
See Remarks 4 and 5 in \cite{LYZZ3} for discussions on the critical angles.
Further, in \cite{LYZZ3} we have established the uniform far-field asymptotics of the
scattered field $u^s(x,d)$; see also Lemma \ref{Le:my} below for some of our results in \cite{LYZZ3}.
For more discussions on the far-field asymptotic properties of
the scattered field $u^s(x,d)$, we refer to \cite{BL16,Car17,LYZZ3}.

We note that the well-posedness of the direct scattering problem in a two-layered medium with
a general unbounded rough interface (that is, the interface is a nonlocal perturbation
of an infinite plane) has been studied in \cite{DTS,CSZ1999,LI2016}, where
the scattered field is required to satisfy the upward and downward propagating radiation
conditions instead of the Sommerfeld radiation condition.
We mention that the well-posedness of this kind of scattering problem will be used for the theoretical
analysis of our inversion algorithms in Section \ref{sec:3}.

In this paper, we focus on the following two inverse problems (see Figure \ref{sca}).

{\bf Inverse problem with phaseless total-field data (IP1)}: Given the incident plane wave
$u^{i}(x,d)$ with fixed wave number $k_+$, reconstruct the shape and location of the
penetrable locally rough surface $\Gamma$ from the phaseless total-field data
$|u^{tot}(x,d)|$ for all  $x \in {\partial{B^{+}_R}}$ and $d\in \Sp^1_{-}$.

{\bf Inverse problem with phased far-field data (IP2)}: Given the incident plane wave $u^{i}(x,d)$
with fixed wave number $k_+$, reconstruct the shape and location of the  penetrable locally rough
surface $\Gamma$ from the phased far-field data $u^\infty(\hat{x},d)$ for all
$\hat{x}\in\Sp^1_+$ and $d\in \Sp^1_{-}$.

\section{Direct imaging methods for the inverse problems}\label{sec:3}

In this section, we will develop direct imaging methods for the inverse problems (IP1) and (IP2).
For this aim, we introduce some notations which will be used in the rest of the paper.
For the case $k_+>k_-$,
let $\theta_c\in(0,\pi/2)$ be the angle defined as in Section \ref{sec3}.
For any $\theta\in\mathbb{R}$, let $\mathcal{R}_0(\theta):=\mathcal{R}(\theta+\pi)$
with the function $\RR$ given in (\ref{eq:0}).
It is easily seen that for both the cases $k_+<k_-$ and $k_+>k_-$,
the function $\mathcal{R}_0$ and its (distributional) derivative satisfy
\begin{align}\label{eq28}
\|\mathcal R_0(\cdot)\|_{C[\pi,2\pi]}\leq C,\quad \|\mathcal R'_0(\cdot)\|_{L^1(\pi,2\pi)}\le C
\end{align}
with some constant $C>0$.
For any $d\in\mathbb{S}^1_-$, let $d=(d_1,d_2)$.
Throughout the paper, the constants may be different at different places.

For the inverse problem (IP1), we introduce the following imaging function: for $z\in\mathbb{R}^2$,
\begin{align}
I_{P}(z,R) : =&
 \int_{{\partial{B^{+}_R}}}\bigg|\int_{\Sp_{-}^1}\Big\{\Big[|u^{tot}(x,d)|^2-
\big(1+|\RR_0(\theta_d)|^2+\overline{\RR_0(\theta_d)}e^{2ik_{+}x_2d_2}\big)
\Big]
e^{ik_{+}(x-z)\cdot d}\no\\
&-e^{ik_{+}(x{'}-z{'})\cdot d }\Big\}ds(d)\bigg|^2 ds(x),\label{eq34}
\end{align}
where $x'$ and $z'$ is given as in Section \ref{sec3}.
For the inverse problem (IP2),
we introduce the following imaging function: for $z\in\mathbb{R}^2$,
\begin{align} \label{c:8}
&I_F(z):=\no\\
&\int_{\s^1_{+}}\left|\int_{\s^1_{-}}u^{\infty}(\hat x, d)e^{-ik_{+}z\cdot d} ds(d)+{\left(\frac{2\pi}{k_{+}}\right)}^{\frac 12}e^{{-i\pi}/{4}}\left(\RR(\theta_{\hat x})
e^{-ik_{+}\hat x\cdot z{'}}-e^{{-ik_{+}}\hat x \cdot z}\right)\right|^2ds(\hat x).
\end{align}
In Sections \ref{sec1} and \ref{sec:4}, we will study the asymptotic
property of $I_{P}(z,R)$ as $R\rightarrow+\infty$ and the property of $I_F(z)$, respectively,
by analyzing the asymptotic properties of relevant oscillatory integrals.
In doing so, an essential role is played by
the uniform far-field asymptotic properties of the scattered field $u^s(x,d)$ obtained in our work \cite{LYZZ3}.
Based on the results in Sections \ref{sec1} and \ref{sec:4}, we will propose the direct imaging methods for
the inverse problems (IP1) and (IP2) in Section \ref{sec2}.
It should be remarked that currently there is no uniqueness result for the inverse problems (IP1) and (IP2).
However, the numerical examples carried out in Section \ref{num} show that
our inversion algorithms can provide a satisfactory reconstruction of
the locally rough surface $\Gamma$.

\subsection{Asymptotic property of the imaging function $I_P(z,R)$}\label{sec1}

We will study the asymptotic property of the imaging function $I_P(z,R)$ given in \eqref{eq34}
when the radius $R$ is sufficiently large.
For $x\in\Omega_+$ and $z\in\mathbb{R}^2$, define
\begin{align*}
& U_1(x,z):= \int_{\Sp_{-}^1}u^s(x,d)e^{-ik_{+}z\cdot d} ds(d), \\
& U_2(x,z):= \int_{\Sp_{-}^1}\RR_0(\theta_d)e^{ik_{+}(x{'}-z)\cdot d} ds(d),\\
& U_3(x,z):= \int_{\Sp_{-}^1}-e^{ik_{+}(x{'}-z{'})\cdot d} ds(d),
\end{align*}
and
\begin{align*}
& W_{1}(x,z):= \int_{\Sp_{-}^1}\ov{{u}^s(x,d)}e^{2ik_{+}x\cdot d} e^{-ik_{+}z\cdot d}ds(d),\\
& W_{2}(x,z):= \int_{\Sp_{-}^1}\ov{{u}^s(x,d)}\RR_0(\theta_d)e^{2ik_{+}x_1d_1} e^{-ik_{+}z\cdot d}ds(d),\\
& W_{3}(x,z):= \int_{\Sp_{-}^1}u^s(x,d)\overline{\RR_0(\theta_d)}e^{2ik_{+}x_2d_2} e^{-ik_{+}z\cdot d}ds(d),\\
& W_{4}(x,z):= \int_{\Sp_{-}^1}|u^s(x,d)|^2e^{ik_{+}(x-z)\cdot d} ds(d).
\end{align*}
Further, for $x\in\Omega_+$ and $z\in\mathbb{R}^2$, let $U(x,z)$ and $W(x,z)$ be given by
\begin{align}
& U(x,z):= \int_{\Sp_{-}^1}[u^{tot}(x,d)-u^{i}(x,d)]e^{-ik_{+}z\cdot d} ds(d) +U_3(x,z),\label{eq38}\\
& W(x,z):= W_{1}(x,z)+ W_{2}(x,z)+ W_{3}(x,z)+ W_{4}(x,z).\no
\end{align}
It is clear that
\begin{align}\label{eq33}
  U(x,z) = U_1(x,z)+U_2(x,z)+U_3(x,z)\quad
  \textrm{for}~x\in\Omega_+\cap\mathbb{R}^2_+~\textrm{and}~z\in\R^2.
\end{align}
Thus, using the relations $u^{tot}(x,d)=u^{i}(x,d)+u^{r}(x,d)+u^{s}(x,d)$, $|u^{i}(x,d)|=1$ and $|u^r(x,d)|=|\RR_0(\theta_d)|$ for any $x\in \partial B^+_R$ and $d\in\mathbb{S}^1_-$,
we can rewrite $I_{P}(z,R)$ as
\begin{align}\label{eq39}
I_{P}(z,R) = \int_{{\partial{B^{+}_R}}}\left|U(x,z)+W(x,z)\right|^2 ds(x).
\end{align}

Define the function space
\ben
C(\ov{\mathbb{S}^1_+}):=\{\varphi\in C(\mathbb{S}^1_+):\varphi~\textrm{is}~\textrm{uniformly}
~\textrm{continuous}~\textrm{on}~\mathbb{S}^1_+\}
\enn
with the norm
$\|\varphi\|_{C(\ov{\mathbb{S}^1_+})}:=\sup_{x\in \mathbb{S}^1_+}|\varphi(x)|$
and the function space
\ben
C^1(\ov{\Sp^1_+}):=\{\varphi\in C^1(\Sp^1_+):\varphi~\textrm{and}~
\Grad \varphi~\textrm{are}~\textrm{uniformly}
~\textrm{continuous}~\textrm{on}~\Sp^1_+\}
\enn
with the norm
$\|\varphi\|_{C^1(\ov{\mathbb{S}^1_+})}:=\sup_{x\in\mathbb{S}^1_+}|\varphi(x)|
+\sup_{x\in\mathbb{S}^1_+}|\Grad \varphi(x)|$,
where ${\rm Grad}$ denotes the surface gradient on $\Sp^1_+$.
Then we need the following
uniform far-field asymptotic properties of the scattered field $u^s(x,d)$ for $x\in\Omega_+$
which were obtained in \cite{LYZZ3}.

\begin{lemma}[see Theorems 13 and 14 in \cite{LYZZ3}]\label{Le:my}
 Let $x=|x|\hat x=|x|(\cos \theta_{\hat x}, \sin \theta_{\hat x} )\in \Omega_{+}$ with $\theta_{\hat{x}}\in(0,\pi)$
and $\vert x\vert > R$, where $R>0$ is large enough
such that $\Gamma_p \subset B_R$.
For $d \in \Sp^1_{-}$,
let $u^s(x,d)$ be the scattered field of the scattering problem
(\ref{e:1.1})--(\ref{eq2}).
Then the following statements hold true.
\begin{enumerate}[(a)]
\item\label{sta1}
For the case  $k_+<k_-$, the scattered field $u^s(x,d)$ has the asymptotic behavior
\ben
u^{s}(x,d)=\frac{e^{ik_{+}|x|}}{\sqrt{|x|}}u^{\infty}(\hat x, d)+ u^s_{Res}(x,d)\quad\textrm{for}~x\in\Omega_+\ba\ov{B_R}
\enn
with the far-field patten $u^{\infty}(\hat x, d)$ of the scattered field given by (\ref{eq29}), where
$u^{\infty}(\hat x, d)$ satisfies $u^{\infty}(\cdot, d)\in C^1(\ov{\Sp^1_{+}})$ with
\begin{align*}
\| u^{\infty}(\cdot, d)\|_{C^1(\ov{\Sp^1_{+}})}\le C\quad\textrm{for~all}~d \in \Sp^1_{-},
\end{align*}
and
$u^s_{Res}(x,d)$ satisfies
\begin{align*}
&|u^s_{Res}(x,d)|\le {C}{|x|^{-3/2}},\quad |x|\rightarrow+\infty,\notag
\end{align*}
uniformly for all $\theta_{\hat{x}}\in(0,\pi)$ and $d \in \Sp^1_{-}$.

\item\label{sta2}
For the case $k_+>k_-$, the scattered field $u^s(x,d)$ has the asymptotic behavior
\be \label{eq27}
u^{s}(x,d)=\frac{e^{ik_{+}|x|}}{\sqrt{|x|}}u^{\infty}(\hat x, d)+ u^s_{Res}(x,d)\quad\textrm{for}~x\in\Omega_+\ba\ov{B_R}
\en
with the far-field pattern $u^\infty(\hat{x},d)$ of the scattered field given by (\ref{eq29}),
where $u^\infty(\hat{x},d)$ satisfies
$u^{\infty}(\cdot, d)\in C(\ov{\Sp^1_{+}})$ and
${\rm Grad}_{\hat x}\,u^{\infty}(\cdot, d)\in L^1(\Sp^1_{+})$ with
\be\label{eq:t6}
&\|u^{\infty}(\cdot, d)\|_{C(\ov{\Sp^1_{+}})},~
\left\|{\rm Grad}_{\hat x}\,u^{\infty}(\cdot, d)\right\|_{L^1(\Sp^1_{+})} \le C
\quad\textrm{for~all}~d\in\Sp^1_{-},
\en
and $u^s_{Res}(x,d)$ satisfies
\begin{align*}
&\left|u^s_{Res}(x,d)\right|\le{C}{|x|^{-3/4}},\quad |x|\rightarrow+\infty,
\end{align*}
uniformly for all $\theta_{\hat{x}}\in(0,\pi)$ and $d \in \Sp^1_{-}$,
\begin{align*}
&\left|u^s_{Res}(x,d)\right|\le{C}{{{\left|
\theta_c-\theta_{\hat x}
\right|}^{-\frac 32}}|x|^{-\frac 32}},\quad|x|\rightarrow+\infty,
\end{align*}
uniformly for all $\theta_{\hat{x}}\in (0,\theta_c)\cup(\theta_c,\pi/2)$ and $d \in \Sp^1_{-}$, and
\begin{align*}
&\left|u^s_{Res}(x,d)\right|\le{C}{{{\left|
\pi-\theta_c-\theta_{\hat x}
\right|}^{-\frac 32}}|x|^{-\frac 32}},\quad|x|\rightarrow+\infty,
\end{align*}
uniformly for all $\theta_{\hat{x}}\in \left.\left[\pi/2,\pi-\theta_c\right.\right)\cup(\pi-\theta_c,\pi)$
and $d \in \Sp^1_{-}$.
\end{enumerate}
Here, $C>0$ is a constant independent of $x$ and $d$.
\end{lemma}

As a consequence of Lemma \ref{Le:my}, we have the following lemma on the residual term
$u^s_{Res}(x,d)$ in \eqref{eq27} for the case $k_+>k_-$.

\begin{lemma}\label{le:2}
Assume $k_+>k_-$. For $d \in \Sp^1_{-}$,
let $u_{Res}^s(x,d)$ be the residual term given in (\ref{eq27}).
Then we have
\begin{align*}
\int_{\partial B^+_R} |u_{Res}^s(x,d)|ds(x) \le C R^{-{1}/4}, \quad
\int_{\partial B^+_R} |u_{Res}^s(x,d)|^2ds(x) \le C R^{-1}
\end{align*}
as $R\rightarrow+\infty$
uniformly for all $d \in \Sp^1_{-}$. Here, $C>0$ is a constant independent of $R$.
\end{lemma}

\begin{proof}
Let $R$ be large enough. Choose $\varepsilon = R^{-1/2}$ and define the set $\mathbb S^1_{\theta_c,\varepsilon}:=\{(\cos\theta,\sin\theta): \theta\in \II_\varepsilon\}$ with $\II_\varepsilon:=\{\theta\in (0,\pi): |\theta-\theta_c|\ge \varepsilon, |\theta-(\pi-\theta_c)|\ge \varepsilon\}$.
Then it follows from the statement (\ref{sta2}) of Lemma \ref{Le:my} that
\begin{align*}
\int_{\partial B^+_R}&|u_{Res}^s(x,d)|ds(x) = R\bigg\{\int_{\Sp^1_{+}\setminus \Sp^1_{\theta_c,\varepsilon}}+  \int_{\Sp^1_{\theta_c,\varepsilon}} \bigg\}\left|{u}^s_{Res}(R\hat x,d)\right|ds(\hat x)\\
&\le C\varepsilon R^{\frac14} + \frac{C}{R^{\frac 12}}\left(\int_{(0,\pi/2)\cap\II_{\varepsilon}}\frac{1}{|\theta_c-\theta_{\hat x}|^{\frac32}}d\theta_{\hat x}  + \int_{\left[\left.\pi/2,\pi\right)\right.\cap \II_{\varepsilon}}\frac{1}{|\pi-\theta_c-\theta_{\hat x}|^{\frac32}}d\theta_{\hat x}\right)\\
&\le C\varepsilon R^{\frac14} + \frac{C}{(\varepsilon R)^{\frac 12}}+\frac{C}{R^{\frac 12}} \le C R^{-\frac{1}4}
\end{align*}
and
\begin{align*}
\int_{\partial B^+_R} &|u_{Res}^s(x,d)|^2ds(x) = R\bigg\{\int_{\Sp^1_{+}\setminus \Sp^1_{\theta_c,\varepsilon}}+  \int_{\Sp^1_{\theta_c,\varepsilon}} \bigg\}\left|{u}^s_{Res}(R\hat x,d)\right|^2ds(\hat x)\\
&\le\frac{C\varepsilon}{R^{\frac12}}+ \frac{C}{R^{2}}\left(\int_{(0,\pi/2)\cap\mathbb I_{\varepsilon}}\frac{1}{|\theta_c-\theta_{\hat x}|^{3}}d\theta_{\hat x}  + \int_{\left[\left.\pi/2,\pi\right)\right.\cap\mathbb I_{\varepsilon}}\frac{1}{|\pi-\theta_c-\theta_{\hat x}|^{3}}d\theta_{\hat x}\right)\\
&\le \frac{C\varepsilon}{R^{\frac12}} + \frac{C}{(\varepsilon R)^{2}}+\frac{C}{R^2}\le C R^{-1}.
\end{align*}
The proof is thus complete.
\end{proof}

We also need the following reciprocity relation of the far-field pattern.

\begin{lemma}\label{lem1}
For $\hat x\in \Sp^1_{+}$ and $d\in \Sp^1_{-}$,
let $u^{\infty}(\hat x, d)$ be the far-field pattern of the scattering problem (\ref{e:1.1})--(\ref{eq2}).
Then we have the reciprocity relation $u^{\infty}(\hat x, d)= u^{\infty}(-d,-\hat x)$
for all $\hat x\in \Sp^1_{+}$ and $d\in \Sp^1_{-}$.
\end{lemma}

\begin{proof}
For the scattering problem (\ref{e:1.1})--(\ref{eq2}) in the limiting case $k_+=k_-$,
it is well-known that the reciprocity relation of the far-field pattern holds (see, e.g., \cite[Theorem 3.23]{DK13}).
For the considered scattering problem, it is easily seen that $G^{\infty}(\hat x,y) = {e^{i{\pi}/4}}(8\pi k_{+})^{-1/2}u^{0}(y,-\hat x)$ for $\hat{x}\in\mathbb S^1_+$
and $y\in\mathbb{R}^2_+\cup\mathbb{R}^2_-$.
Therefore, by using similar arguments as in the proof of \cite[Theorem 3.23]{DK13},
we can apply formulas (\ref{eq:0.1}), (\ref{eq:0.2}), (\ref{e:1.1}), (\ref{eq5}), (\ref{eq2}) and (\ref{eq29}) to obtain that
the assertion of this lemma holds.
\end{proof}

Further, we will apply the theory of oscillatory integrals to obtain
some inequalities.
We need the following result proved in \cite{Chen1}.

\begin{lemma}[Lemma 3.9 in \cite{Chen1}]\label{le1}
For any $-\infty<a<b<\infty$ let $u\in C^2[a,b]$ be real-valued and satisfy that
$|u'(t)|\geq1$ for all $t\in(a,b)$. Assume that $a=x_0<x_1<\cdots<x_N=b$ is a division of $(a,b)$
such that $u'$ is monotone in each interval $(x_{i-1},x_i)$, $i=1,\ldots,N$.
Then for any function $\phi$ defined on $(a,b)$ with integrable derivative and for any $\la>0$,
\ben
\left|\int_a^b e^{i\la u(t)}\phi(t)dt\right|\leq(2N+2)\la^{-1}\left[|\phi(b)|+\int^b_a|\phi'(t)|dt\right].
\enn
\end{lemma}

\begin{remark}\label{re1}
By the theory on function approximation (see, e.g, Section 5.3 and Appendix C.5 in \cite{E10}),
it can be seen that Lemma \ref{le1} still holds under the assumption that the function
$\phi\in C[a,b]\cap W^{1,1}(a,b)$.
\end{remark}

Define
$C(\ov{\mathbb{S}^1_-}):=\{\varphi\in C(\mathbb{S}^1_-):\varphi~\textrm{is}~\textrm{uniformly}
~\textrm{continuous}~\textrm{on}~\mathbb{S}^1_-\}$ with the norm
$\|\varphi\|_{C(\ov{\mathbb{S}^1_-})}:=\sup_{x\in \mathbb{S}^1_-}|\varphi(x)|$
and
$W^{1,1}(\mathbb{S}^1_-) := \{ f\in L^1(\mathbb{S}^1_-):
{\rm Grad}\,f\in L^1(\mathbb{S}^1_-)\}$ with the norm $\|f\|_{W^{1,1}(\mathbb{S}^1_-)}:=\|f\|_{L^1(\mathbb{S}^1_-)}+\|{\rm Grad}\,f\|_{L^1(\mathbb{S}^1_-)}$, where ${\rm Grad}$
denotes the surface gradient on $\mathbb{S}^1_-$.
By using Lemma \ref{le1} and Remark \ref{re1}, we have the following lemma.

\begin{lemma}\label{le:n}
Let $x\in \R^2_+$ and $d\in \mathbb S^1_-$. For $\hat x= x/|x| \in \mathbb S^1_+$,
assume that $f(\hat{x},\cdot), g(\hat{x},\cdot)\in C\big(\ov{\mathbb S^1_-}\big)\cap
W^{1,1}(\mathbb{S}^1_-)$  and define
\begin{align*}
F(x):= \int_{\mathbb S^1_-}e^{ik_+x\cdot d}f(\hat x, d)ds(d),\quad G(x):= \int_{\mathbb S^1_-}e^{ik_+x'\cdot d}g(\hat x, d)ds(d).
\end{align*}
Then we have
\begin{align*}
&|F(x)|\le C\left(\|f(\hat{x},\cdot)\|_{C\big(\ov{\mathbb S^1_-}\big)}+
\int_{\Sp^1_-}|{\rm Grad}_d\,f(\hat{x},d)|ds(d)
\right)|x|^{-1/2},\\
&|G(x)|\le C\left(\|g(\hat{x},\cdot)\|_{C\big(\ov{\mathbb S^1_-}\big)}+
\int_{\Sp^1_-}|{\rm Grad}_d\,g(\hat{x},d)|ds(d)
\right)|x|^{-1/2}
\end{align*}
as $|x|\rightarrow+\infty$ uniformly for all $\theta_{\hat{x}}\in(0,\pi)$,
where $C>0$ is a constant independent of $x$.
\end{lemma}

\begin{proof}
With the aid of Lemma \ref{le1} and Remark \ref{re1},
the statement of this lemma can be derived similarly as in the proof of Lemma 3.2 in \cite{xzz19} with minor modifications. See also Lemma 2.4 in \cite{LYZZ} for similar derivations.
\end{proof}

Next, with the aid of the above lemmas, we study the asymptotic properties of $U_j$ $(j=1,2,3)$ and $W_j$ $(j=1,2,3,4)$, which are presented in Lemmas \ref{NL1}, \ref{NL2} and \ref{NL3} below.

\begin{lemma}\label{NL1}
Let $x\in\mathbb{R}^2_+$ with $|x|$ large enough and $z\in\mathbb{R}^2$. Then we have
\begin{align}
 & \qquad \qquad |U_1(x,z)|\le {C}{|x|^{-1/2}}, \label{eq:t1}\\
 &\qquad \qquad |U_2(x,z)| \le {C(1+|z|)}{|x|^{- 1/2}}, \label{eq:t3}\\
 &\qquad \qquad |W_j(x,z)|\le {C}{|x|^{-1/2}}, \quad j=1,2,3, \label{eq:t2}\\
 &\qquad \qquad |W_{4}(x,z)|\le {C}{|x|^{-1}}  \label{eq:t4}
\end{align}
as $|x|\rightarrow+\infty$ uniformly for all $\theta_{\hat{x}}\in(0,\pi)$ and $z\in\mathbb{R}^2$,
where $C>0$ is a constant independent of $x$ and $z$.
\end{lemma}

\begin{proof}
The formulas (\ref{eq:t1}), (\ref{eq:t2}) and (\ref{eq:t4}) can be obtained by using Lemma \ref{Le:my} and the formula (\ref{eq28}).
It follows from Lemma \ref{le:n} and the formula (\ref{eq28}) that
the formula (\ref{eq:t3}) holds.
\end{proof}

\begin{remark}\label{re2}
It was proved in Lemma 3.4 in \cite{xzz19} that $U_3(x,z)$ also satisfies the estimate (\ref{eq:t3}) as $|x|\rightarrow+\infty$ uniformly for all $\theta_{\hat{x}}\in(0,\pi)$ and $z\in\mathbb{R}^2$.
\end{remark}

\begin{lemma} \label{NL2}
Let $R>0$ be large enough and $z\in\mathbb{R}^2$.
Then we have
\begin{align*}
&\left|\int_{\partial{B^{+}_R}}U(x,z)\overline {{W}_{4}(x,z)} ds(x)\right| \le {C(1+|z|)}{R^{- 1/2}},\;\\
&\int_{\partial{B^{+}_R}}\left|{{W}_{4}(x,z)}\right|^2 ds(x) \le {C}{R^{-1}}
\end{align*}
as $R\rightarrow+\infty$ uniformly for all $z\in \R^2$, where $C>0$ is a constant independent of $R$ and $z$.
\end{lemma}

\begin{proof}
This lemma is a direct consequence of Lemma \ref{NL1}, Remark \ref{re2} and the formula \eqref{eq33}.
\end{proof}

\begin{lemma} \label{NL3}
Let $R>0$ be large enough and $z\in\mathbb{R}^2$. Then we have
\begin{align}
\label{eq:3.1}&\left|\int_{\partial{B^{+}_R}}U(x,z)\overline {{W}_{1}(x,z)} ds(x)\right| \le
{C(1+|z|)^2}{R^{- 1/2}},\\
\label{eq:3.2}& \int_{\partial{B^{+}_R}}\left|{{W}_{1}(x,z)}\right|^2 ds(x) \le
{C(1+|z|)^2}{R^{-1}}
\end{align}
and
\begin{align}\label{eq:3.5}
\sum^4_{l=1}\left|\int_{{\partial{B^{+}_R}}}W_l(x,z) \overline{W_j(x,z)}ds(x)\right| +\left|\int_{{\partial{B^{+}_R}}}U(x,z)\overline{W_j(x,z)}ds(x)\right| \le {C(1+|z|)^2}{R^{-1/3}},&\no\\
j=2,3,&
\end{align}
as $R\rightarrow+\infty$ uniformly for all $z\in\mathbb{R}^2$.
Here, $C>0$ is a constant independent of $R$ and $z$.
\end{lemma}

\begin{proof}
Let $R$ be large enough throughout the proof.
We distinguish between the following two cases.

{\bf Case 1:} $k_{+}<k_{-}$.
Due to statement (\ref{sta1}) of Lemma \ref{Le:my} and Lemma \ref{lem1},
we can apply similar arguments as in the derivation of (3.18) in \cite{xzz19} to obtain that
\ben
|W_1(x,z)|\leq C(1+|z|)|x|^{-1}
\enn
for all $x\in\mathbb{R}^2_+$ with $|x|$ large enough.
Note that the formula \eqref{eq33} holds.
Thus it follows from Lemma \ref{NL1} and Remark \ref{re2} that (\ref{eq:3.1}) and (\ref{eq:3.2}) hold. Moreover,
using statement (\ref{sta1}) of Lemma \ref{Le:my} and Lemma \ref{lem1} again, we can deduce
(\ref{eq:3.5}) in the same manner as in the proof of Lemma 3.7 in \cite{xzz19}.

{\bf Case 2:} $k_{+}>k_{-}$. Let $x=|x|\hat x=|x|(\cos\theta_{\hat x}, \sin\theta_{\hat x})$ with $\theta_{\hat x} \in (0,\pi)$ and large enough $|x|$ and let $d= (\cos\theta_d, \sin\theta_d)$ with $\theta_d\in (\pi,2\pi)$.

First, we prove that (\ref{eq:3.1}) and (\ref{eq:3.2}) hold. In terms of (\ref{eq27}), we have
\begin{align}
&W_{1}(x,z)\notag\\
&=\frac{e^{-ik_+|x|}}{|x|^{1/2}}\int_{\mathbb S^1_-}e^{2ik_+x\cdot d}\overline{u^{\infty}(\hat x, d)}e^{-ik_+z\cdot d} ds(d) + \int_{\mathbb S^1_-} e^{2ik_+x\cdot d}\overline{u^s_{Res}(x,d)}e^{-ik_+z\cdot d} ds(d)\notag\\
&=: W_{1,1}(x,z) + W_{1,2}(x,z).\label{eq:3.0}
\end{align}
By applying
(\ref{eq:t6}) and Lemmas \ref{lem1} and \ref{le:n}, we have
\begin{align}\label{eq8}
|W_{1,1}(x,z)| \le C(1+|z|)|x|^{-1}.
\end{align}
Then it follows from Lemma \ref{NL1}, Remark \ref{re2} and the formula \eqref{eq33} that
\begin{align}
&\left|\int_{\partial{B^{+}_R}}U(x,z)\overline {{W}_{1,1}(x,z)} ds(x)\right| \le {C(1+|z|)^2}{R^{- 1/2}},\label{eq:3.3}\\
&\int_{\partial{B^{+}_R}}\left|{{W}_{1,1}(x,z)}\right|^2 ds(x) \le {C(1+|z|)^2}{R^{-1}}.\label{eq:3.6}
\end{align}
Moreover, by using Lemmas \ref{le:2} and \ref{NL1}, Remark \ref{re2}, \eqref{eq33} and (\ref{eq8}), we arrive at
\begin{align*}
&
\begin{aligned}
\left|\int_{\partial{B^{+}_R}}U(x,z)\overline {{W}_{1,2}(x,z)} ds(x)\right|
&\le C(1+|z|)R^{-1/2}\int_{\mathbb{S}^1_-}\int_{\pa B^+_R}|u^s_{Res}(x,d)|ds(x)ds(d)\\
&\le C(1+|z|) R^{-{3}/4},
\end{aligned}\\
&
\begin{aligned}
\left|\int_{\partial{B^{+}_R}}W_{1,1}(x,z)\overline{{W}_{1,2}(x,z)} ds(x)\right|
&\le C(1+|z|)R^{-1}\int_{\mathbb{S}^1_-}\int_{\pa B^+_R}|u^s_{Res}(x,d)|ds(x)ds(d)\\
&\le {C(1+|z|)}R^{-5/4},
\end{aligned}\\
&\int_{\partial{B^{+}_R}}\left|{{W}_{1,2}(x,z)}\right|^2 ds(x)
\le \pi \int_{\mathbb{S}^1_-}\int_{\partial B^+_R} |u_{Res}^s(x,d)|^2ds(x)ds(d)\le C R^{-1}.
\end{align*}
These, together with (\ref{eq:3.0}), (\ref{eq:3.3}) and (\ref{eq:3.6}), imply that (\ref{eq:3.1}) and (\ref{eq:3.2}) hold.

Secondly, we prove that (\ref{eq:3.5}) holds. For $\theta_{\hat x} \in (0,\pi)$ and $\theta_d\in (\pi,2\pi)$, define
\begin{align*}
&f_z(\theta_{\hat x}, \theta_d):=  \overline{u^{\infty}(\hat x,d)}\RR_0(\theta_d)e^{-ik_{+}z\cdot d},
\quad
g_z(\theta_{\hat x}, \theta_d):= {u}^{\infty}(\hat x,d)\overline{\RR_0(\theta_d)}e^{-ik_{+}z\cdot d}.
\end{align*}
Then in view of (\ref{eq27}), we have
\begin{align*}
W_{2}(x,z)= &\frac{e^{-ik_{+}|x|}}{\sqrt {|x|}}{w}_{z,2}(|x|,\theta_{\hat x})+ W_{2,Res}(x,z), \\
W_{3}(x,z)= &\frac{e^{ik_{+}|x|}}{\sqrt {|x|}}{w}_{z,3}(|x|,\theta_{\hat x})+W_{3,Res}(x,z),
\end{align*}
where
\begin{align*}
&
{w}_{z,2}(|x|,\theta_{\hat x}):=
\int_{\pi}^{2\pi} e^{2ik_{+}|x|\cos\theta_{\hat{x}}\cos\theta_d} f_z(\theta_{\hat x}, \theta_d)d\theta_d,\\
&{w}_{z,3}(|x|,\theta_{\hat x}):=
\int_{\pi}^{2\pi} e^{2ik_{+}|x|\sin\theta_{\hat{x}}\sin{\theta_d}} g_z(\theta_{\hat x},\theta_d)d\theta_d,
\end{align*}
and
\begin{align*}
&
W_{2,Res}(x,z):=\int_{\Sp_{-}^1}e^{2ik_{+}x_1d_1}
\overline{u_{Res}^{s}( x,d)}\RR_0(\theta_d)e^{-ik_{+}z\cdot d}
ds(d),\\
&
W_{3,Res}(x,z):=\int_{\Sp_{-}^1}e^{2ik_{+}x_2d_2}
{u}_{Res}^{s}( x,d)\overline{\RR_0(\theta_d)}e^{-ik_{+}z\cdot d}
ds(d).
\end{align*}
These, together with (\ref{eq28}), \eqref{eq33}, Remark \ref{re2}
and Lemmas \ref{le:2} and \ref{NL1}, imply that
\begin{align}
&\sum^4_{l=1}\left|\int_{{\partial{B^{+}_R}}}W_l(x,z) \overline{W_j(x,z)}ds(x)\right| +\left|\int_{{\partial{B^{+}_R}}}U(x,z)\overline{W_j(x,z)}ds(x) \right|\notag\\
&\le C(1+|z|)\left({\int^{\pi}_0|{w}_{z,j}(R,\theta_{\hat x})|d\theta_{\hat x} + R^{-1/2}\int_{\partial B_R^+}\left|W_{j,Res}(x, z)\right|ds(x)}\right)\no\\
&
\le C(1+|z|)\int^{\pi}_0|{w}_{z,j}(R,\theta_{\hat x})|d\theta_{\hat x} +
C(1+|z|)R^{-3/4},
\quad j=2,3. \label{eq:3.7}
\end{align}
Moreover, by using the formula (\ref{eq28}) and Lemmas \ref{Le:my} and \ref{lem1}, we have that
for any $\theta_{\hat{x}}\in(0,\pi)$, $f_z(\theta_{\hat{x}},\cdot)$
and $g_z(\theta_{\hat{x}},\cdot)$ can be continuously extended
from $(\pi,2\pi)$ to $[\pi,2\pi]$ with
\be\label{eq30}
\|f_z(\theta_{\hat{x}},\cdot)\|_{C[\pi,2\pi]},~
\|g_z(\theta_{\hat{x}},\cdot)\|_{C[\pi,2\pi]}
\leq
C\|u^{\infty}({\hat x}, \cdot)\|_{C(\ov{\Sp^1_{-}})}\leq C,
\en
and
$d f_z(\theta_{\hat{x}},\cdot)/d\theta_{d},d g_z(\theta_{\hat{x}},\cdot)/d\theta_{d}\in L^1(\pi,2\pi)$
with
\begin{align}
&\left\|\frac{d}{d\theta_d}f_z(\theta_{\hat{x}},\cdot)\right\|_{L^1(\pi,2\pi)},~
\left\|\frac{d}{d\theta_d}g_z(\theta_{\hat{x}},\cdot)\right\|_{L^1(\pi,2\pi)}\nonumber\\
&\label{eq31}
\leq
C(1+|z|)
\left(
\|u^{\infty}({\hat x}, \cdot)\|_{C(\ov{\Sp^1_{-}})}+
\left\|{\rm Grad}_{d}\,u^{\infty}({\hat x}, \cdot)\right\|_{L^1(\Sp^1_{-})}\right)
\leq C(1+|z|),
\end{align}
where the constants are independent of $\theta_{\hat{x}}$ and $z$.

Choose $\varepsilon=R^{-1/3}$ and define
\begin{align*}
&\II_{\varepsilon,1}:=[\pi/2-\varepsilon,\pi/2+\varepsilon],
&&\II_{\varepsilon,2}:=(0,\pi)\ba\II_{\varepsilon,1},\\
&\II_{\varepsilon,3}:=\left(\left.\pi,\pi+\varepsilon\right]\right.\cup\left[\left.2\pi-\varepsilon, 2\pi\right)\right.,
&&\II_{\varepsilon,4}:=(\pi,2\pi)\ba\II_{\varepsilon,3},\\
&\II_{\varepsilon,5}:=\left(\left.0,\varepsilon\right]\right.\cup\left[\left.\pi-\varepsilon, \pi\right)\right.,
&&\II_{\varepsilon,6}:=(0,\pi)\ba\II_{\varepsilon,5},\\
&\II_{\varepsilon,7}:=[3\pi/2-\varepsilon,3\pi/2+\varepsilon],
&&\II_{\varepsilon,8}:=(\pi,2\pi)\ba\II_{\varepsilon,7}.
\end{align*}
Then arguing similarly as in the derivations of the estimates (3.31) and (3.34) in \cite{xzz19},
we can apply Lemma \ref{le1}, Remark  \ref{re1} and the formulas (\ref{eq30})
and (\ref{eq31}) to deduce that
\begin{align}
&\int^{\pi}_0|{w}_{z,2}(R,\theta_{\hat x})|d{\theta_{\hat x}}\no\\
\le&\int_{\II_{\varepsilon,1}}|{w}_{z,2}(R,\theta_{\hat x})|d{\theta_{\hat x}}+
\int_{\II_{\varepsilon,2}}\left|\int_{\II_{\varepsilon,3}}e^{2ik_+R\cos\theta_{\hat x}\cos\theta_d}f_{z}
(\theta_{\hat x},\theta_d)d\theta_d\right|d\theta_{\hat x}\nonumber\\
& +\int_{\II_{\varepsilon,2}}\left|\int_{\II_{\varepsilon,4}}e^{2ik_+R\cos\theta_{\hat x} \cos\theta_d}f_{z}(\theta_{\hat x},\theta_d)d\theta_d\right|d\theta_{\hat x}\notag\\
\le& C\varepsilon +\frac{C}{R\varepsilon^2}
\max_{\theta_{\hat{x}}\in\II_{\varepsilon,2}}\left(\|f_z(\theta_{\hat{x}},\cdot)\|_{C[\pi,2\pi]}+
\left\|\frac{d}{d\theta_d}f_z(\theta_{\hat{x}},\cdot)\right\|_{L^1(\pi,2\pi)}\right)
\le \frac{C(1+|z|)}{R^{1/3}}\label{eq10}
\end{align}
and
\begin{align}
&\int^{\pi}_0|{w}_{z,3}(R,\theta_{\hat x})|d{\theta_{\hat x}}\no\\
\le&\int_{\II_{\varepsilon,5}}|{w}_{z,3}(R,\theta_{\hat x})|d{\theta_{\hat x}}+
\int_{\II_{\varepsilon,6}}\left|\int_{\II_{\varepsilon,7}}e^{2ik_+R\sin\theta_{\hat x}\sin\theta_d}g_{z}
(\theta_{\hat x},\theta_d)d\theta_d\right|d\theta_{\hat x} \notag\\
&+\int_{\II_{\varepsilon,6}}\left|\int_{\II_{\varepsilon,8}}e^{2ik_+R\sin\theta_{\hat x} \sin\theta_d}g_{z}(\theta_{\hat x},\theta_d)d\theta_d\right|d\theta_{\hat x}\notag\\
\le& C\varepsilon +\frac{C}{R\varepsilon^2}
\max_{\theta_{\hat{x}}\in\II_{\varepsilon,6}}\left(\|g_z(\theta_{\hat{x}},\cdot)\|_{C[\pi,2\pi]}+
\left\|\frac{d}{d\theta_d}g_z(\theta_{\hat{x}},\cdot)\right\|_{L^1(\pi,2\pi)}\right)
\le \frac{C(1+|z|)}{R^{1/3}}.\label{Ne8}
\end{align}
Combining the formulas (\ref{eq:3.7}), (\ref{eq10}) and (\ref{Ne8}), we obtain that (\ref{eq:3.5}) holds.
The proof is thus completed.
\end{proof}

Finally, as a direct consequence of Lemmas \ref{NL2} and \ref{NL3}, we can apply the
formula (\ref{eq39}) to obtain the following theorem on the imaging function $I_P(z,R)$.

\begin{theorem}\label{th:1}
Let $R>0$ be large enough  and $z\in \R^2$. Define the function
\be \label{Ne14}
I_S(z,R):= \int_{\partial{B^{+}_R}}|U(x,z)|^2 ds(x).
\en
Then the imaging function $I_{P}(z,R)$ can be written as
\[
I_{P}(z,R)= I_S(z,R) +I_{P,Res}(z,R)
\]
with $I_{P,Res}(z,R)$ satisfying the estimate
\[
|I_{P,Res}(z,R)|\le {C(1+|z|)^2}{R^{-1/3}}
\]
as $R\rightarrow+\infty$ uniformly for all $z\in\mathbb{R}^2$.
Here, $C>0$ is a constant independent of $R$ and $z$.
\end{theorem}

\subsection{Property of the imaging function $I_F(z)$}\label{sec:4}

In this subsection, we study the asymptotic property between
the imaging function $I_F(z)$ given in \eqref{c:8}
and the function $I_S(z,R)$ given in (\ref{Ne14}) when the radius $R$ is large enough.
To achieve this, we will derive the uniform far-field expansions of
$U_i(x,z)$ ($i=1,2,3$) in what follows.

First, we have the following uniform far-field expansion of $U_1(x,z)$.

\begin{lemma}\label{LL:1}
Let $x\in \R_{+}^2$ with sufficiently large $|x|$ and $z\in \R^2$. Then $U_1(x,z)$ has the asymptotic behavior
\begin{align}\label{eq24}
U_{1}(x,z)=\frac{e^{ik_{+}|x|}}{\sqrt {|x|}}\int_{\Sp^1_{-}}u^{\infty}(\hat x, d)e^{-ik_{+}z\cdot d} ds(d)+ U_{1,Res}(x,z)
\end{align}
with the residual term $U_{1,Res}(x,z)$ satisfying
\begin{align*}
|U_{1,Res}(x,z)|\le \begin{cases} {C}{|x|^{- 3/2}} & \textrm{in the case}~k_+<k_-,\\
C|x|^{-3/4} &\textrm{in the case}~ k_+>k_-,
\end{cases}
\end{align*}
as $|x|\rightarrow+\infty$ uniformly for all $\theta_{\hat{x}}\in (0,\pi)$ and $z\in\mathbb{R}^2$.
Here, $C>0$ is a constant independent of $x$ and $z$.
\end{lemma}
\begin{proof} This lemma is a direct consequence of Lemma \ref{Le:my}.
\end{proof}

Secondly, we analyze the uniform far-field expansions of $U_2(x,z)$ and $U_3(x,z)$. To do this,
we need the following two lemmas, which will be proved in Appendix \ref{sec4}.

\begin{lemma}\label{Le:3}
Let $a,b,\lambda \in\mathbb{R}$ with $a<0<b$ and $\lambda>0$. Then the integral $I(\lambda):= \int^{b}_{a} e^{-i\lambda{\eta^2}/2}d\eta$ satisfies
\[
I(\lambda)={e^{-i{\pi}/4}\sqrt{2\pi}}\lambda^{-1/2}+ I_{Res}(\lambda)
\]
with $|I_{Res}(\lambda)|\le {2}{\lambda}^{-1}(|a|^{-1}+b^{-1})$.
\end{lemma}

\begin{lemma}\label{NLe:6}
Assume $a,b,t,\lambda\in\mathbb{R}$ with $a < 0 < b$ and $t, \lambda>0$. Define the integral
\begin{align*}
I(\lambda):=\int^{b}_a e^{-i\lambda{\eta^2}/{2}}\left(f(\eta)-f(0)\right)d\eta
\end{align*}
with $f(\eta):= q(\eta)e^{i t p(\eta)}$, where
$p(\eta)\in C^{3}[a,b]$ is a real-valued function and $q(\eta)\in C^{3}[a,b]$ is a complex-valued
function. Then we have
\ben
|I(\lambda)|\le C(1+b-a)(1+\|p\|_{C^3[a,b]})^3\|q\|_{C^3[a,b]}(1+t)^2\lambda^{-1},
\enn
where $C>0$ is a constant independent of $a,b,t,\lambda$ and the functions $p,q$.
\end{lemma}

With the aid of Lemmas \ref{Le:3} and \ref{NLe:6}, we have the following properties for
some relevant integrals.

\begin{lemma}\label{LL:2}
Let $x=|x|\hat x= |x|(\cos \theta_{\hat x},\sin\theta_{\hat x})\in\mathbb{R}^2_+$ with
$\theta_{\hat x}\in (0,\pi)$ and $z=|z|(\cos\theta_{\hat{z}},\sin\theta_{\hat{z}})\in\mathbb{R}^2$
with $\theta_{\hat{z}} \in [0,2\pi)$. Define
\begin{align*}
V(x,z) := \int^{2\pi}_{\pi}  e^{ik_{+}|x|\cos(\theta_d+{\theta}_{\hat x})}e^{-ik_{+}|z|\cos(\theta_d-\theta_{\hat{z}})}f(\theta_d)d \theta_d
\end{align*}
with $f\in C[\pi,2\pi]$.
Then the following statements hold true.

\begin{enumerate}[(a)]
  \item \label{a} If $f\in C^3[\pi,2\pi]$, then $V(x,z)$ has the form
  \begin{align}\label{eq:13}
   V(x,z) = \frac{e^{ik_{+}|x|}}{|x|^{\frac 12}}e^{-\frac{i\pi}{4}}{\left(\frac{2\pi}{k_{+}}\right)}^{ 1/ 2}f(2\pi-\theta_{\hat x})e^{-ik_{+}\hat x\cdot z{'}} + V_{Res}(x,z)
\end{align}
with the residual term $V_{Res}(x,z) $ satisfying
\begin{align*}
| V_{Res}(x,z) | \le C\|f\|_{C^3[\pi,2\pi]}E(\theta_{\hat x},z){|x|}^{-1}
\end{align*}
for all $x\in\mathbb{R}^2_+$ and $z\in\mathbb{R}^2$. Here, $E(\theta_{\hat x},z)$ is given by
\begin{align}\label{eq:29}
&E(\theta_{\hat x},z)\notag\\
&:=(1+|z|)^2+\frac{1}{|\sin\frac{\theta_{\hat x}}2|}+ \frac{1}{|\sin \frac{\pi-\theta_{\hat x}}2|}+\frac{1}{|\sin\theta_{\hat x}|}+
(1+|z|)
\left|\int_{{\theta}_{\hat x}}^{\pi-{\theta}_{\hat x}}\frac{1}{\sin^2 t}dt\right|
\end{align}
and $C>0$ is a constant independent of $x$, $z$ and $f$.
  \item \label{b} Let $a,b \in \R$ such that $\pi<a<b<2\pi$ and let $\theta_0\in(a,b)$. If $f$ has the form
  \begin{align*}
      f(\theta) = \mathcal S_j(\cos \theta-\cos\theta_0)g(\theta), \quad \theta \in [\pi,2\pi],
  \end{align*}
  with $j\in\{1,2\}$, $g\in C^3[\pi,2\pi]$ and $\textrm{Supp}(g)\subset (a,b)$, then
  $V(x,z)$ has the form (\ref{eq:13})
  with the residual term $V_{Res}(x,z) $ satisfying
\begin{align}\label{eq:20}
|V_{Res}(x,z)| \le C {(1+|z|)^2\|g\|_{C^3[a,b]}}|x|^{-3/4}  \quad \text{as} \; |x| \rightarrow +\infty
\end{align}
uniformly for all $\theta_{\hat x} \in (0,\pi)$ and $z\in\mathbb{R}^2$.
Here, $C>0$ is a constant independent of $x$, $z$, $\theta_0$ and $g$ but dependent of $a$ and $b$.
\end{enumerate}
\end{lemma}

\begin{proof}
Let  $\widetilde{\theta}_{\hat x}:= \pi-{\theta}_{\hat x}$ and
$\mathcal{P}(t):= k_+\cos(t+\widetilde{\theta}_{\hat x}-\theta_{\hat{z}})$.
A straightforward calculation gives that
\begin{align}
V(x,z) = \int^{\pi-\widetilde{\theta}_{\hat x}}_{-\widetilde{\theta}_{\hat x}} f(t+\pi+\widetilde{\theta}_{\hat x})e^{ik_{+}|x|\cos t}e^{i|z|\mathcal{P}(t)}dt. \label{eq:28}
\end{align}

First, we prove the statement (\ref{a}). To do this, we consider the following Parts 1.1 and 1.2.

\textbf{Part 1.1}: Estimate of $V(x,z)$ with $\theta_{\hat x} \in \left[\left.\pi/2,\pi\right)\right.$.
We rewrite $V$ as follows:
\begin{align*}
V(x,z) &= \left\{\int^{\widetilde{\theta}_{\hat x}}_{-\widetilde{\theta}_{\hat x}}+\int^{\pi-\widetilde{\theta}_{\hat x}}_{\widetilde{\theta}_{\hat x}}\right\}f(t+\pi+\widetilde{\theta}_{\hat x})e^{ik_{+}|x|\cos t}e^{i|z|\mathcal{P}(t)}dt =:V_1(x,z)+ V_2(x,z).
\end{align*}
For the function $V_1$, the change of variable $\eta = 2\sin (t/2)$ gives that
\begin{align*}
V_1(x,z)&= \int^{\alpha_2}_{\alpha_1} e^{-ik_{+}|x|({\eta^2}/2-1)}F(0,z)d\eta
+
\int^{\alpha_2}_{\alpha_1} e^{-ik_{+}|x|({\eta^2}/2-1)}
\left(F(\eta,z)-F(0,z)\right)
d\eta
\\
&=:V_{1,1}(x,z)+ V_{1,2}(x,z),
\end{align*}
where
\begin{align}
\label{eq18}
&\alpha_1 := -2\sin{({\widetilde{\theta}_{\hat x}}/{2})},\quad \alpha_2 := 2\sin{({\widetilde{\theta}_{\hat x}}/{2})},\\
\label{eq19}
&F(\eta, z):=f(w(\eta)
+\pi
+\widetilde{\theta}_{\hat x})e^{i|z|\mathcal{P}(w(\eta))}w'(\eta)\quad \textrm{with}~ w(\eta) := 2\arcsin(\eta/2).
\end{align}
Note that $\|w\|_{C^{4}\left[\alpha_1, \alpha_2\right]}\le C$ due to $\theta_{\hat x} \in \left[\left.\pi/2,\pi\right)\right.$.
Thus it follows from
Lemmas \ref{Le:3} and \ref{NLe:6} that $V_{1,1}(x,z)$ has the form
\begin{align*}
V_{1,1}(x,z)=  \frac{e^{ik_{+}|x|}}{|x|^{\frac 12}}e^{-\frac{i\pi}{4}}{\left(\frac{2\pi}{k_{+}}\right)}^{ 1/ 2}f(2\pi-\theta_{\hat x})e^{-ik_{+}\hat x\cdot z{'}} + V_{1,1,Res}(x,z)
\end{align*}
with the residual term $V_{1,1,Res}(x,z)$ satisfying
\ben
|V_{1,1,Res}(x,z)| \le C|\sin (\widetilde{\theta}_{\hat x}/2)|^{-1}{|x|}^{-1}\|f\|_{C[\pi,2\pi]}
\enn
and that $V_{1,2}(x,z)$ satisfies the estimate
\ben
|V_{1,2}(x,z)|\le C{(1+|z|)^2\|f\|_{C^3[\pi,2\pi]}}{|x|^{-1}}.
\enn
Moreover,
for the function $V_2$, we can apply an integration by parts to obtain that
\begin{align*}
|V_2(x,z)|&= \left|\frac{i}{k_+|x|}\left[ f_1(t)e^{ik_{+}|x|\cos t}\bigg|^{\pi-\widetilde{\theta}_{\hat x}}_{t=\widetilde{\theta}_{\hat x}} - \int^{\pi-\widetilde{\theta}_{\hat x}}_{\widetilde{\theta}_{\hat x}} f'_1(t)e^{ik_{+}|x|\cos t}dt\right]\right|\\
&\le C\left(|\sin\widetilde{\theta}_{\hat x}|^{-1}+(1+|z|)\left|\int_{\widetilde{\theta}_{\hat x}}^{\pi-\widetilde{\theta}_{\hat x}}(\sin t)^{-2}dt\right|\right) \|f\|_{C^1[\pi,2\pi]}|x|^{-1},
\end{align*}
where $f_1(t):=f(t+\widetilde{\theta}_{\hat x}+\pi)e^{i|z|\mathcal{P}(t)}(\sin t)^{-1}$.
Hence, the above arguments imply that
$V(x,z)$ has the form (\ref{eq:13}) with $V_{Res}(x,z)$ satisfying
\begin{align*}
&| V_{Res}(x,z) | \notag\\
&\le C\|f\|_{C^3[\pi,2\pi]}\left((1+|z|)^2+ \frac{1}{|\sin
\frac{\widetilde{\theta}_{\hat x}}{2}|}+\frac{1}{|\sin\widetilde{\theta}_{\hat x}|}+
(1+|z|)
\left|\int_{\widetilde{\theta}_{\hat x}}^{\pi-\widetilde{\theta}_{\hat x}}\frac{1}{\sin^2 t}dt\right|\right){|x|}^{-1}
\end{align*}
for all $x\in\mathbb{R}^2_+$ with $\theta_{\hat x} \in \left[\left.\pi/2,\pi\right)\right.$ and all $z\in\mathbb{R}^2$.

\textbf{Part 1.2}: Estimate of $V(x,z)$ with $\theta_{\hat x}\in (0,\pi/2)$. It is clear that (\ref{eq:28}) can be rewritten as
\begin{align*}
V(x,z)=\int^{\pi-{\theta}_{\hat x}}_{-{\theta}_{\hat x}} f(-t+\pi+\widetilde{\theta}_{\hat x})e^{ik_{+}|x|\cos t}e^{i|z|\mathcal{P}(-t)}dt.
\end{align*}
Then, by using similar arguments as in Part 1.1, we can deduce that $V(x,z)$ has the form (\ref{eq:13}) with $V_{Res}(x,z)$ satisfying
\begin{align*}
&| V_{Res}(x,z) |\notag\\
&\le C\|f\|_{C^3[\pi,2\pi]}\left((1+|z|)^2+ \frac{1}{|\sin \frac{\theta_{\hat x}}2|}+\frac{1}{|\sin\theta_{\hat x}|}+
(1+|z|)
\left|\int_{{\theta}_{\hat x}}^{\pi-{\theta}_{\hat x}}\frac{1}{\sin^2 t}dt\right|\right){|x|}^{-1}
\end{align*}
for all $x\in\mathbb{R}^2_+$ with $\theta_{\hat x}\in (0,\pi/2)$ and all $z\in\mathbb{R}^2$.

Therefore, it follows from the discussions in the above two parts that the statement (\ref{a}) holds.

Secondly, we prove the statement (\ref{b}). We only consider the case $j=1$ since the proof for the case $j=2$ can be obtained in the same manner.
In the rest of the proof, we assume that $|x|$
is sufficiently large.
Let $\delta_a:=a-\pi$ and $\delta_b:=2\pi-b$.
Our proof consists of the following Parts 2.1 and 2.2.

\textbf{Part 2.1}: Estimate of $V(x,z)$ with $\theta_{\hat x}\in [\delta_b/2,\pi-\delta_a/2]$. By the change of variable $\eta= 2\sin ({t}/2)$, \eqref{eq:28} can be rewritten as
\begin{align}\label{eq15}
V(x,z)&= \int^{\alpha_3}_{\alpha_1} e^{-ik_{+}|x|({\eta^2}/2-1)}F(\eta,z)d\eta,
\end{align}
where $\alpha_1$ and $F(\eta, z)$ are given as in \eqref{eq18} and \eqref{eq19},
respectively, and $\alpha_3:= 2\sin{({(\pi-\widetilde{\theta}_{\hat x})}/{2})}$.
For this part, due to $\theta_{\hat x}\in [\delta_b/2,\pi-\delta_a/2]$, it is clear that
\begin{align}\label{eq17}
-2\cos(\delta_b/4)\leq\alpha_1\leq -2\sin(\delta_a/4),~
2\sin(\delta_b/4)\leq\alpha_3\leq 2\cos(\delta_a/4),~
\|w\|_{C^{4}\left[\alpha_1, \alpha_3\right]}\le C,
\end{align}
where $w(\cdot)$ is defined as in \eqref{eq19}.
Let $\vartheta_0:=2\pi-\theta_0$.
It is easy to see that
\begin{align}\label{eq12}
&\sin\left(({t+
2\pi-\vartheta_0-\theta_{\hat{x}}})/2\right)>\min[\sin(\delta_a/2),\sin(\delta_b/2)],\\
\label{eq13}
&\cos\left(({t\pm(
\theta_{\hat{x}}-\vartheta_0
))}/4\right)> \min[\sin(\delta_a/4),\sin(\delta_b/4)]
\end{align}
for $t\in [-\widetilde{\theta}_{\hat x}, \pi-\widetilde{\theta}_{\hat x}]$.
Thus it follows that for $t\in [-\widetilde{\theta}_{\hat x}, \pi-
\widetilde{\theta}_{\hat x}]$,
\begin{align*}
\cos(t-\theta_{\hat{x}})-\cos&\vartheta_0= \frac{\sin{\left(\frac{t+2\pi
-\vartheta_0-\theta_{\hat{x}}
}{2}
\right)}\cos\left(\frac{t
+\vartheta_0-\theta_{\hat{x}}
}{4}\right)\left[4\sin{\left(\frac{t
+\vartheta_0-\theta_{\hat{x}}
}{4}\right)}\cos\left(\frac{t
+\theta_{\hat{x}}-\vartheta_0
}{4}\right)\right]}{\cos\left(\frac{t
+\theta_{\hat{x}}-\vartheta_0
}{4}\right)}\notag\\
&=\frac{\sin{\left(\frac{t
+2\pi-\vartheta_0-\theta_{\hat{x}}
}{2}\right)}\cos\left(\frac{t
+\vartheta_0-\theta_{\hat{x}}
}{4}\right)}{\cos\left(\frac{t
+\theta_{\hat{x}}-\vartheta_0
}{4}\right)}\left[2\sin{\left(\frac{t}{2}\right)}-2\sin\left(\frac{
\theta_{\hat{x}}-\vartheta_0
}{2}\right)\right],
\end{align*}
which yields $\mathcal S_1\left(\cos(t-\theta_{\hat{x}}
)-\cos\vartheta_0\right)=S_1(2\sin(t/2)-\beta)h(t)$
with $\beta:= 2\sin((\theta_{\hat{x}}-\vartheta_0)/2)$
and $h(t)$ given by
\begin{align*}
&h(t):=\notag\\
&{\mathcal{S}_1\left(\sin{\left(\frac{t
+2\pi-\vartheta_0-\theta_{\hat{x}}
}{2}\right)}\right)
\mathcal{S}_1\left(\cos\left(\frac{t
+\vartheta_0-\theta_{\hat{x}}
}{4}\right)\right)
}\bigg/{\mathcal{S}_1\left(\cos\left(\frac{t
+\theta_{\hat{x}}-\vartheta_0
}{4}\right)\right)}.
\end{align*}
Hence, we have
\begin{align}\label{eq16}
F(\eta,z)= F_1(\eta,z)\mathcal S_1(\eta-\beta)\quad \textrm{for}~ \eta\in[\alpha_1,\alpha_3],
\end{align}
where $F_1(\eta,z)$ is given by
\begin{align}\label{eq:31}
F_1(\eta,z):= e^{i|z|\mathcal{P}(w(\eta))}w'(\eta)g(w(\eta)+\widetilde{\theta}_{\hat x}+\pi)h(w(\eta)).
\end{align}
It is clear from \eqref{eq12} and \eqref{eq13} that
\begin{align}\label{eq:21}
\|h\|_{C^{3}[-\widetilde{\theta}_{\hat x},\pi-\widetilde{\theta}_{\hat x}]}\le C.
\end{align}
This, together with \eqref{eq17}, implies that
\begin{align}\label{eq:15}
\|F_1(\cdot,z)\|_{C^{j}\left[\alpha_1, \alpha_3\right]}\le C(1+|z|)^j\|g\|_{C^j[a,b]}, \quad j=1,2.
\end{align}
The rest proof of this part is divided into three cases.

\textbf{Case 1}: $\theta_{\hat x}\in [\delta_b/2,\pi-\delta_a/2]$ with $|\sin((\theta_{\hat{x}}-\vartheta_0)/2)|\le (2\sqrt{k_{+}|x|})^{-1}$ (that is, $\sqrt{k_{+}|x|}|\beta|\leq 1$). Let $\lambda := k_+|x|$
and $\sigma := \sqrt{\lambda}\beta$.
Then in terms of \eqref{eq15} and \eqref{eq16}, we can introduce the change of variable $\eta = {y}/{\sqrt{\lambda}}+\beta$ to obtain that
\begin{align}\label{eq14}
V(x,z)&= \int^{{\sqrt{\lambda}\alpha_3-\sigma}}_{\sqrt{\lambda}\alpha_1-\sigma}
{e^{-i(\frac{\sigma^2}2-\lambda)}}{\lambda^{ -3/4}}{e^{-i\frac{y^2+2y\sigma}2}}{\mathcal S_1{(y)}}F_1\left(\frac{y+\sigma}{\sqrt{\lambda}},z\right)dy.
\end{align}

In this case, we claim that
\begin{align}\label{eq:22}
|V(x,z)|\le C{(1+|z|)^2\|g\|_{C^2[a,b]}}{|x|^{-3/4}}\quad\textrm{as}~|x|\rightarrow+\infty
\end{align}
uniformly for all
$\theta_{\hat x}\in [\delta_b/2,\pi-\delta_a/2]$ with $|\sin((\theta_{\hat{x}}-\vartheta_0)/2)|\le (2\sqrt{k_{+}|x|})^{-1}$ and all $z\in\mathbb{R}^2$.
To prove this, we set $\gamma_0:=\max(2\sin^{-1}(\delta_a/4),2\sin^{-1}(\delta_b/4),1)$
and choose $|x|$ to be large enough such that $\sqrt{\lambda}>\gamma_0$.
Noting that $\sqrt{\lambda}\alpha_1-\sigma<-3$ and $\sqrt{\lambda}\alpha_3-\sigma>3$ due to \eqref{eq17},
we can rewrite (\ref{eq14}) as
\begin{align}
V(x,z)
&=\left\{\int_{\sqrt{\lambda}\alpha_1-\sigma}^{-2}+\int^2_{-2}+ \int_{2}^{\sqrt{\lambda}\alpha_3-\sigma}\right\}
\frac{e^{-i(\frac{\sigma^2}2-\lambda)}}{\lambda^{ -3/4}}e^{-i\frac{y^2+2y\sigma}{2}}\mathcal S_1{(y)}F_1\left(\frac{y+\sigma}{\sqrt{\lambda}},z\right)dy\notag\\
&=: \mathcal{V}_1(x,z)+\mathcal{V}_2(x,z)+\mathcal{V}_3(x,z). \label{eq:24}
\end{align}
Clearly, $|\mathcal{V}_2(x,z)|\leq C\|F_1(\cdot,z)\|_{C[\alpha_1,\alpha_3]}{{\lambda}^{-3/4}}$.
Moreover, an integration by parts gives that
\begin{align*}
&{e^{i(\frac{\sigma^2}2-\lambda)}}{\lambda^{ \frac34}}\mathcal{V}_1(x,z)=\\
&\frac{e^{-i\frac{y^2+2y\sigma}{2}}\mathcal{S}_1{(y)}F_1\left(\frac{y+\sigma}{\sqrt{\lambda}},z\right)}{-i(y+\sigma)}
\bigg|_{y=\sqrt{\lambda}\alpha_1-\sigma}^{-2}-\frac{e^{-i\frac{y^2+2y\sigma}{2}}
\frac{d}{dy}{\left(\frac{i\mathcal{S}_1{(y)}}{y+\sigma}F_1
\left(\frac{y+\sigma}{\sqrt{\lambda}},z\right)\right)}}{-i(y+\sigma)}{\bigg|}_{y=\sqrt{\lambda}\alpha_1-\sigma}^{-2}\\
&+\int_{\sqrt{\lambda}\alpha_1-\sigma}^{-2} e^{-i\frac{y^2+2y\sigma}{2}}
{\frac{d}{dy}\left(\frac{i}{y+\sigma}
{\frac{d}{dy}\left(\frac{i\mathcal S_1 {(y)}}{y+\sigma}F_1\left(\frac{y+\sigma}{\sqrt{\lambda}},z\right)\right)}
\right)}dy=: D_1+D_2+D_3.
\end{align*}
Since $\sqrt\lambda \alpha_1-\sigma <-3$ and $|\sigma|\le 1$, it can be seen that $|\mathcal S_1(y)/(y+\sigma)|\le 1/(\sqrt {|y|} -\sqrt 2/2) \le \sqrt 2$,
$|1/(y+\sigma)| \le 1$ and $|1/\mathcal S_1(y)|\le 1/{\sqrt 2}$  for $y \in [\sqrt \lambda \alpha_1-\sigma,-2]$. From this together with $\lambda >1$
and the fact that
\begin{align}\label{eq20}
d\mathcal{S}_1(s)/ds=(\mathcal{S}_1(s))^{-1}/2\quad \textrm{for}~ s\in\mathbb{R}\ba\{0\},
\end{align}
we can use direct but patient calculations to obtain that $|D_1|\le C\|F_1(\cdot,z)\|_{C[\alpha_1,\alpha_3]}$, $|D_2|\le C\|F_1(\cdot,z)\|_{C^1[\alpha_1,\alpha_3]}$ and
\begin{align*}
|D_3|\le{C}\|F_1(\cdot,z)\|_{C^2[\alpha_1,\alpha_3]}\int^{-2}_{-\infty}
\frac{1}{{|y+1|}(\sqrt {|y|}-\sqrt 2/2)}dy\le C \|F_1(\cdot,z)\|_{C^2[\alpha_1,\alpha_3]}.
\end{align*}
Thus, it follows that $|\mathcal{V}_1(x,z)| \le {C\|F_1(\cdot,z)\|_{C^2[\alpha_1,\alpha_3]}}{{\lambda}^{-3/4}}$. Similarly to the analysis of  $\mathcal{V}_1$, we also have  $|\mathcal{V}_3(x,z)|\le {C\|F_1(\cdot,z)\|_{C^2[\alpha_1,\alpha_3]}}{{\lambda}^{-3/4}}$. Combining the above estimates of $\mathcal{V}_1$, $\mathcal{V}_2$ and $\mathcal{V}_3$ and the formulas  (\ref{eq:15}) and (\ref{eq:24}) gives that $V(x,z)$ satisfies (\ref{eq:22}) uniformly for all
$\theta_{\hat x}\in [\delta_b/2,\pi-\delta_a/2]$ with $|\sin((\theta_{\hat{x}}-\vartheta_0)/2)|\le (2\sqrt{k_{+}|x|})^{-1}$ and all $z\in\mathbb{R}^2$.

Note that $|\cos\theta_{\hat x}-\cos \vartheta_0|= \left|2\sin\left({(
\vartheta_0+\theta_{\hat x}
)}/{2}\right)\sin\left({(\vartheta_0-\theta_{\hat x})}/{2}\right)\right|\le  {C} {|x|}^{-1/2}$
under the assumption $|\sin((\theta_{\hat{x}}-\vartheta_0)/2)|\le (2\sqrt{k_{+}|x|})^{-1}$.
Thus we have that $|f(2\pi-\theta_{\hat x})| \le {C}{|x|^{- 1/4}}\|g\|_{C[a,b]}.$
This, together with (\ref{eq:22}), implies that $V(x,z)$ has the form (\ref{eq:13}) with $V_{Res}(x,z)$
satisfying (\ref{eq:20}) uniformly for all
$\theta_{\hat x}\in[\delta_b/2,\pi-\delta_a/2]$ with $|\sin((\theta_{\hat{x}}-\vartheta_0)/2)|\le (2\sqrt{k_{+}|x|})^{-1}$ and all $z\in\mathbb{R}^2$.

\textbf{Case 2}: $\theta_{\hat x}\in [\delta_b/2,\pi-\delta_a/2]$ with
$\sin((\theta_{\hat{x}}-\vartheta_0)/2)> (2\sqrt{k_{+}|x|})^{-1}$ (that is, $\sqrt{k_{+}|x|}\beta>1$).
Note that $\alpha_1<0<\alpha_3$.
Then by using \eqref{eq16}, we divide $V$ in \eqref{eq15} into three parts:
\begin{align*}
V(x,z)=&
\int^{\alpha_3}_{\alpha_1} e^{-ik_{+}|x|({\eta^2}/2-1)}F(0,z)d\eta\\
&+
\int^{\alpha_3}_{\alpha_1}e^{-ik_{+}|x|({\eta^2}/2-1)}
{F_1(\eta, z)}\left({\mathcal S_1 {(\eta-\beta)} - \mathcal S_1 {(-\beta)}}\right)d\eta\notag\\
&+  \int^{\alpha_3}_{\alpha_1}e^{-ik_{+}|x|({\eta^2}/2-1)}
{\mathcal S_1 {(-\beta)}\left(F_1(\eta, z)-F_1(0,z)\right)}d\eta\notag\\
=:&\mathcal{J}_{1}(x,z)+ \mathcal{J}_{2}(x,z)+ \mathcal{J}_3(x,z).
\end{align*}
For $\mathcal{J}_1(x,z)$, it easily follows from \eqref{eq17} and Lemma \ref{Le:3} that
\begin{align*}
\mathcal{J}_1(x,z)= \frac{e^{ik_{+}|x|}}{|x|^{\frac 12}}e^{-\frac{i\pi}{4}}{\left(\frac{2\pi}{k_{+}}\right)}^{ 1/ 2}f(2\pi-\theta_{\hat x})e^{-ik_{+}\hat x\cdot z{'}} + \mathcal{J}_{1,Res}(x,z)
\end{align*}
with the residual term $\mathcal{J}_{1,Res}(x,z)$ satisfying
\begin{align*}
|\mathcal{J}_{1,Res}(x,z)|\le {C}|x|^{-1}\|g\|_{C[a,b]}.
\end{align*}
Next, we estimate $\mathcal{J}_2(x,z)$. Note that $\alpha_1<\beta<\alpha_3$. Then with the aid of \eqref{eq20} and the facts that $|\mathcal S_1 (\eta-\beta)+\mathcal S_1 (-\beta)|\geq\sqrt{\beta}>0$ for $\eta\in\mathbb{R}$
and $\int^{\alpha_3}_{\alpha_1}|\mathcal{S}^{-1}_1(\eta-\beta)|d\eta$ is bounded, we can
apply an integration by parts to obtain that
\begin{align*}
& e^{-ik_+|x|}\mathcal{J}_{2}(x,z)=\int_{\alpha_1}^{\alpha_3} e^{-ik_+|x|\eta^2/2}
\frac{\eta F_1(\eta,z)}{\mathcal{S}_1(\eta-\beta)+\mathcal{S}_1(-\beta)}d\eta\\
 &= \frac{i}{k_+|x|}\frac{e^{-ik_+|x|{\eta^2}/2}F_1(\eta, z)}{\mathcal S_1(\eta -\beta) + \mathcal S_1 ({-\beta})}\bigg|^{\alpha_3}_{\eta=\alpha_1}
+\frac{0.5i}{k_+|x|}\int^{\alpha_3}_{\alpha_1} \frac{e^{-ik_+|x|{\eta^2}/2}F_1(\eta, z)}{{(\mathcal S_1 (\eta-\beta)+\mathcal S_1 (-\beta))}^2 \mathcal S_1(\eta -\beta) }d\eta \\
&- \frac{i}{k_+|x|}\int^{\alpha_3}_{{\alpha_1}} \frac{e^{-ik_+|x|{\eta^2}/2} F'_1(\eta, z) }{\mathcal S_1(\eta -\beta)+\mathcal S_1 ({-\beta})}d\eta =: \mathcal{J}^{(1)}_{2}(x,z)+\mathcal{J}^{(2)}_{2}(x,z)+ \mathcal{J}^{(3)}_{2}(x,z).
\end{align*}
It follows from \eqref{eq17} that
\begin{align}\label{eq:18}
|\mathcal{J}^{(1)}_{2}(x,z)|+|\mathcal{J}^{(3)}_{2}(x,z)|\le\frac{C}{|x| \sqrt{\beta}}\|F_1(\cdot,z)\|_{C^1[
\alpha_1,\alpha_3]}.
\end{align}
Since
\begin{align*}
|\mathcal S_1(\eta-\beta)+ \mathcal S_1 ({-\beta})|=\begin{cases}
\sqrt{\beta-\eta}+ \sqrt{\beta},  &\;\eta \in (\alpha_1,\beta),\\
\sqrt {\eta},  & \; \eta\in (\beta,\alpha_3),
\end{cases}
\end{align*}
we have
\begin{align*}
&\left|\mathcal{J}^{(2)}_2(x,z)\right|\\
&\le\frac{\|F_1(\cdot,z)\|_{C[
\alpha_1,\alpha_3]}}{2k_+|x|}
\left[\int^{\beta}_{\alpha_1} \frac{1}{\left(\sqrt{\beta-\eta}+ \sqrt{\beta}\right)^2\sqrt{\beta-\eta}}d\eta + \int^{\alpha_3}_{\beta} \frac{1}{\eta\sqrt{\eta-\beta}}d\eta\right]\\
&= \frac{\|F_1(\cdot,z)\|_{C[
\alpha_1,\alpha_3]}}{k_+|x|}\left[\frac{1}{\sqrt{\beta-\eta}+\sqrt{\beta}}\Big|_{\eta=\alpha_1}^{\beta}+ \int^{\sqrt{\alpha_3-\beta}}_{0} \frac{1}{t^2+\beta}dt\right]\\
&\le \frac{(1+\pi/2)\|F_1(\cdot,z)\|_{C[
\alpha_1,\alpha_3]}}{k_+|x|\sqrt\beta}.
\end{align*}
This, together with (\ref{eq:15}), (\ref{eq:18}) and the assumption $\sqrt{k_{+}|x|}\beta> 1$, yields that
\begin{align*}
|\mathcal{J}_2(x,z)|\le  C {(1+|z|)\|g\|_{C^1[a,b]}}{|x|^{-3/4}}.
\end{align*}
Further, for $\mathcal{J}_3(x,z)$, it follows from the formulas \eqref{eq17}, (\ref{eq:31}) and (\ref{eq:21}) and Lemma \ref{NLe:6} that
\begin{align*}
|\mathcal{J}_3(x,z)|\le {C(1+|z|)^2\|g\|_{C^3[a,b]}}{|x|^{-1}}.
\end{align*}
Based on the above discussions, we now obtain that $V(x,z)$ has the form (\ref{eq:13}) with the residual term $V_{Res}(x,z)$ satisfying (\ref{eq:20}) uniformly for all $\theta_{\hat x} \in  [\delta_b/2,\pi-\delta_a/2]$ with $\sin((\theta_{\hat{x}}-\vartheta_0)/2)> (2\sqrt{k_+|x|})^{-1}$ and all $z\in\mathbb{R}^2$.

\textbf{Case 3}:
$\theta_{\hat x}\in [\delta_b/2,\pi-\delta_a/2]$ with $\sin((\theta_{\hat{x}}-\vartheta_0)/2)<-(2\sqrt{k_{+}|x|})^{-1}$
(that is, $\sqrt{k_{+}|x|}\beta<-1$). Using similar arguments as in Case 2, we can obtain that $V(x,z)$ has the form (\ref{eq:13}) with  the residual
term $V_{Res}(x,z)$ satisfying (\ref{eq:20}) uniformly for all $\theta_{\hat x} \in  [\delta_b/2,\pi-\delta_a/2]$ with $\sin((\theta_{\hat{x}}-\vartheta_0)/2)< -(2\sqrt{k_+|x|})^{-1}$ and all $z\in\mathbb{R}^2$.

\textbf{Part 2.2}: Estimate of $V(x,z)$ with $\theta_{\hat x}\in (0,\delta_b/2)\cup(\pi-\delta_a/2,\pi)$.
In this part, it is easy to see that for $\theta_d\in[a,b]$,
\begin{align*}
&|\sin(\theta_{\hat{x}}+\theta_d)|>\min[\sin(\delta_a/2),\sin(\delta_b/2)],\\
&|\cos\theta_d-\cos\theta_0|
\geq(2/\pi)\min(\sin\delta_a,\sin\delta_b)|\theta_d-\theta_0|.
\end{align*}
Thus by the fact that $\textrm{Supp}(g)\subset (a,b)$ and an integration by parts, we arrive at
$f(2\pi-\theta_{\hat{x}})=0$ and
\begin{align*}
&|V(x,z)|=\left|\frac{i}{k_+|x|}\int^b_a
\frac{d}{d\theta_d}\left(e^{ik_+|x|\cos(\theta_d+\theta_{\hat{x}})}\right)
\frac{e^{-ik_+|z|\cos(\theta_d-\theta_{\hat{z}})}f(\theta_d)}{\sin(\theta_d+\theta_{\hat{x}})}
d\theta_d
\right|\\
&\leq \frac{C \|g\|_{C^1[a,b]}}{|x|}
\left[
(1+|z|)\int^b_a|\mathcal{S}_1(\cos\theta_d-\cos\theta_0)|d\theta_d
+\int^b_a|\mathcal{S}_1(\cos\theta_d-\cos\theta_0)|^{-1}d\theta_d
\right]\\
&\leq C(1+|z|)|x|^{-1}\|g\|_{C^{1}[a,b]}.
\end{align*}
These imply that $V(x,z)$ has the form (\ref{eq:13}) with $V_{Res}(x,z)$ satisfying (\ref{eq:20}) uniformly for all $\theta_{\hat x} \in  (0,\delta_b/2)\cup(\pi-\delta_a/2,\pi)$ and all $z\in\mathbb{R}^2$.

Therefore, we obtain that the statement (\ref{b}) holds and the proof of this lemma is complete.
\end{proof}

Based on Lemma \ref{LL:2}, we have the following uniform far-field expansions of $U_2$ and $U_3$.

\begin{lemma}\label{lem2}
Let $x= |x|\hat x= |x|(\cos \theta_{\hat x}, \sin \theta_{\hat x})\in\mathbb{R}^2_+$ with $\theta_{\hat x} \in (0,\pi)$ and $z\in \R^2$.
Then the following statements hold true.
\begin{enumerate}[(a)]
\item\label{lem2-a}
$U_2(x,z)$ has the asymptotic behavior
\begin{align}\label{eq:33}
U_{2}(x,z) = \frac{e^{ik_{+}|x|}}{|x|^{1/2}}e^{-\frac{i\pi}{4}}{\left(\frac{2\pi}{k_{+}}\right)}^{1/2} \mathcal R(\theta_{\hat x}) e^{-ik_+\hat x\cdot z{'}} + U_{2,Res}(x,z)
\end{align}
with the residual term $U_{2,Res}(x,z)$ satisfying
\begin{empheq}[left={| U_{2,Res}(x,z) |\le \empheqlbrace}]{align}
&  \frac{C E(\theta_{\hat x},z)}{|x|}    &&\textrm{in the case}~ k_+<k_-,   \label{eq:36}
\\
& C\left[\frac{E(\theta_{\hat x},z)}{|x|}+\frac{(1+|z|)^2}{|x|^{3/4}}\right]&&\textrm{in the case}~ k_+>k_-, \label{eq:35}
\end{empheq}
as $|x| \rightarrow +\infty$ uniformly for all $\theta_{\hat x} \in (0,\pi)$ and $z\in\mathbb{R}^2$.
\item\label{lem2-b}
$U_3(x,z)$ has the asymptotic behavior
\begin{align}\label{eq:34}
U_{3}(x,z) = -\frac{e^{ik_{+}|x|}}{|x|^{ 1/2}}e^{-\frac{i\pi}{4}}\left({\frac{2\pi}{k_{+}}}\right)^{ 1/2}e^{-ik_{+}\hat x\cdot z} + U_{3,Res}(x,z)
\end{align}
with the residual term  $U_{3,Res}(x,z)$ satisfying
\begin{align*}
| U_{3,Res}(x,z) | \le C E(\theta_{\hat x},z)|x|^{-1}
\end{align*}
as $|x| \rightarrow +\infty$
uniformly for all $\theta_{\hat x} \in (0,\pi)$ and $z\in\mathbb{R}^2$.
\end{enumerate}
Here, $E(\theta_{\hat x},z)$ is given by (\ref{eq:29}) and $C>0$ is a constant independent of $x$, $z$.
\end{lemma}

\begin{proof}
We only prove the statement (\ref{lem2-a}). The proof of the statement (\ref{lem2-b})
is similar and easier, and thus we omit it.

Let $\theta_{\hat{z}}$ be given as in Lemma \ref{LL:2}.
Then it easily follows that
\begin{align}\label{eq21}
U_2(x,z) &=\int^{2\pi}_{\pi}  e^{ik_{+}|x|\cos(\theta_d+{\theta}_{\hat x})}e^{-ik_{+}|z|\cos(\theta_d-\theta_{\hat{z}})}\RR_0(\theta_d)d \theta_d. 
\end{align}
We note that
\begin{align}\label{eq23}
\RR_0(2\pi-\theta)=\RR(\theta)\quad \textrm{for}~\theta\in (0,\pi).
\end{align}
We distinguish between the two cases $k_+<k_-$ and $k_+>k_-$ to estimate $U_2$.

\textbf{Case 1}: $k_+<k_-$. Since $n>1$, it easily follows that
$\|\RR_0\|_{C^3[\pi,2\pi]}\leq C$. This, together with \eqref{eq23} and the statement (\ref{a}) of Lemma \ref{LL:2}, implies that $U_2(x,z)$ has the asymptotic behavior (\ref{eq:33}) with $U_{2,Res}(x,z)$ satisfying (\ref{eq:36}) as $|x|\rightarrow+\infty$ uniformly for all $\theta_{\hat x} \in  (0,\pi)$ and $z\in\mathbb{R}^2$.

\textbf{Case 2}: $k_+>k_-$. It is easy to see that for $\theta\in[\pi,2\pi]$,
\begin{align*}
\RR_0(\theta) = \frac{2i\sin\theta\mathcal S(\cos \theta, n)}{{1-n^2}} + \frac{n^2-1+2\sin^2\theta}{{1-n^2}}=:\RR_1(\theta)+\RR_2(\theta).
\end{align*}
Due to $k_+>k_-$, we notice that
\begin{align}\label{eq22}
\mathcal S(\cos\theta,n)&= \mathcal S_{1}(\cos\theta-\cos\theta_c) \mathcal S_{2}(\cos\theta+\cos\theta_c)\no\\
&=\mathcal S_{1}(\cos\theta-\cos\theta^{(2)}_c) \mathcal S_{2}(\cos\theta-\cos\theta^{(1)}_c)
\end{align}
with $\theta^{(1)}_c:=\pi+\theta_c\in(\pi,3\pi/2)$ and $\theta^{(2)}_c:=2\pi-\theta_c\in(3\pi/2,2\pi)$, which implies that $\mathcal S(\cos\theta,n)$
is infinitely differentiable for all $\theta\in[\pi,2\pi]$ except for the points
$\theta^{(1)}_c$ and $\theta^{(2)}_c$.

Let $\varepsilon>0$ be a fixed number such that $[\theta^{(1)}_c-2\varepsilon,\theta^{(1)}_c+2\varepsilon]\subset(\pi+\theta_c/2,3\pi/2)$
and $[\theta^{(2)}_c-2\varepsilon,\theta^{(2)}_c+2\varepsilon]\subset(3\pi/2,2\pi-\theta_c/2)$.
Choose the cutoff functions $\chi_1,\chi_2\in C^\infty_0(\mathbb{R})$ such that
\begin{align*}
0\leq\chi_l\leq 1~\textrm{in}~\mathbb{R},\quad\chi_l=1~\textrm{in}~[\theta^{(l)}_c-\varepsilon,\theta^{(l)}_c+\varepsilon],
\quad\textrm{Supp}(\chi_l)\subset[\theta^{(l)}_c-2\varepsilon,\theta^{(l)}_c+2\varepsilon]
\end{align*}
for $l=1,2$.
Then $\RR_0(\theta)$ can be rewritten as
\begin{align*}
\RR_0(\theta)&=\chi_1(\theta)\RR_1(\theta)
+\chi_2(\theta)\RR_1(\theta)+[(1-\chi_1(\theta)-\chi_2(\theta))\RR_1(\theta)+\RR_2(\theta)]\\
&=:\widetilde{\RR}_1(\theta)+ \widetilde{\RR}_2(\theta) + \widetilde{\RR}_3(\theta),\qquad\qquad\qquad\qquad\qquad\theta\in[\pi,2\pi],
\end{align*}
and thus we have from \eqref{eq21} that $U_2(x,z)=U_{2,1}(x,z)+U_{2,2}(x,z)+U_{2,3}(x,z)$
with
\begin{align*}
U_{2,m}(x,z) &:=\int^{2\pi}_{\pi}  e^{ik_{+}|x|\cos(\theta_d+{\theta}_{\hat x})}e^{-ik_{+}|z|\cos(\theta_d-\theta_{\hat{z}})}\widetilde{\RR}_m(\theta_d)d \theta_d,
\quad m=1,2,3.
\end{align*}
It follows from \eqref{eq22} that
$\widetilde{\RR}_1(\theta)=\mathcal{S}_2(\cos\theta-\cos\theta^{(1)}_c)\mathcal{G}_1(\theta)$
and $\widetilde{\RR}_2(\theta)=\mathcal{S}_1(\cos\theta-\cos\theta^{(2)}_c)\mathcal{G}_2(\theta)$
with
\begin{align*}
\mathcal{G}_1(\theta):=\frac{2i\sin\theta\mathcal{S}_1(\cos\theta-\cos\theta^{(2)}_c)\chi_1(\theta)}{1-n^2},
\quad\mathcal{G}_2(\theta):=\frac{2i\sin\theta\mathcal{S}_2(\cos\theta-\cos\theta^{(1)}_c)\chi_2(\theta)}{1-n^2}.
\end{align*}
Clearly, we have that $\textrm{Supp}(\mathcal{G}_m)\subset(\pi+\theta_c/2,2\pi-\theta_c/2)$
and $\|\mathcal{G}_m\|_{C^3[\pi+\theta_c/2,2\pi-\theta_c/2]}\leq C$ ($m=1,2$) and that
$\|\widetilde{\mathcal{R}}_3\|_{C^3[\pi,2\pi]}\leq C$.
Hence, by using \eqref{eq23}, applying the statement (\ref{b}) of Lemma \ref{LL:2} to $U_{2,m}$ $(m=1,2)$ and applying the statement (\ref{a}) of Lemma \ref{LL:2} to $U_{2,3}$, we have that
$U_2(x,z)$ has the asymptotic behavior (\ref{eq:33}) with $U_{2,Res}(x,z)$ satisfying (\ref{eq:35}) as $|x|\rightarrow+\infty$ uniformly for all $\theta_{\hat x} \in  (0,\pi)$ and $z\in\mathbb{R}^2$.

Therefore, the proof is complete.
\end{proof}

\begin{remark}
It was proved in \cite[Lemma A.3]{xzz19} that $U_3(x,z)$ has
the asymptotic behavior \eqref{eq:34}
with the residual term  $U_{3,Res}(x,z)$ satisfying
\begin{align*}
|U_{3,Res}(x,z)|
\leq&
C\bigg(|z|+\frac{1}{|\sin\frac{\theta_{\hat{x}}}{2}|}
+\frac{1}{|\sin\frac{\pi-\theta_{\hat{x}}}{2}|}
+\frac{1}{|\sin\theta_{\hat{x}}|}
\\
&+\int^{\theta_{\hat{x}}}_0\frac{(1+|z|)^3t^2}{\sin^2t}dt
+\int^{\pi-\theta_{\hat{x}}}_0\frac{(1+|z|)^3 t^2}{\sin^2t}dt\bigg)\frac{1}{|x|}
\end{align*}
as $|x|\rightarrow+\infty$ uniformly for all $\theta_{\hat{x}}\in(0,\pi)$ and $z\in\mathbb{R}^2$.
We note that the statement (\ref{lem2-b}) of Lemma \ref{lem2} improves the above result due to the fact that for any $\theta_{\hat{x}}\in(0,\pi)$,
\begin{align*}
&(1+|z|)^2 + (1+|z|)\left|\int^{\pi-\theta_{\hat{x}}}_{\theta_{\hat{x}}}\frac{1}{\sin^2 t}dt\right|\\
&\leq C(1+|z|)^2\left(\int^{{\pi}/{2}}_{0}\frac{t^2}{\sin^2 t}dt+\int^{\max(\theta_{\hat{x}},\pi-\theta_{\hat{x}})}_{{\pi}/{2}}\frac{t^2}{\sin^2 t}dt\right)\\
&\leq C\left(\int^{\theta_{\hat{x}}}_0\frac{(1+|z|)^2t^2}{\sin^2 t}dt+\int^{\pi-\theta_{\hat{x}}}_0\frac{(1+|z|)^2t^2}{\sin^2 t}dt\right),
\end{align*}
where $C>0$ is a constant independent of $\theta_{\hat{x}}$ and $z$.
\end{remark}

Based on the above lemmas, we now have the following theorem on the relation between $I_F(z)$ given in (\ref{c:8}) and $I_S(z,R)$ given in \eqref{Ne14}.

\begin{theorem} \label{NL4}
Let $z\in \R^2$ and $R>0$ be large enough, then we have $I_S(z,R)= I_F(z)+ I_{S,Res}(z,R)$ with
the residual term
$I_{S,Res}(z,R)$ satisfying
\ben
|I_{S,Res}(z,R)|\le {C(1+|z|)^3}{R^{-1/4}}\quad\textrm{as}~R\rightarrow+\infty
\enn
uniformly for all $z\in\mathbb{R}^2$. Here, $C>0$ is a constant independent of $R$ and $z$.
\end{theorem}
\begin{proof}
We only consider the case $k_+>k_-$ since the proof for the case $k_+<k_-$ is similar.
Let $R$ be large enough throughout the proof.
It follows from \eqref{eq33}, \eqref{eq24}, \eqref{eq:33} and \eqref{eq:34} that
for $x\in\pa B^+_R$ and $z\in\mathbb{R}^2$,
we can write
$U(x,z)=U_0(x,z)+U_{Res}(x,z)$ with $U_0(x,z)$ and $U_{Res}(x,z)$ given by
\begin{align*}
&U_0(x,z):= \\
&\frac{e^{ik_{+}|x|}}{\sqrt{|x|}}\left[\int_{\Sp^1_{-}}u^{\infty}(\hat x, d)e^{-ik_{+}z\cdot d}ds(d)+{\bigg(\frac{2\pi}{k_{+}}\bigg)}^{\frac 12}e^{{-i\pi}/{4}}\left(\RR(\theta_{\hat x})e^{-ik_{+}\hat x\cdot z{'}}-e^{-ik_{+}\hat x \cdot z}\right)\right],\\
&U_{Res}(x,z):= U_{1,Res}(x,z)+ U_{2,Res}(x,z) + U_{3,Res}(x,z).
\end{align*}
Then it is clear that
\begin{align}\label{c:5}
&I_S(z,R)=I_F(z)+\int_{\partial B^+_R}\left[U_0(x,z)\overline{U_{Res}(x,z)}+\overline{U(x,z)}U_{Res}(x,z)\right]
ds(x).
\end{align}
By using the fact that $\|\mathcal{R}\|_{C[0,\pi]}\leq C$, the formula \eqref{eq33}, Lemma \ref{NL1} and Remark \ref{re2},
we obtain that for $x\in\partial B^+_R$ and $z\in\mathbb{R}^2$,
\begin{align}\label{eq35}
|U_0(x,z)|\leq CR^{-1/2}\quad\textrm{and}\quad
|U(x,z)|,~|U_{Res}(x,z)|\leq C(1+|z|)R^{-1/2}.
\end{align}
Let $\delta:=R^{-1/4}$ to be small enough
and define $\partial {B^{+}_{R,\delta}} :=\{x=R(\cos\theta_{\hat x},\sin\theta_{\hat x})\in \R^2: \theta_{\hat x}\in(0,\delta)\cup(\pi-\delta,\pi)\}$. It is easy to see that for $\theta_{\hat{x}}\in[\delta,\pi-\delta]$
and $z\in\mathbb{R}^2$,
\begin{align*}
|E(\theta_{\hat{x}},z)|\leq (1+|z|)^2 + C\delta^{-1} +(1+|z|)\int^{\pi-\delta}_{\delta}\frac{1}{\sin^2 t}dt\leq C(1+|z|)^2\delta^{-1},
\end{align*}
where $E(\theta_{\hat{x}},z)$ is given as in \eqref{eq:29}.
Thus we have from Lemmas \ref{LL:1} and \ref{lem2} that for $x\in\pa B^+_R\ba\pa B^+_{R,\delta}$ and $z\in\mathbb{R}^2$,
\begin{align*}
|U_{Res}(x,z)|\leq C\left[\frac{1}{R^{3/4}}+\frac{(1+|z|)^2}{\delta R}+\frac{(1+|z|)^2}{R^{3/4}}\right]\leq
C(1+|z|)^2 R^{-3/4}.
\end{align*}
This, together with \eqref{c:5} and \eqref{eq35}, implies that
\begin{align*}
&|I_S(z,R)-I_F(z)|\\
&=\left|\left\{\int_{\partial B^+_R\ba \pa B^+_{R,\delta}}+\int_{\pa B^+_{R,\delta}}\right\}\left[U_0(x,z)\overline{U_{Res}(x,z)}+\overline{U(x,z)}U_{Res}(x,z)\right]
ds(x)\right|\\
&\leq CR\left[\frac{1}{R^{1/2}}\frac{(1+|z|)^2}{R^{3/4}}+\frac{1+|z|}{R^{1/2}}\frac{(1+|z|)^2}{R^{3/4}}\right]
+CR\delta\left[\frac{1}{R^{1/2}}\frac{1+|z|}{R^{1/2}}+\frac{1+|z|}{R^{1/2}}\frac{1+|z|}{R^{1/2}}\right]\\
&\leq C(1+|z|)^3 R^{-1/4}.
\end{align*}
Therefore, the proof is complete.
\end{proof}

\subsection{Direct imaging methods}
\label{sec2}

With the analysis in Sections \ref{sec1} and \ref{sec:4},
now we are ready to study
the direct imaging methods for the inverse problems (IP1) and (IP2).
In the rest of the paper, let $K$ be a bounded domain containing the local perturbation $\Gamma_p$ of the locally rough surface $\Gamma$.
Note that the imaging function $I_F(z)$ is independent of the radius $R$.
Thus it can be seen from Theorems \ref{th:1} and \ref{NL4} that when $R$ is sufficiently large, the imaging functions $I_{P}(z,R)$ and $I_F(z)$ are approximately equal to the function
$I_S(z,R)$ for any $z\in K$. This means that when $R$ is sufficiently large, $I_{P}(z,R)$ and $I_F(z)$ have similar properties as $I_S(z,R)$ for $z\in K$.

Next, we study the properties of $I_S(z,R)$ by employing the theory of scattering by a penetrable unbounded rough surface.
To this end, we introduce some notations. For $b\in \mathbb{R}$, let $U_b^{\pm}:=\big\{(x_1,x_2)\in \R^2: x_2\gtrless b \big\}$  and $\Gamma_b :=\big\{(x_1,x_2)\in \R^2 : x_2= b \big\}$.
For $V\subset \R^n$ ($n=1,2$), denote by $BC(V)$ the set of bounded and continuous functions in $V$, a Banach space under the norm $\|\phi\|_{\infty,V}:=\sup_{x\in V}|\phi(x)|$.
For $0<\alpha <1$ and $V\subset \R^n$ ($n=1,2$), denote by $BC^{0,\alpha}(V)$ the Banach space of functions $\phi \in BC(V)$, which are uniformly H\"{o}lder continuous with exponent $\alpha$,
with the norm $\|\cdot\|_{0,\alpha,V}$ defined by $\|\phi\|_{0,\alpha,V}:= \|\phi\|_{\infty,V}+ \sup_{x,y\in V,x\ne y}{|\phi(x)-\phi(y)|}/{|x-y|^{\alpha}}$. Further, for $V\subset\mathbb{R}^2$, define $BC^{1}(V):=\big\{\phi \in BC(V) : \partial_j\phi \in BC(V), \; j=1,2\big\}$ with the norm $\|\phi\|_{1,V}:=\|\phi\|_{\infty,V}+ \sum^2_{j=1}\|\partial_j\phi\|_{\infty,V}$, where $\partial_j\phi$ denotes the derivative ${\partial{\phi}}/{\partial x_j}$ for $j=1,2$.
Moreover, for $0<\alpha<1$ and the surface $\Gamma$, let $BC^{1,\alpha}(\Gamma):=\{\varphi\in BC(\Gamma)
:\Grad \varphi\in BC^{0,\alpha}(\Gamma)\}$
under the norm $\|\varphi\|_{1,\alpha,\Gamma}:=\|\varphi\|_{\infty,\Gamma}
+\|\Grad \varphi\|_{0,\alpha,\Gamma}$
where $\textrm{Grad}$ denotes the surface gradient.
Then the scattering problem by a penetrable unbounded rough surface can be formulated as follows.

\textbf{Transmission scattering problem (TSP)}. Let $\alpha \in (0,1)$,
$h_1>\sup_{x_1\in\mathbb{R}}h_\Gamma(x_1)$ and
$h_2<\inf_{x_1\in\mathbb{R}}h_\Gamma(x_1)$.
Given $g_1\in BC^{1,\alpha}(\Gamma)$ and $g_2\in BC^{0,\alpha}(\Gamma)$, determine a pair of
solutions $(v_1,v_2)$ with $v_1\in C^2(\Omega_{+}) \cap BC^{1}(\overline{\Omega_{+}}\se U^+_{h_1})$
and $v_2\in C^2(\Omega_{-})\cap BC^{1}(\overline{\Omega_{-}}\se U^-_{h_2})$
such that the following hold:

(i) $v_1$ is a solution of the Helmholtz equation $\Delta v_1+ k^2_{+} v_1=0$ in $\Omega_{+}$
and $v_2$ is a solution of the Helmholtz equation $\Delta v_2+k^2_{-}v_2 =0$ in $\Omega_{-}$.

(ii) $v_1$ and $v_2$ satisfy the transmission boundary condition
\begin{align*}
v_1-v_2=g_1,\quad \partial v_1/\partial{\nu}-\partial v_2/\partial{\nu}=g_2\quad\textrm{on}~ \Gamma.
\end{align*}

(iii) $v_1$ and $v_2$ satisfy the growth conditions in the $x_2$ direction:
for some $\beta\in \mathbb{R}$,
\be\label{c:3}
\sup_{x\in\Omega_{+}} |x_2|^{\beta}|v_1(x)|<+\infty, \; \quad
\sup_{x\in\Omega_{-}} |x_2|^{\beta}|v_2(x)|<+\infty.
\en

(iv) $v_1$ satisfies the upward propagating radiation condition (UPRC): for some
$\phi_1\in L^{\infty}(\Gamma_{h_1})$,
\begin{align}\label{eq36}
v_1(x)=2\int_{\G_{h_1}}\frac{\partial\Phi_{k_{+}}(x,y)}{\partial y_2}\phi_1(y)ds(y),\quad x\in U^+_{h_1}.
\end{align}
$v_2$ satisfies the downward propagating radiation condition (DPRC): for some $\phi_2\in L^{\infty}(\Gamma_{h_2})$,
\begin{align}\label{eq37}
v_2(x)= -2\int_{\Gamma_{h_2}}\frac{\partial \Phi_{k_{-}}(x,y)}{\partial y_2}\phi_2(y)ds(y),\quad x\in U^-_{h_2}.
\end{align}
Here, $\Phi_{k}(x,y)$ with $k>0$ is
the fundamental solution of the Helmholtz equation $\Delta u + k^2 u=0$ in two dimensions, that is,
$\Phi_{k}(x,y):= ({i}/4) H^1_0(k|x-y|)$, $x,y\in\mathbb{R}^2$, $x\neq y$,
where $H^1_0$ denotes the Hankel function of the first kind of order zero.

The well-posedness of the problem (TSP) has been established in \cite{CSZ1999,DTS,LI2016} by using the integral equation method.

To proceed further, we need the following property of the total-field $u^{tot}(x,d)$.

\begin{lemma}\label{lem3}
For any $d\in\mathbb{S}^1_-$, we have $u^{tot}(\cdot,d)|_{\Omega_\pm}\in C^2(\Omega_\pm)$
and $u^{tot}(\cdot,d)\in BC^1(\mathbb{R}^2)$.
\end{lemma}

\begin{proof}
Let $d\in\mathbb{S}^1_-$.
It easily follows from Theorem \ref{thm2} and elliptic regularity estimates (see, e.g., \cite[Section 6.3]{E10})
that $u^{tot}(\cdot,d)\in H^2_{loc}(\mathbb{R}^2)$ and that
$u^{tot}(\cdot,d)|_{V_0}\in H^\ell(V_0)$ for any positive integer $\ell$ and any bounded open set $V_0$ satisfying $\ov{V_0}\subset \Omega_{+}\cup \Omega_{-}$,
which implies
that
$u^{tot}(\cdot,d)\in C(\mathbb{R}^2)$
and
$u^{tot}(\cdot,d)|_{\Omega_\pm}\in C^2(\Omega_\pm)$.
Moreover,
it can be seen from \cite[Theorems 13 and 14]{LYZZ3} that for both the case $k_+<k_-$ and the case $k_+>k_-$, the scattered field $u^s(x,d)$ has the asymptotic behavior
\ben
u^{s}(x,d)=\frac{e^{ik_{-}|x|}}{\sqrt{|x|}}u^{\infty}(\hat x, d)+ u^s_{Res}(x,d)\quad\textrm{for}~x\in\Omega_-\ba\ov{B_R},
\enn
where the far-field patten $u^\infty(\hat{x},d)$ of the scattered field $u^s(x,d)$  satisfies $u^\infty(\cdot,d)\in C(\ov{\mathbb{S}^1_-})$ and $u^s_{Res}(x,d)$
satisfies
\begin{align*}
&\left|u^s_{Res}(x,d)\right|\le{C}{|x|^{-3/4}},\quad |x|\rightarrow+\infty,
\end{align*}
uniformly for all $\theta_{\hat{x}}\in(\pi,2\pi)$ and $d \in \Sp^1_{-}$ (the expression of $u^\infty(\hat{x},d)$ with $\hat{x},d\in\mathbb{S}^1_-$
can be seen in \cite[formula (110)]{LYZZ3}, which is similar to \eqref{eq29}).
Note further that $u^0(\cdot,d)\in BC(\R^2)$.
Thus, it follows from the above discussions and Lemma \ref{Le:my} that $u^{tot}(\cdot,d)\in BC(\R^2)$. This, together with the local regularity estimate in \cite[Theorem 2.7]{CSZ1998},
implies that $u^{tot}(\cdot,d)\in BC^1(\R^2)$.
\end{proof}

For $x\in \Omega_{+}$ and $z\in \R^2$, let $U(x,z)$ be given as in \eqref{eq38}, which is
involved in $I_S(z,R)$.
For $x\in \Omega_{-}$ and $z\in \R^2$,
define $V(x,z):= \int_{\s^1_{-}}u^{tot}(x,d)e^{-ik_+z\cdot d}ds(d)$.
Then with the aid of Lemma \ref{lem3}, we show in the following theorem that for any fixed $z\in\mathbb{R}^2$,
the pair of functions $(U(x,z),V(x,z))$ is the unique solution to the problem (TSP) with the boundary data related to the Bessel function of order $0$.

\begin{theorem}\label{thm1}
For any fixed $z\in\mathbb{R}^2$, the pair of functions $(U(x,z),V(x,z))$ solves the problem (TSP) with the boundary data
\begin{align*}
g_1(x)= -2\pi J_0(k_{+}|x-z|),\quad g_2(x)=-2\pi \partial J_0(k_{+}|x-z|)/\partial{\nu}(x),\quad x\in\Gamma,
\end{align*}
where $J_0$ is the Bessel function of order $0$.
\end{theorem}

\begin{proof}
Let $d\in \Sp^1_{-}$, $h_1>\sup_{x_1\in\mathbb{R}}h_\Gamma(x_1)$ and
$h_2<\inf_{x_1\in\mathbb{R}}h_\Gamma(x_1)$. Define ${\widetilde{v}}_1(x,d)= u^{tot}(x,d)-u^i(x,d)$ for $x\in \Omega_{+}$ and $\widetilde{v}_2(x,d)= u^{tot}(x,d)$ for $x\in \Omega_{-}$.
It follows from Lemma \ref{lem3} that
${\wid{v}}_1(\cdot,d)\in C^2(\Omega_+)\cap BC^{1}(\overline{\Omega_+}\se U^+_{h_1})$,
${\wid{v}}_2(\cdot,d)\in C^2(\Omega_-)\cap BC^{1}(\overline{\Omega_{-}}\se U^-_{h_2})$,
and ${\widetilde{v}}_1(\cdot,d)$ and ${\widetilde{v}}_2(\cdot,d)$ satisfy (\ref{c:3})
with $\beta=0$. Furthermore, we note that ${\wid{v}}_1(x,d)= u^{s}(x,d)+u^r(x,d)$ for
$x\in  U^+_{h_1}$ and ${\widetilde{v}}_2(x,d)= u^{s}(x,d)+u^t(x,d)$ for $x\in U^-_{h_2}$.
Thus, applying (\ref{eq3}) and (\ref{eq2}) and using \cite[Theorem 2.9 and Remark 2.14]{CSZ1998}
give that ${\widetilde{v}}_1(\cdot,d)$ and ${\wid{v}}_2(\cdot,d)$ fulfill (\ref{eq36})
and \eqref{eq37}, respectively. Moreover, we can obtain from \eqref{e:1.1} and \eqref{eq5}
that $\Delta_x {\widetilde{v}}_1(x, d)+ k^2_{+}{\wid{v}}_1(x, d)=0$ in $\Omega_{+}$,
$\Delta_x {\widetilde{v}}_2(x, d)+ k^2_{-}{\wid{v}}_2(x, d)=0$ in $\Omega_{-}$, and
${\widetilde{v}}_1(x, d)$ and ${\wid{v}}_2(x, d)$ satisfy the transmission boundary condition
\begin{align*}
\widetilde v_1(x,d)-\widetilde v_2(x,d)= -e^{ik_{+}x\cdot d}, \quad \frac{\partial \widetilde v_1(x,d)}{\partial{\nu(x)}}-\frac{\partial \widetilde v_2(x,d)}{\partial{\nu(x)}}=-\frac{\partial e^{ik_{+}x\cdot d}}{\partial{\nu}(x)}\quad \textrm{on}~ \Gamma.
\end{align*}
Based on the above discussions, we obtain that the pair of functions $(\wid v_1(x,d),\wid v_2(x,d))$
is the solution of the problem (TSP) with the boundary data $g_1(x)=-e^{ik_{+}x\cdot d}$ and
$g_2(x)=-\partial e^{ik_{+}x\cdot d}/\partial{\nu}(x)$ on $\Gamma$, whence the statement follows
with the aid of \cite[Theorem 3.2]{Z_20}.
\end{proof}

\begin{remark}\label{re:1}
In \cite[Section 3.1]{Lzzd}, the properties of the solution to the problem (TSP) with the boundary
data $g_1(x)=a J_0(k_{+}|x-z|)$
and $ g_2(x)= a \partial J_0(k_{+}|x-z|)/{\partial\nu}(x)$, $x\in \Gamma$, for some constant
$a\in\mathbb{C}$ have been studied in the case when $\Gamma$ is a globally  rough surface.
With the help of the discussions in \cite[Section 3.1]{Lzzd} and Theorem \ref{thm1}, it is
expected that for any $x$ in the bounded subset of $\Omega_+$,
$U(x,z)$ will take a large value when $z\in \Gamma$ and decay as $z$ moves away from $\Gamma$.
Consequently, it is expected that for any fixed $R>0$ such that $\Gamma_p\subset B_R$,
$I_S(z,R)$ will take a large value when $z\in \Gamma$ and decay as $z$ moves away from $\Gamma$.
\end{remark}

With these preparations, we then turn to the direct imaging method for the inverse problem (IP1).
Based on the discussions at the beginning of this subsection and Remark \ref{re:1},
it is expected that if $R$ is sufficiently large, then
for $z\in K$
the imaging function
$I_{P}(z,R)$ will take a large value when $z\in \Gamma\cap K$ and decay as $z$ moves away from
$\Gamma\cap K$. This property is indeed confirmed in the numerical examples carried out later.
In the numerical experiments, we measure the  phaseless total-field data $|u^{tot}(x^{(p)},d^{(q)})|$
with $p = 1,2, \ldots, M_P$ and $q = 1, 2, \ldots, N_P$, where $x^{(p)}$ and $d^{(q)}$ are uniformly
distributed points on $\partial {B_{R}^{+}}$ and $\Sp_{-}^1$, respectively.
Accordingly, the imaging function $I_{P}(z,R)$ can be approximated as
\begin{align}
I_P(z,R)\approx
\label{a:1}
&\frac{R\pi^3}{M_P N^2_P}\sum^{M_P}_{p=1}
\Bigg|\sum^{N_P}_{q=1}\bigg\{\left[\left|u^{tot}(x^{(p)},d^{(q)})\right|^2-A^{(1)}_P(x^{(p)},d^{(q)})\right]
e^{ik_+(x^{(p)}-z)\cdot d^{(q)}}\no\\
&-
A^{(2)}_P(x^{(p)},d^{(q)},z)\bigg\}\Bigg|^2,
\end{align}
where $A^{(1)}_P(x,d):=1+|\mathcal{R}_0(\theta_d)|^2+\ov{\mathcal{R}_0(\theta_d)}e^{2ik_+ x_2 d_2}$
for $x\in\pa B^+_R$ and $d\in\mathbb{S}^1_-$, and
$A^{(2)}_P(x,d,z):=\exp(ik_+(x'-z')\cdot d)$ for $x\in\pa B^+_R$, $d\in\mathbb{S}^1_-$
and $z\in\mathbb{R}^2$.
Then for the inverse problem (IP1), we expect that the locally rough surface $\Gamma$
can be reconstructed by using the formula \eqref{a:1}.
Now the direct imaging method for the inverse problem (IP1) is described in Algorithm \ref{A1}.

\begin{algorithm}
\caption{\textbf{Direct imaging method for the inverse problem (IP1)}\label{A1}}
Let $K$ be the sampling region which contains the local perturbation $\Gamma_p$
of the penetrable locally rough surface $\Gamma$.

\begin{algorithmic}[1]
\STATE{Choose $\mathcal{T}_m$ to be a mesh of $K$
and let $R$ be sufficiently large.}

\STATE{
Collect the phaseless total-field data
$|u^{tot}(x^{(p)}, d^{(q)})|$,  $p = 1,2, \ldots, M_P$, $q = 1, 2, \ldots, N_P$, with
$x^{(p)}\in \partial {B_{R}^{+}}$ and $d^{(q)}\in \Sp^1_{-}$,
generated by the incident plane waves $u^i(x, d^{(q)})=e^{ik_+ x\cdot d^{(q)}}$,
$q = 1, 2,\ldots,N_P$.}

\STATE{
For each sampling point $z\in \mathcal{T}_m$,
approximately compute the imaging function $I_{P}(z,R)$ by using (\ref{a:1}).}

\STATE{
Locate all those sampling points $z\in\mathcal{T}_m$ such that $I_{P}(z,R)$ takes a large value, which represent the part of the locally rough surface $\Gamma$ in the sampling region $K$.}
\end{algorithmic}
\end{algorithm}

Next, we consider the direct imaging method for the inverse problem (IP2).
By using the discussions at the beginning of this subsection and Remark \ref{re:1} again,
it is expected that for $z\in K$, the imaging function $I_F(z)$ will take a large value when $z\in \Gamma\cap K$ and decay as $z$ moves away from $\Gamma\cap K$. This property is also confirmed in the numerical examples carried out later.
In numerical experiments, we measure the phased far-field data $u^{\infty}(\hat{x}^{(p)}, d^{(q)})$ with  $p = 1,2, \ldots, M_F$ and $q = 1, 2, \ldots, N_F$, where ${\hat x}^{(p)}$ and $d^{(q)}$ are uniformly distributed points on ${\Sp^{1}_{+}}$ and $\Sp^1_{-}$, respectively.
Accordingly, $I_{F}(z)$ can be approximated as
\begin{align}\label{a:2}
I_{F}(z)\approx \frac{\pi}{M_F}\sum^{M_F}_{p=1}\left|\frac{\pi}{N_F}\left(\sum^{N_F}_{q=1}u^{\infty}({\hat x}^{(p)}, d^{(q)})e^{-ik_{+}z\cdot d^{(q)}}\right)+A_F(\hat{x}^{(p)},z)\right|^2,
\end{align}
where $A_F(\hat{x},z):={\left({2\pi}/{k_{+}}\right)}^{1/2}e^{{-i\pi}/{4}}[{\RR}(\theta_{\hat x})e^{-ik_{+}\hat{x} \cdot z{'}}
-e^{-ik_{+}{\hat x}\cdot z}]$ for $\hat{x}\in\mathbb{S}^1_+$ and $z\in\mathbb{R}^2$.
Then similarly to Algorithm \ref{A1}, we describe the direct imaging method for the inverse problem (IP2) in Algorithm \ref{A2}.

\begin{algorithm}
\caption{\textbf{Direct imaging method for the inverse problem (IP2)}\label{A2}}
Let $K$ be the sampling region which contains the local perturbation $\Gamma_p$
of the penetrable locally rough surface $\Gamma$.

\begin{algorithmic}[1]
\STATE{
Choose $\mathcal{T}_m$ to be a mesh of $K$.}

\STATE{
Collect the phased far-field data
$u^{\infty}(\hat{x}^{(p)}, d^{(q)})$,  $p = 1,2, \ldots, M_F$, $q = 1, 2, \ldots, N_F$, with
$\hat{x}^{(p)}\in \mathbb{S}^1_+$ and $d^{(q)}\in \Sp^1_{-}$,
generated by the incident plane waves $u^i(x, d^{(q)})=e^{ik_+ x\cdot d^{(q)}}$,
$q = 1, 2,\ldots,N_F$.}

\STATE{
For each sampling point $z\in \mathcal{T}_m$,
approximately compute the imaging function $I_{F}(z)$ by using (\ref{a:2}).}

\STATE{
Locate all those sampling points $z\in\mathcal{T}_m$ such that $I_{F}(z)$ takes a large value,
which represent the part of the locally rough surface $\Gamma$ in the sampling region $K$.}
\end{algorithmic}
\end{algorithm}

\section{Numerical experiments} \label{num}

In this section, we will present several numerical examples to illustrate the applicability of
our direct imaging methods for the inverse problems (IP1) and (IP2). To generate the synthetic data,
the direct scattering problem (\ref{e:1.1})--(\ref{eq2}) is solved by the perfectly matched
layer-based boundary integral equation method proposed in \cite{LP_17}. In all the examples,
we will present the imaging results of $I_P(z,R)$ with phaseless total-field data (i.e. the
results of Algorithm \ref{A1}) and the imaging results of $I_F(z)$
with phased far-field data (i.e. the results of Algorithm \ref{A2}). Further, the noisy phaseless
total-field data $|u^{tot}_{\delta}(x^{(p_1)}, d)|$ with $x^{(p_1)}\in\partial{B_R^+},d\in\Sp^{1}_{-}$
($p_1 = 1,2, \ldots, M_P$) and the noisy phased far-field data $u^{\infty}_{\delta}(\hat x^{(p_2)},d)$
with ${\hat x}^{(p_2)}\in  {\Sp^{1}_{+}}, d\in \Sp^{1}_{-}$ ($p_2 = 1,2, \ldots, M_F$) are given by
\begin{align*} 
&|u^{tot}_{\delta}(x^{(p_1)},d )|=|u^{tot}(x^{(p_1)},d)|
+\delta\frac{\xi^{(p_1)}}{\sqrt{\sum^{M_P}_{l=1}|\xi^{(l)}|^2}}\sqrt{\sum^{M_P}_{l=1}|u^{tot}(x^{(l)},d)|^2}, \\
&u^{\infty}_{\delta}({\hat x}^{(p_2)}, d)=u^{\infty}({\hat x}^{(p_2)},d)
+\delta\frac{\zeta^{(p_2)}+i\eta^{(p_2)}}{\sqrt{\sum^{M_F}_{l=1}|\zeta^{(l)}+i\eta^{(l)}|^2}}
\sqrt{\sum^{M_F}_{l=1}|u^{\infty}({\hat x}^{(l)}, d)|^2},
\end{align*}
where $\delta$ is the noise ratio and where $\xi^{(p_1)}$ ($p_1=1,2,\ldots, M_P$) and
$\zeta^{(p_2)}$, $\eta^{(p_2)}$ ($p_2=1,2,\ldots, M_F$) are the standard normal distributions.

In each figure presented below, we use a solid line to represent the actual locally rough surface
against the reconstructed locally rough surface.

\textbf{Example 1.} We consider the case when the locally rough surface is given by
\ben
h_{\Gamma}(x_1)= \begin{cases}
0.4\sin\left[(x_1^2-{16}/{25})^2\right]\sin^3{({2\pi x_1}/{3} )}, &|x_1|\le  4/5, \\
0,  &|x_1|>  4/5.
\end{cases}
\enn
We choose the wave numbers $k_+ =40$ and $k_-=80$ and set the noise ratio $\delta  = 10\%$.
First, we consider the inverse problem (IP1)
and investigate the effect of the radius $R$ of the measurement circle $\partial B^+_R$ on the imaging results.
The numbers of the measurement points and the incident directions are chosen to be $M_P = N_P = 300$. Figures \ref{fig1}(a), \ref{fig1}(b) and \ref{fig1}(c) present the imaging results of $I_{P}(z,R)$ with
the measured phaseless total-field data with the radius of the measurement circle $\partial B^+_R$
to be $R = 1.5$, $2$, $3$, respectively. It is shown in Figures \ref{fig1}(a)--\ref{fig1}(c)
that the reconstruction result is getting better if the radius of the measurement circle is getting larger.
Secondly, we consider the inverse problem (IP2). The numbers of the measured observation directions
and the incident directions are chosen to be $M_F = N_F = 100$. Figure \ref{fig1}(d) presents
the imaging result of $I_{F}(z)$ with the measured phased far-field data. As shown in
Figure \ref{fig1}, the reconstruction result of $I_{F}(z)$ with the measured phased far-field
data is better than those of $I_{P}(z,R)$ with the measured phaseless total-field data.
\begin{figure}[htbp]
\centering
\hspace{0.15\textwidth}
\subfigure[$10\%$ noise, $k_{+} = 40$, $k_{-} = 80$, $R = 1.5$]
{
\begin{minipage}[t]{0.3\textwidth}
\centering
\includegraphics[width=\textwidth]{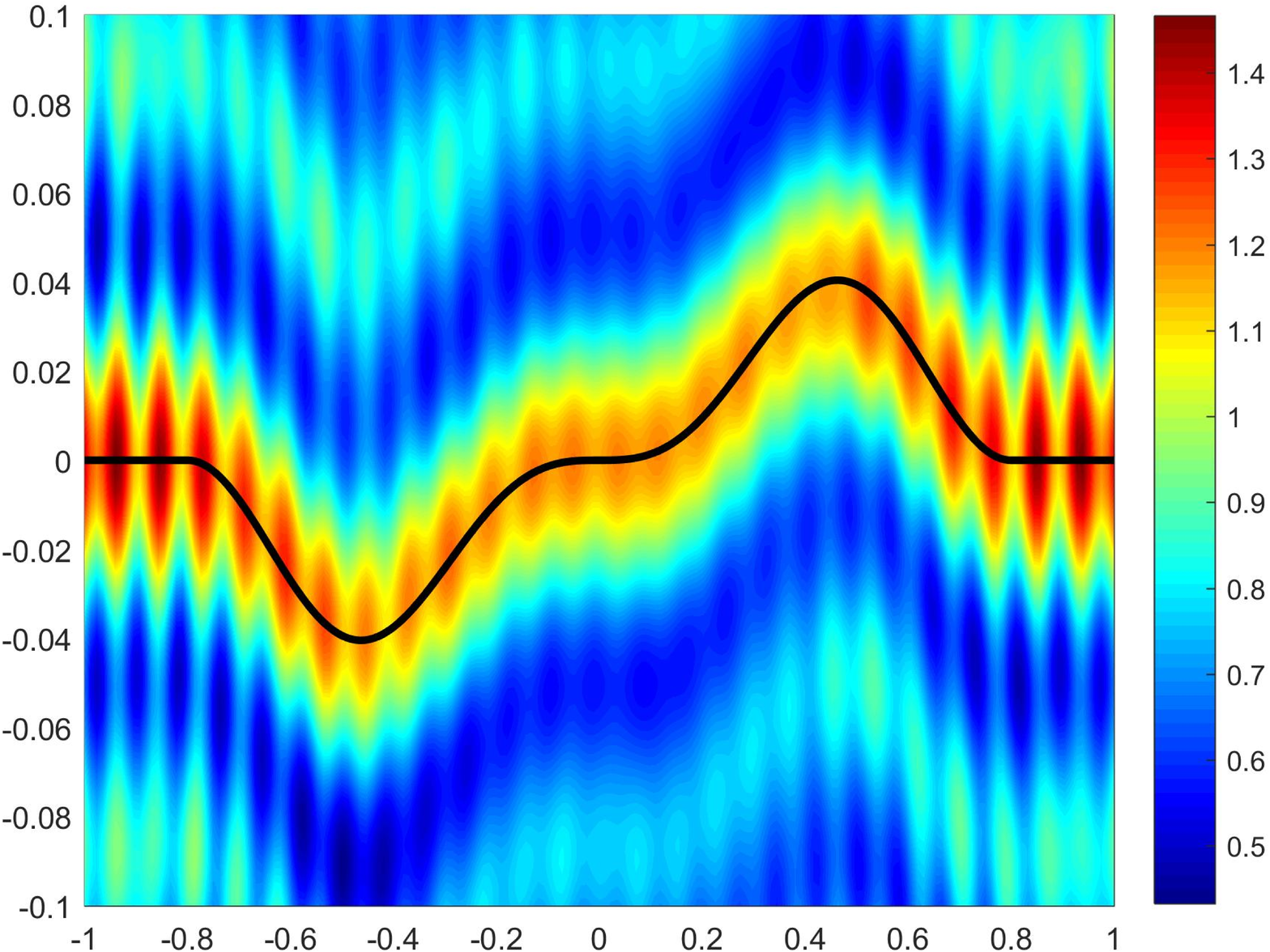}
\end{minipage}%
}%
\hfill
\subfigure[$10\%$ noise, $k_{+} = 40$, $k_{-} = 80$, $R = 2$]
{
\begin{minipage}[t]{0.3\textwidth}
\centering
\includegraphics[width=\textwidth]{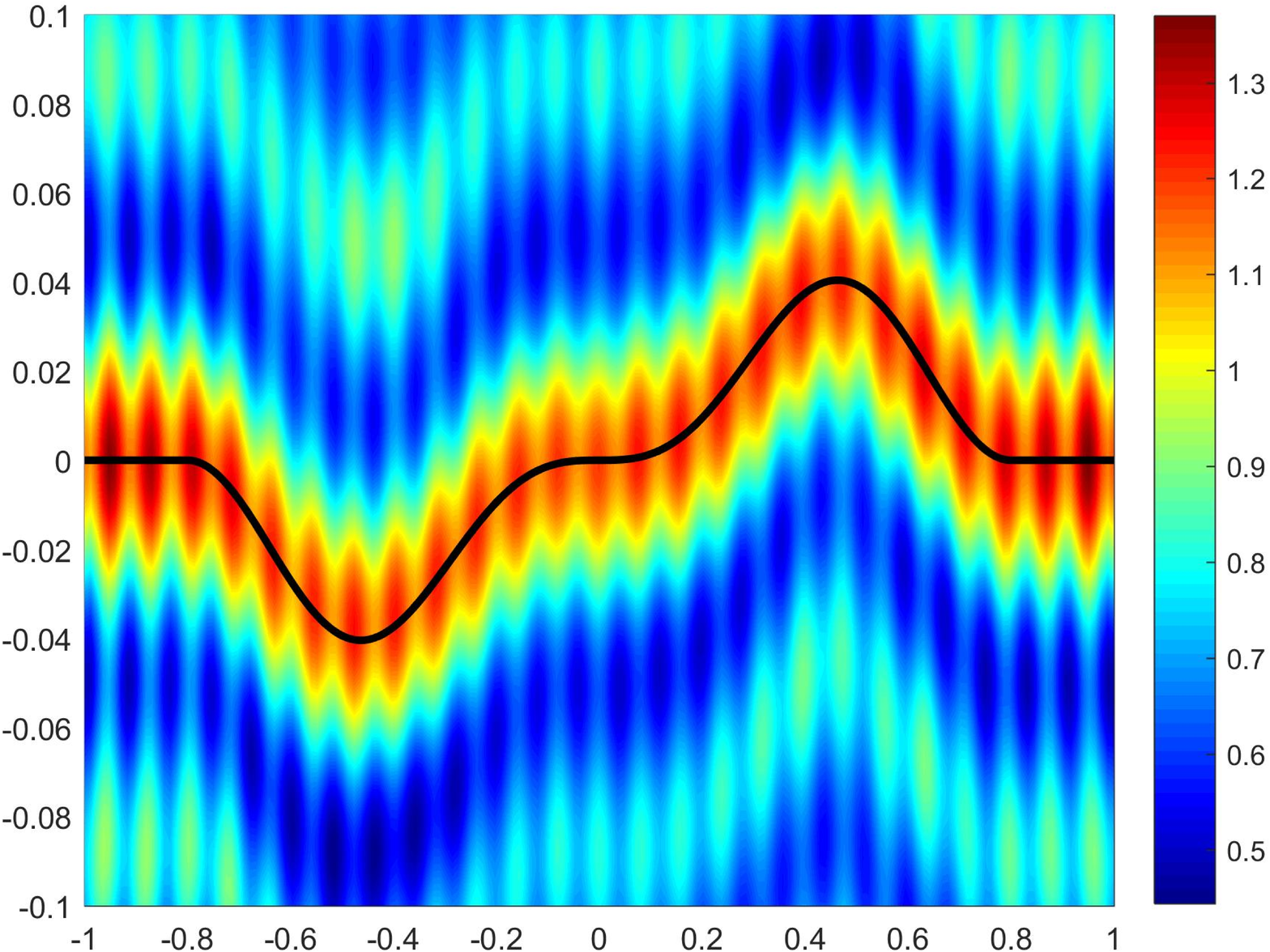}
\end{minipage}%
}
\hspace{0.15\textwidth}

\hspace{0.15\textwidth}
\subfigure[$10\%$ noise, $k_{+} = 40$, $k_{-} = 80$, $R = 3$]
{
\begin{minipage}[t]{0.3\textwidth}
\centering
\includegraphics[width=\textwidth]{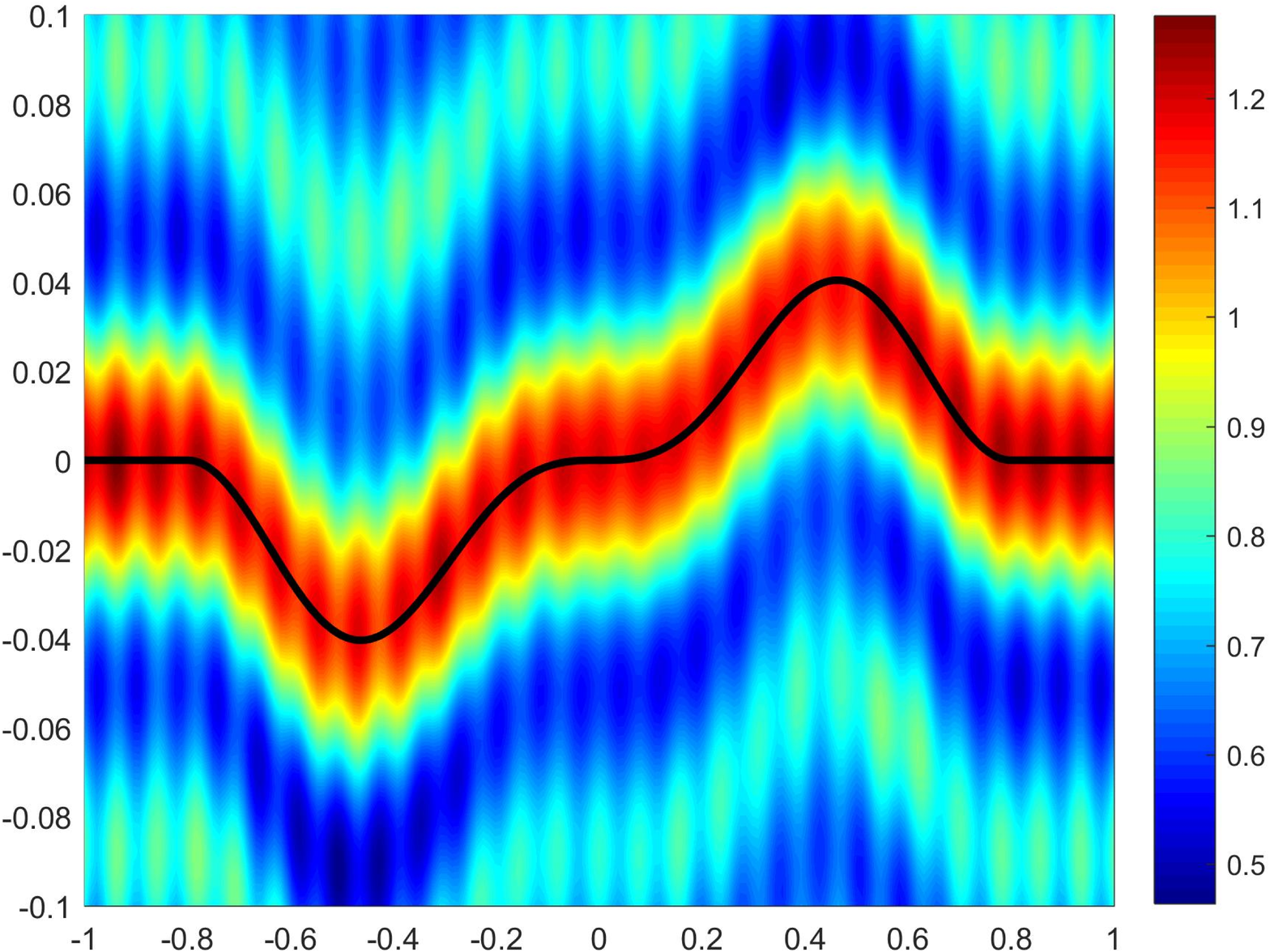}

\end{minipage}%
}
\hfill
\subfigure[$10\%$ noise, $k_{+} = 40$, $k_{-} = 80$]
{
\begin{minipage}[t]{0.3\textwidth}
\centering
\includegraphics[width=\textwidth]{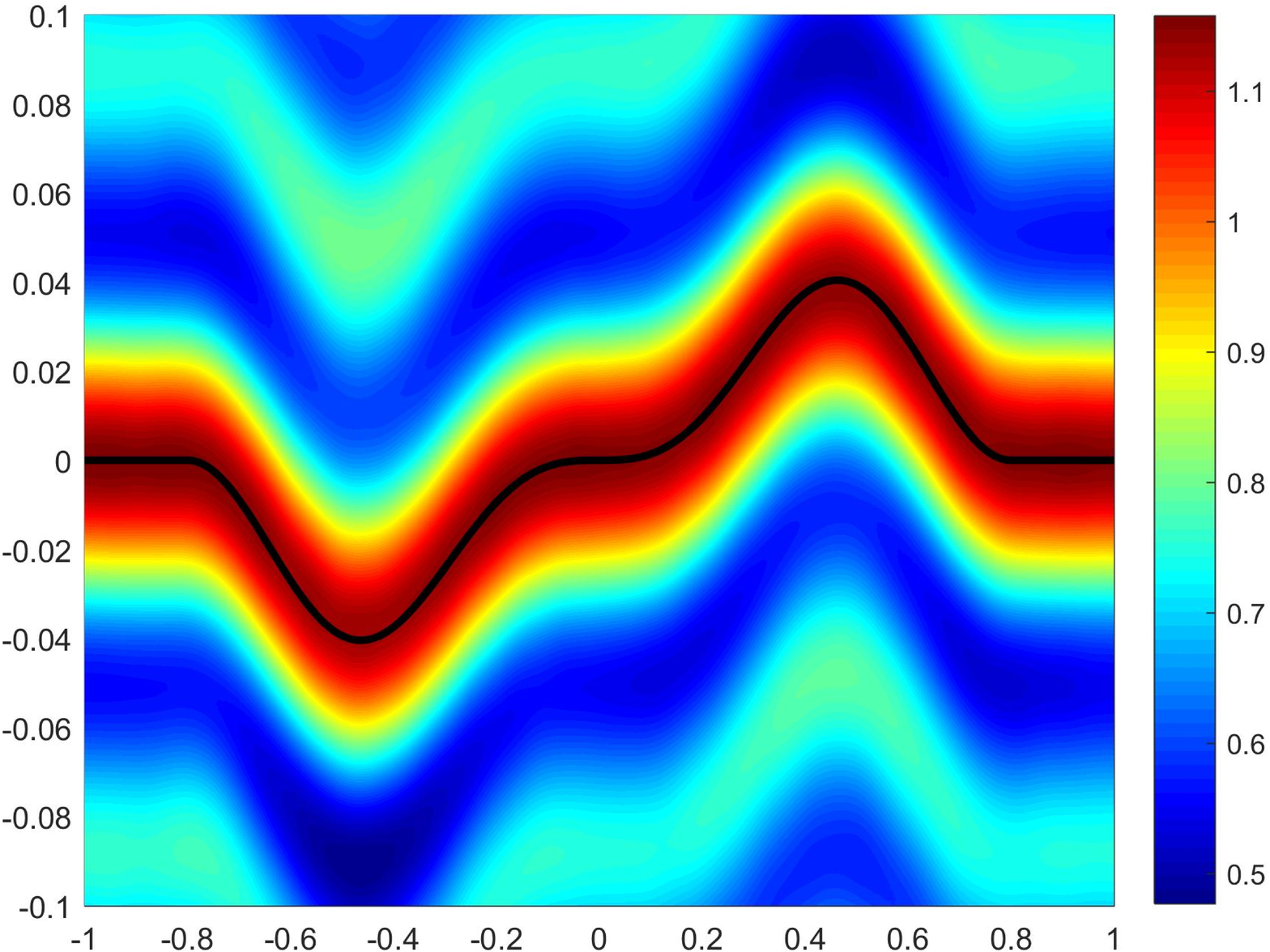}

\end{minipage}%
}
\hspace{0.15\textwidth}
\centering
\caption{(a), (b) and (c) show the imaging results of $I_{P}(z,R)$ with the measured phaseless
total-field data for different values of the radius $R$. (d) shows the imaging result of $I_{F}(z)$
with the measured phased far-field data. The solid line represents the actual curve.}\label{fig1}
\end{figure}

\textbf{Example 2.} We now investigate the effect of the noise ratio $\delta$ on the imaging results.
The locally rough surface considered is given by
\ben
h_{\Gamma}(x_1) = \begin{cases}
0.2\sin\left[(x_1^2-{16}/{25})^3\right]\sin(3\pi x_1)e^{-x_1^2}, &|x_1|\le 4/5,\\
0, & |x_1| > 4/5.
\end{cases}
\enn
We choose the wave numbers $k_+ =40$ and $k_-=80$. First, we consider the inverse problem (IP1).
The radius of the measurement circle $\partial B^+_R$ is set to be $R=3$. The numbers of the
measurement points and the incident directions are chosen to be $M_P = N_P = 300$.
Figure \ref{fig2} presents the imaging results of $I_{P}(z,R)$ from the measured phaseless
total-field data without noise, with $20\%$ noise and with $40\%$ noise, respectively. Next,
we consider the inverse problem (IP2). We choose the numbers of the measured observation
directions and the incident directions to be $M_F=N_F =100$. Figure \ref{fig3} presents
the imaging results of $I_{F}(z)$ from the measured phased far-field data without noise,
with $20\%$ noise and with $40\%$ noise, respectively.

\begin{figure}[htbp]
\centering
\subfigure[No noise, $k_{+} = 40$, $k_{-} = 80$, $R = 3$]
{
\begin{minipage}[t]{0.3\textwidth}
\centering
\includegraphics[width=\textwidth]{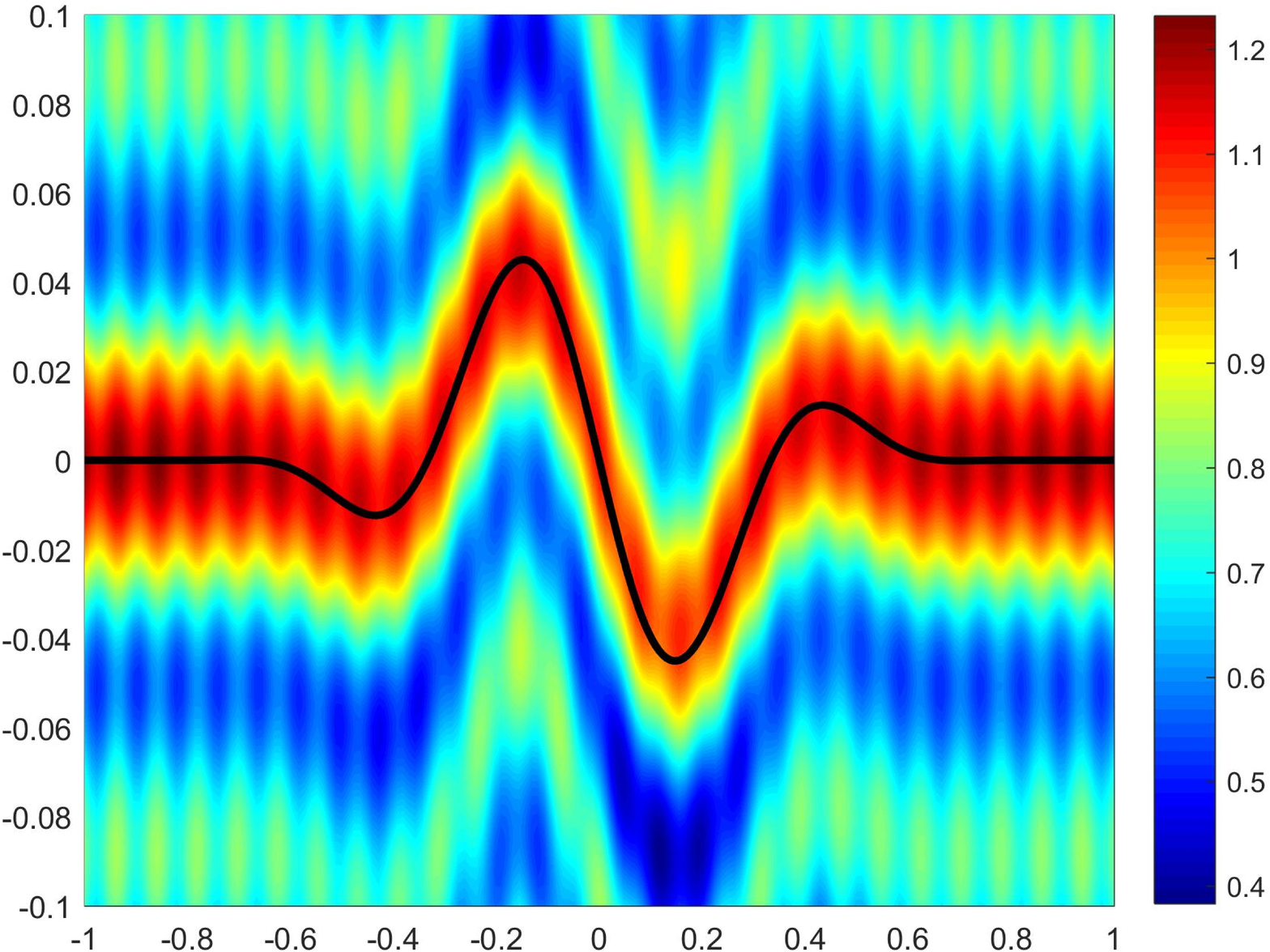}
\end{minipage}%
}%
\hfill
\subfigure[$20\%$ noise, $k_{+} = 40$, $k_{-} = 80$, $R = 3$]
{
\begin{minipage}[t]{0.3\textwidth}
\centering
\includegraphics[width=\textwidth]{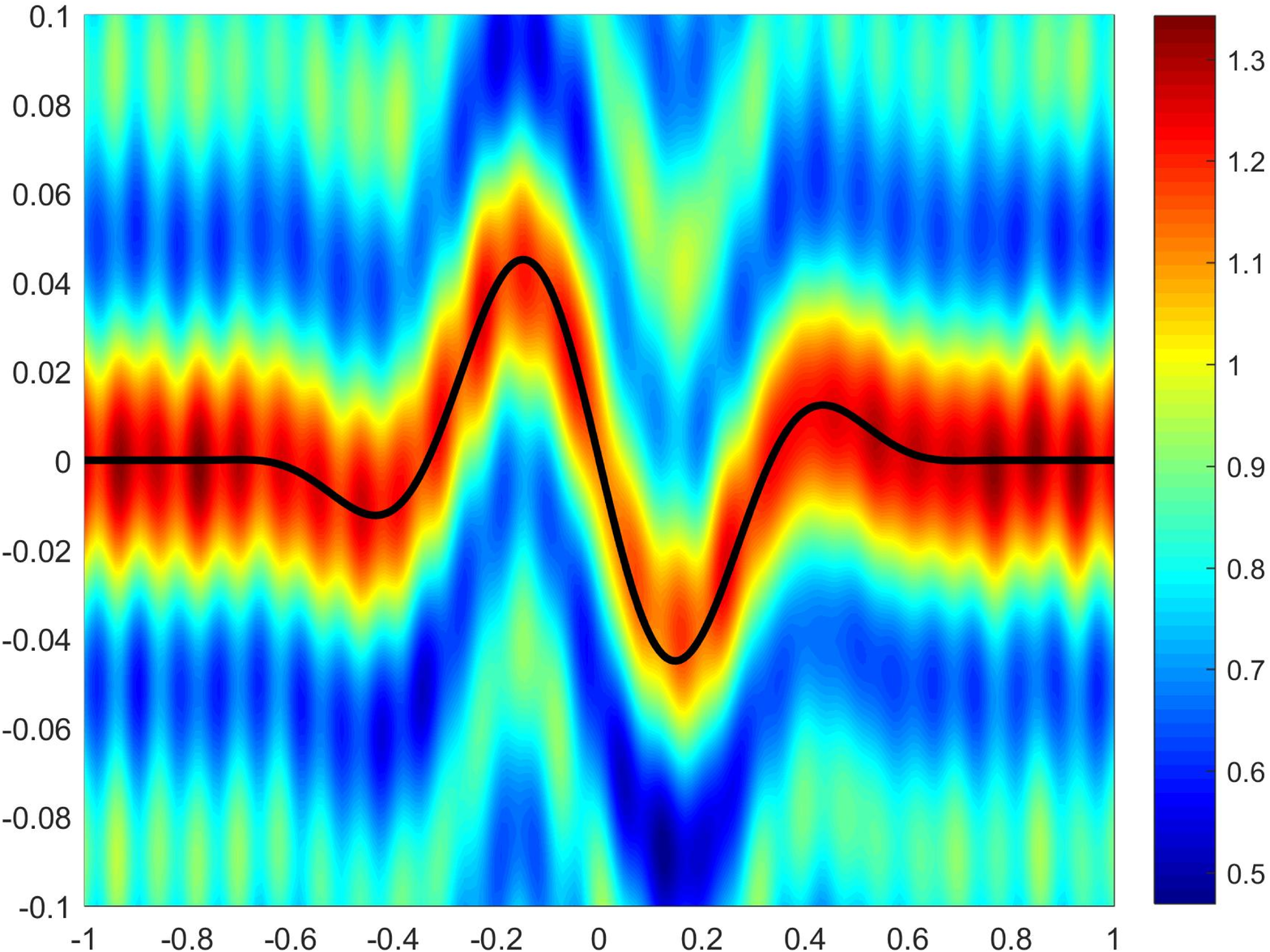}
\end{minipage}%
}
\hfill
\subfigure[$40\%$ noise, $k_{+} = 40$, $k_{-} = 80$, $R = 3$]
{
\begin{minipage}[t]{0.3\textwidth}
\centering
\includegraphics[width=\textwidth]{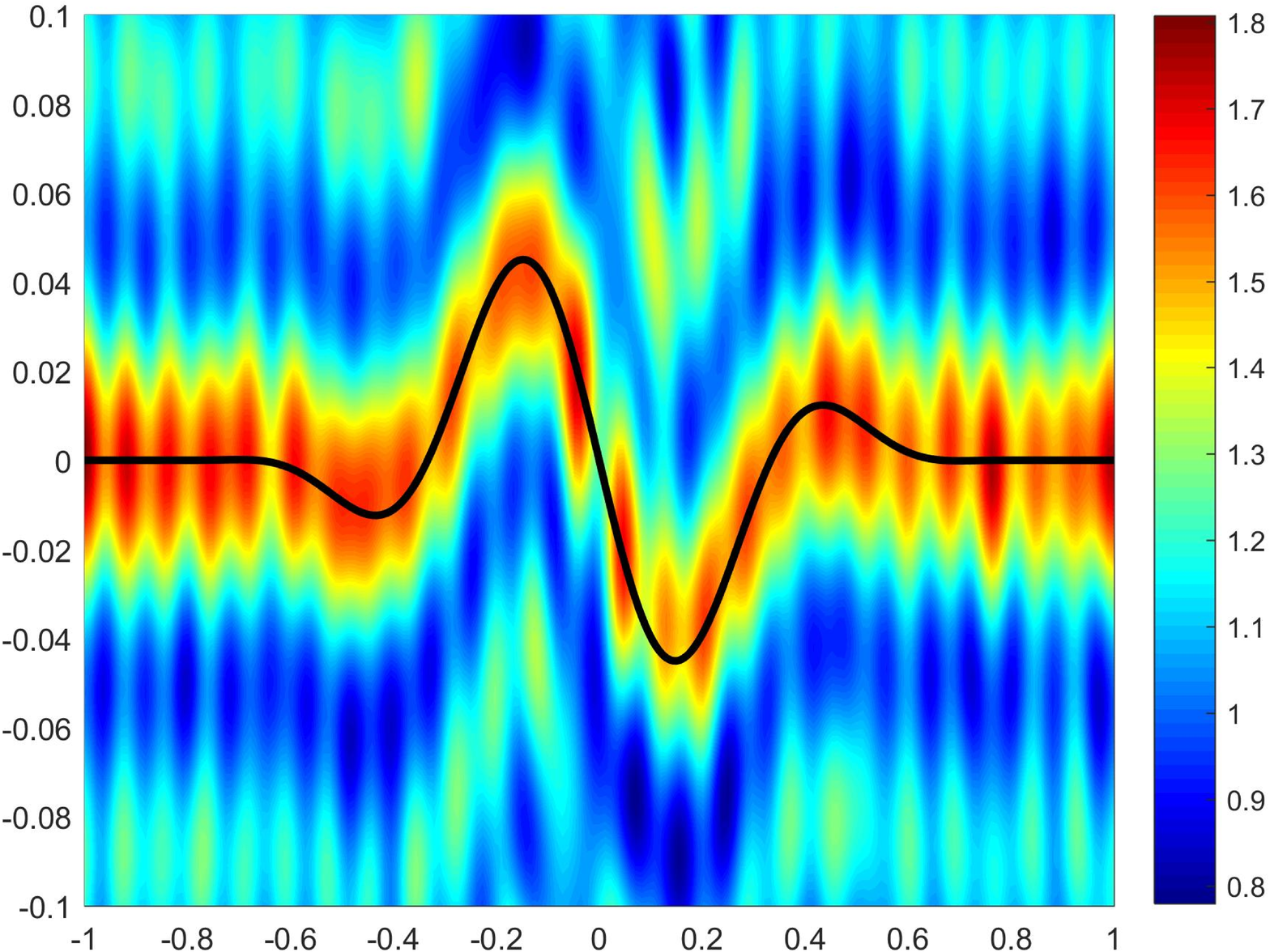}

\end{minipage}%
}%

\centering
\caption{Imaging results of $I_P(z,R)$ with the measured phaseless total-field data for
different noise ratios. The solid line represents the actual curve.}\label{fig2}
\end{figure}

\begin{figure}[htbp]
\centering
\subfigure[No noise, $k_{+} = 40$, $k_{-} = 80$]
{
\begin{minipage}[t]{0.3\textwidth}
\centering
\includegraphics[width=\textwidth]{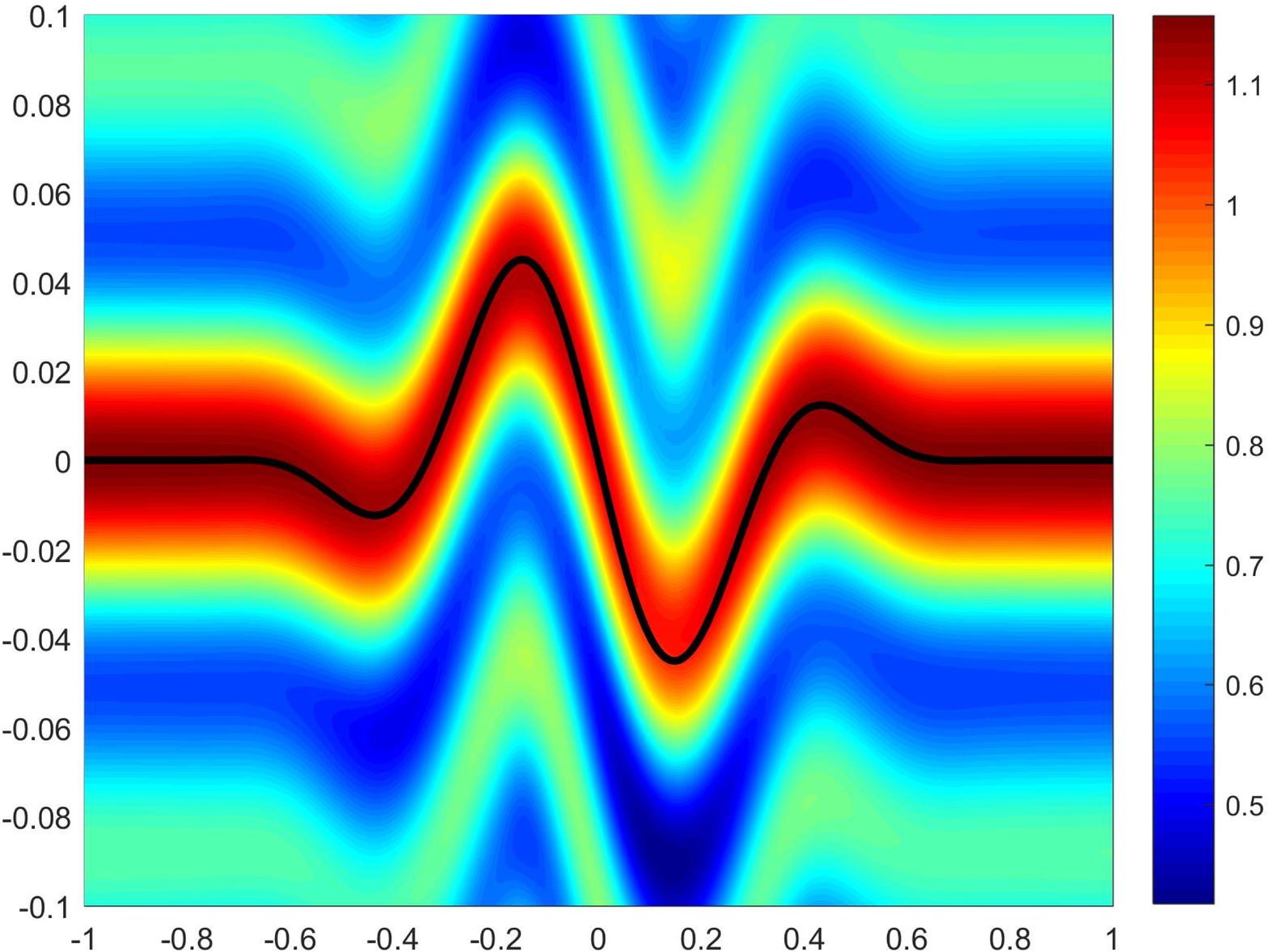}
\end{minipage}%
}%
\hfill
\subfigure[$20\%$ noise, $k_{+} = 40$, $k_{-} = 80$]
{
\begin{minipage}[t]{0.3\textwidth}
\centering
\includegraphics[width=\textwidth]{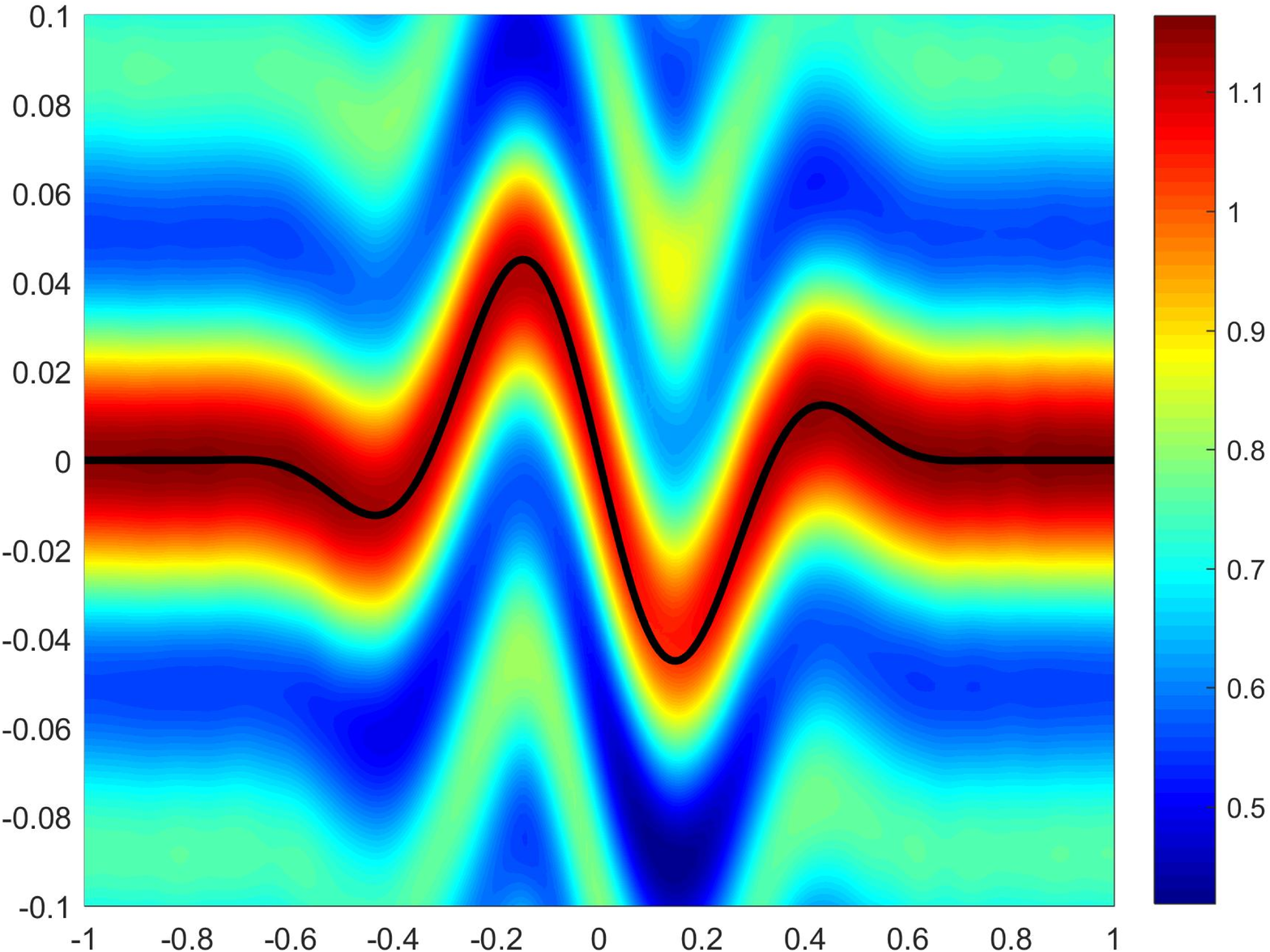}
\end{minipage}%
}
\hfill
\subfigure[$40\%$ noise, $k_{+} = 40$, $k_{-} = 80$]
{
\begin{minipage}[t]{0.3\textwidth}
\centering
\includegraphics[width=\textwidth]{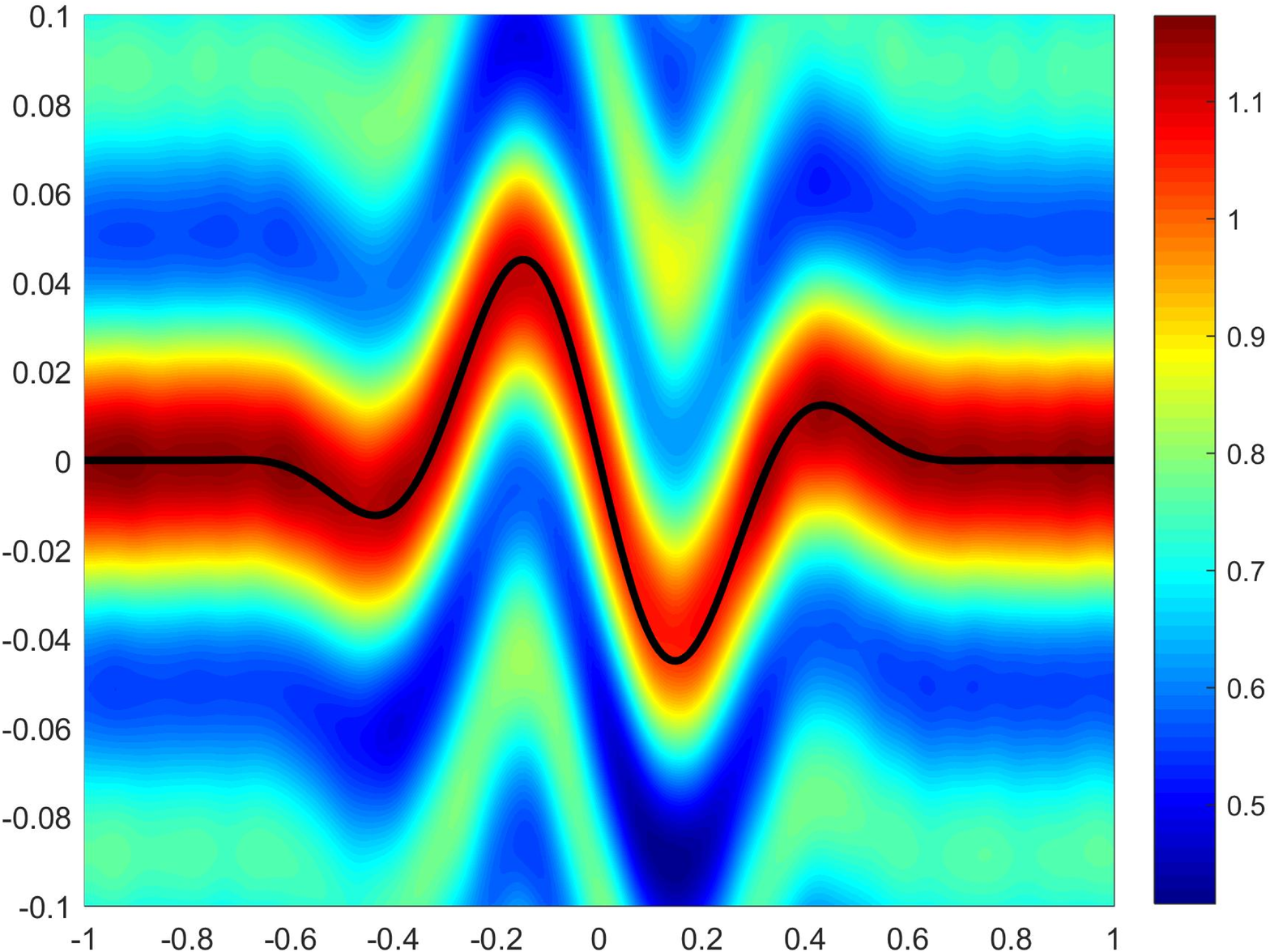}

\end{minipage}%
}%
\centering
\caption{Imaging results of $I_F(z)$ with the measured phased far-field data for different noise ratios.
The solid line represents the actual curve. }\label{fig3}
\end{figure}

\textbf{Example 3.} In this example, we compare the imaging results in the case $k_+>k_-$ with those in the case $k_+<k_-$.
The locally rough surface is given by
\ben
h_\Gamma(x_1) = \begin{cases}
0.2\exp{\left[{16}/({25x_1^2 - 16})\right]}[0.5+0.1\sin(16\pi x_1)], & |x_1| \le 4/5,\\
0, &|x_1| > 4/5.
\end{cases}
\enn
Here, $h_\Gamma(x_1)$ consists of a macroscale represented by
$0.1\exp{\left[{16}/({25x_1^2 - 16})\right]}$
and a microscale represented by
$0.02\exp{\left[{16}/({25x_1^2 - 16})\right]}\sin(16\pi x_1)$.
We choose the noise ratio to be $\delta=10\%$. First, we consider the inverse problem (IP1). The radius of the measurement circle $\partial B^+_R$ is chosen to be $R=3$. The numbers of the measurement points and the incident directions are set to be $M_P = N_P = 400$. Figures \ref{fig4}(a) and \ref{fig4}(b) present the imaging results of $I_{P}(z,R)$ with the measured phaseless total-field data with the pair of wave numbers
$(k_+,k_-)=(90,180)$, $(90,45)$, respectively. Next, we consider the inverse problem (IP2). We choose the numbers of
the measured observation directions
and the incident directions to be $M_F=N_F =100$. Figures \ref{fig4}(c) and \ref{fig4}(d) present the imaging results of $I_{F}(z)$ with the measured phased far-field data
with the pair of wave numbers
$(k_+,k_-)=(90,180)$, $(90,45)$, respectively.

\begin{figure}[htbp]
\centering
\hspace{0.15\textwidth}
\subfigure[$10\%$ noise, $k_{+} = 90$, $k_{-} = 180$, $R=3$]
{
\begin{minipage}[t]{0.3\textwidth}
\centering
\includegraphics[width=\textwidth]{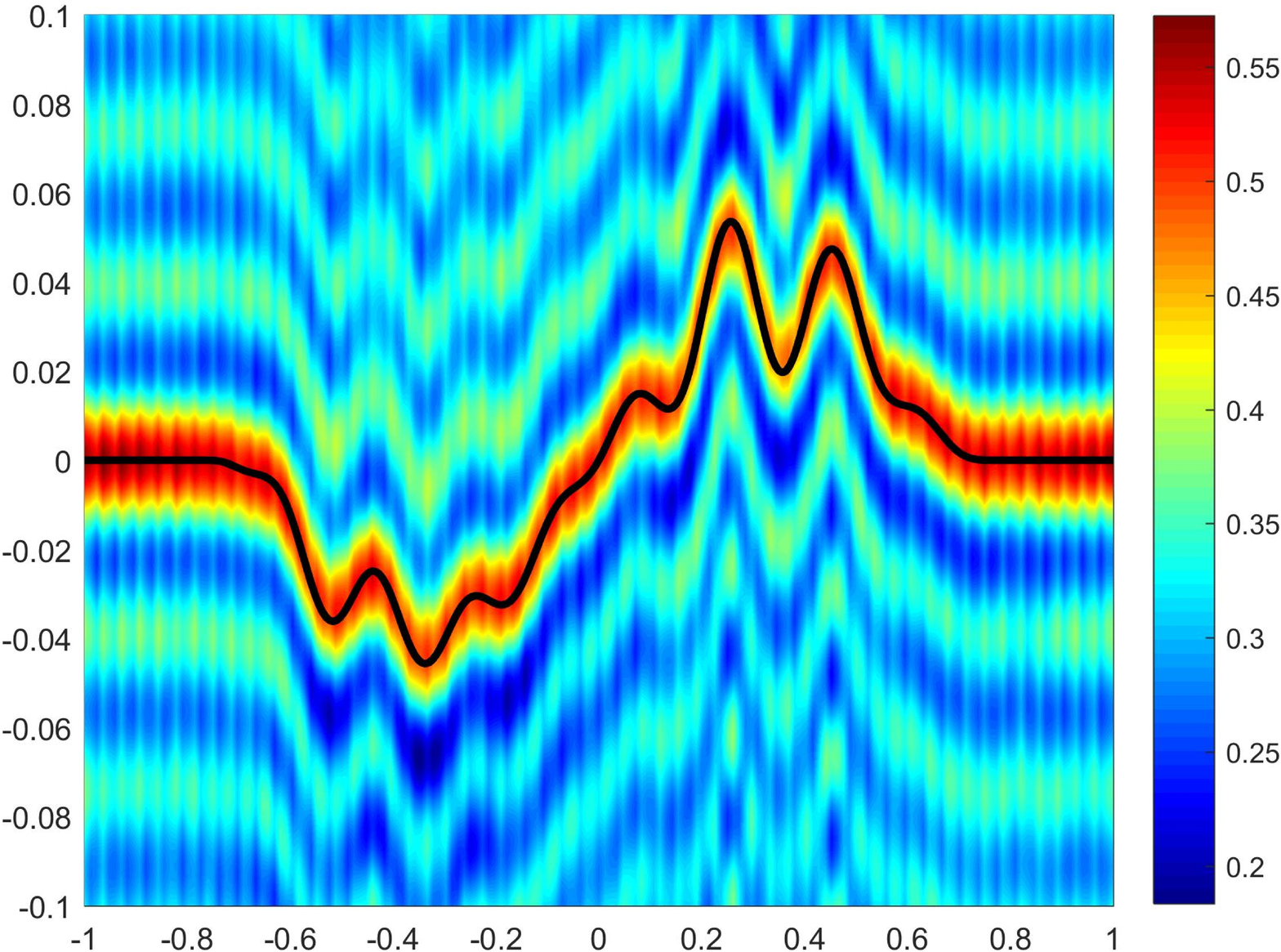}
\end{minipage}%
}%
\hfill
\subfigure[$10\%$ noise, $k_{+} = 90$, $k_{-} = 45$, $R=3$]
{
\begin{minipage}[t]{0.3\textwidth}
\centering
\includegraphics[width=\textwidth]{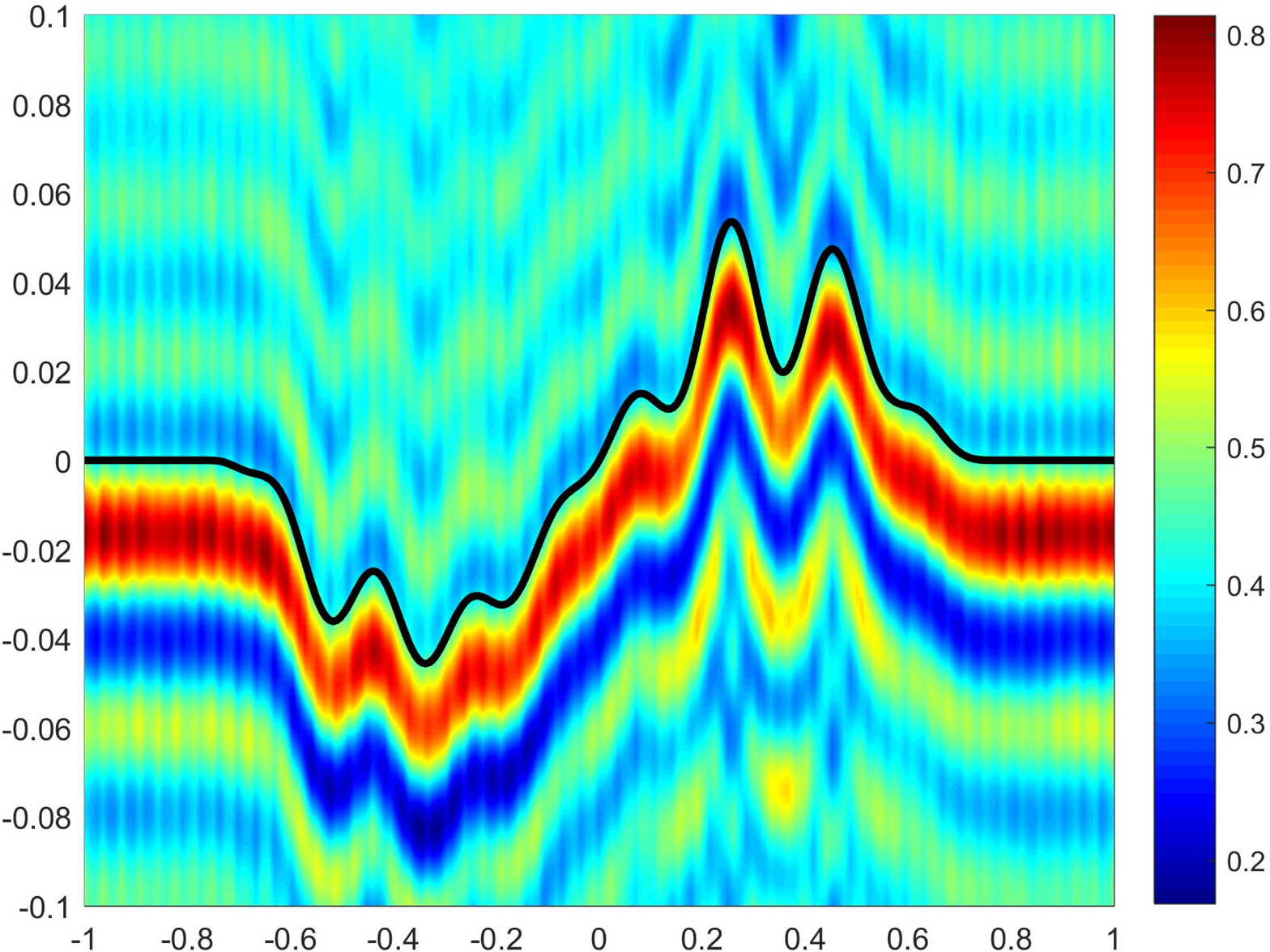}
\end{minipage}%
}
\hspace{0.15\textwidth}

\hspace{0.15\textwidth}
\subfigure[$10\%$ noise, $k_{+} = 90$, $k_{-} = 180$]
{
\begin{minipage}[t]{0.3\textwidth}
\centering
\includegraphics[width=\textwidth]{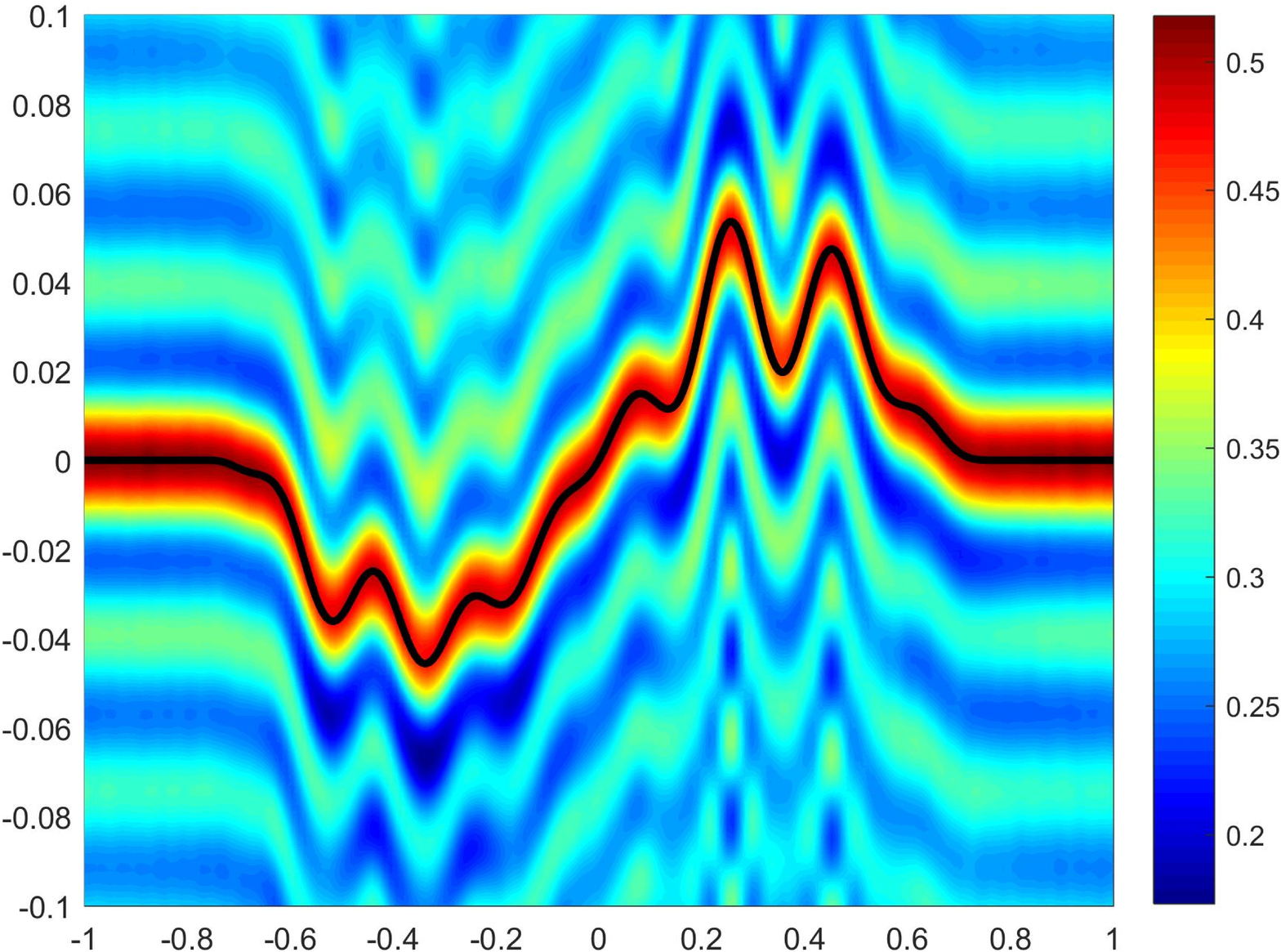}
\end{minipage}%
}%
\hfill
\subfigure[$10\%$ noise, $k_{+} = 90$, $k_{-} = 45$]
{
\begin{minipage}[t]{0.3\textwidth}
\centering
\includegraphics[width=\textwidth]{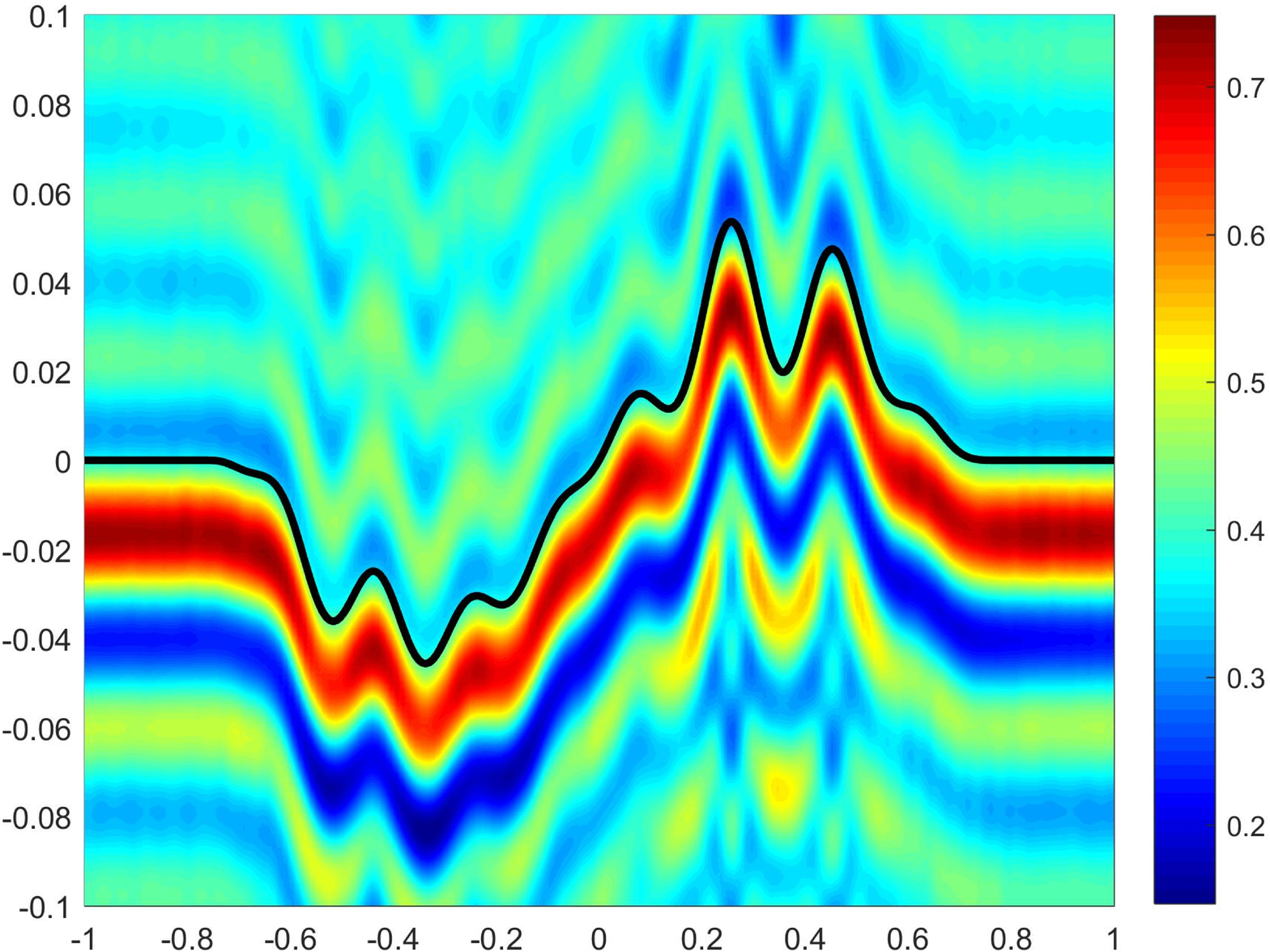}
\end{minipage}%
}%
\hspace{0.15\textwidth}
\centering
\caption{(a) and (b) show the imaging results of $I_{P}(z,R)$ with the measured phaseless total-field data. (c) and (d) show the imaging results of $I_{F}(z)$ with the measured phased far-field data. The solid line represents the actual curve.}\label{fig4}
\end{figure}

\textbf{Example 4.} In this example, we set the ratio $k_-/k_+$ to be a fixed number and
investigate the effect of the wave numbers $k_+,k_-$ on the imaging results.
The locally rough surface is chosen to be a multiscale curve given by
\begin{align*}
&h_\Gamma(x_1) = \\
&\left\{\begin{aligned}
&0.2\exp{\left[{16}/({25x_1^2 - 16})\right]}\left[0.5+0.1\sin(10\pi x_1) + 0.1\cos(8\pi x_1)\right]\sin(\pi x_1), &
|x_1| \le  \frac45,\\
&0,&
|x_1|>  \frac45.
\end{aligned}\right.
\end{align*}
Here, the function $h_\Gamma(x_1)$ has two scales: the macro scale is represented by the function
$0.1\exp{\left[{16}/({25x_1^2 - 16})\right]}\sin(\pi x_1)$,
and the micro scale is represented by the function
$0.2\exp{\left[{16}/({25x_1^2 - 16})\right]}\left[0.1\sin(10\pi x_1) + 0.1\cos(8\pi x_1)\right]\sin(\pi x_1)$.
The noise ratio is set to be $\delta = 10\%$.
First, we consider the inverse problem (IP1). The radius of the measurement circle $\partial B^+_R$ is set to be $R=3$. The numbers of the measurement points and the incident directions are chosen to be $M_P = N_P = 400$. Figure \ref{fig5} presents the imaging results of $I_{P}(z,R)$ with the measured phaseless total-field data with the pair of wave numbers $(k_+,k_-)=(60,30)$, $(90,45)$, $(120,60)$, respectively. Second, we consider the inverse problem (IP2). We choose the numbers of the measured observation directions and the incident directions to be $M_F=N_F =100$. Figure \ref{fig6} presents the imaging results of $I_{F}(z)$ with the measured phased far-field data with
the pair of wave numbers $(k_+,k_-)=(60,30)$, $(90,45)$, $(120,60)$,
respectively. From Figures \ref{fig5} and \ref{fig6}, it can be seen that the reconstruction result is getting
better with the wave numbers $k_+$ and $k_-$ getting larger.

\begin{figure}[htbp]
\centering
\subfigure[$10\%$ noise, $k_{+} = 60$, $k_{-} = 30$, $R=3$]
{
\begin{minipage}[t]{0.3\textwidth}
\centering
\includegraphics[width=\textwidth]{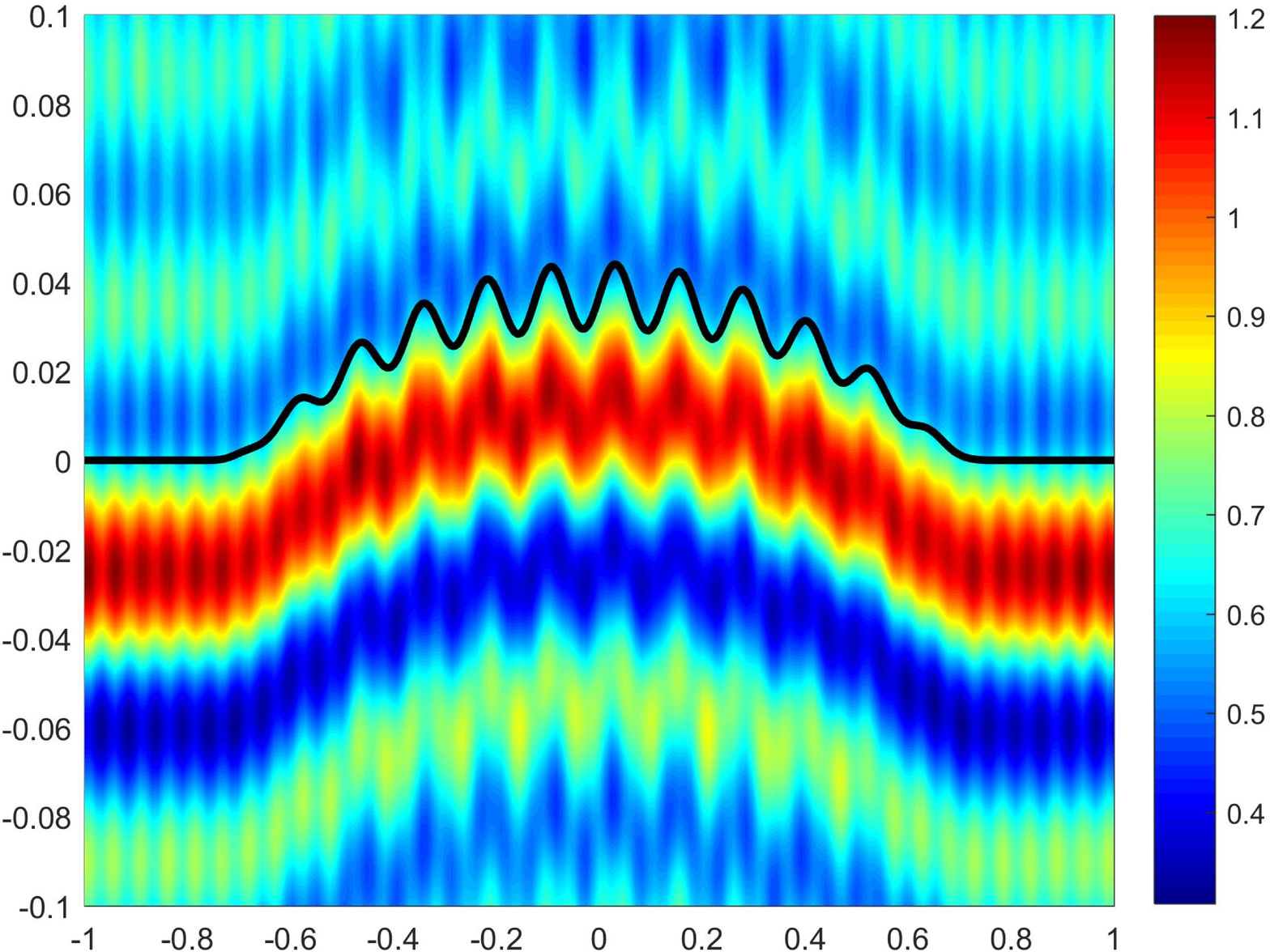}
\end{minipage}%
}%
\hfill
\subfigure[$10\%$ noise, $k_{+} = 90$, $k_{-} = 45$, $R=3$]
{
\begin{minipage}[t]{0.3\textwidth}
\centering
\includegraphics[width=\textwidth]{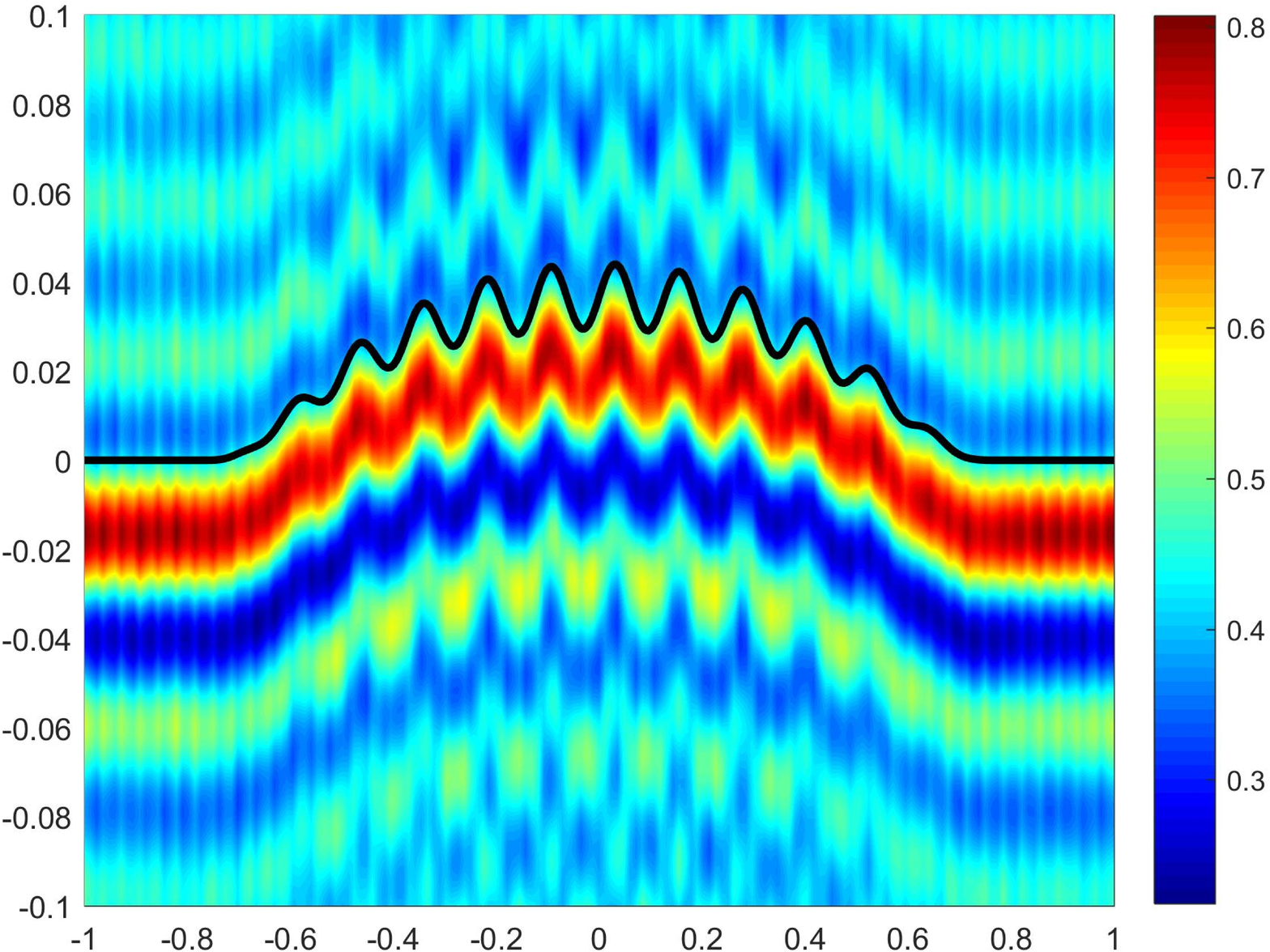}
\end{minipage}%
}
\hfill
\subfigure[$10\%$ noise, $k_{+} = 120$, $k_{-} = 60$, $R=3$]
{
\begin{minipage}[t]{0.3\textwidth}
\centering
\includegraphics[width=\textwidth]{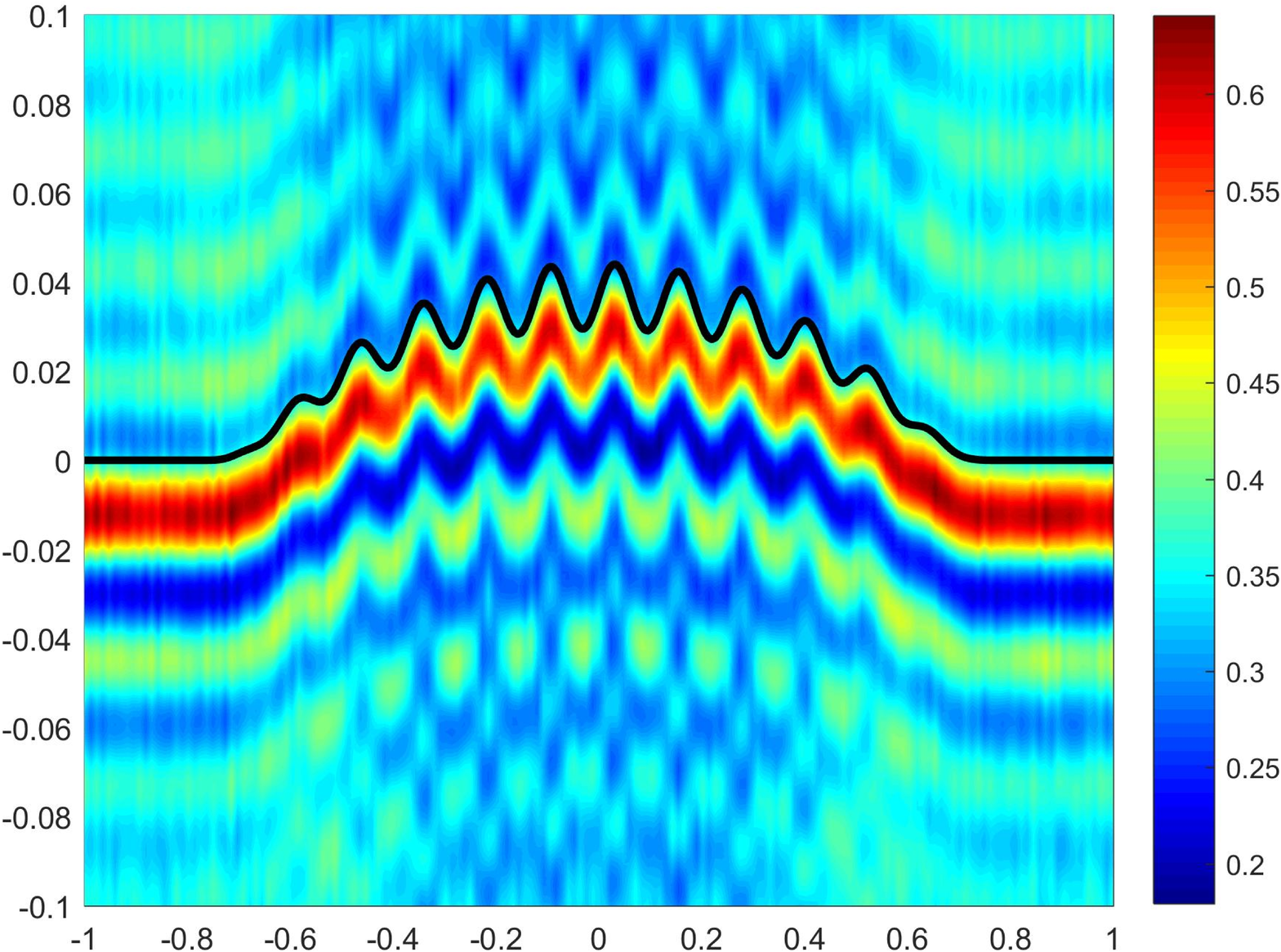}
\end{minipage}%
}%
\centering
\caption{Imaging results of $I_{P}(z,R)$ with the measured phaseless total-field data for different values of the wave numbers $k_+$ and $k_-$. The solid line represents the actual curve.}\label{fig5}
\end{figure}

\begin{figure}[htbp]
\centering
\subfigure[$10\%$ noise, $k_{+} = 60$, $k_{-} = 30$]
{
\begin{minipage}[t]{0.3\textwidth}
\centering
\includegraphics[width=\textwidth]{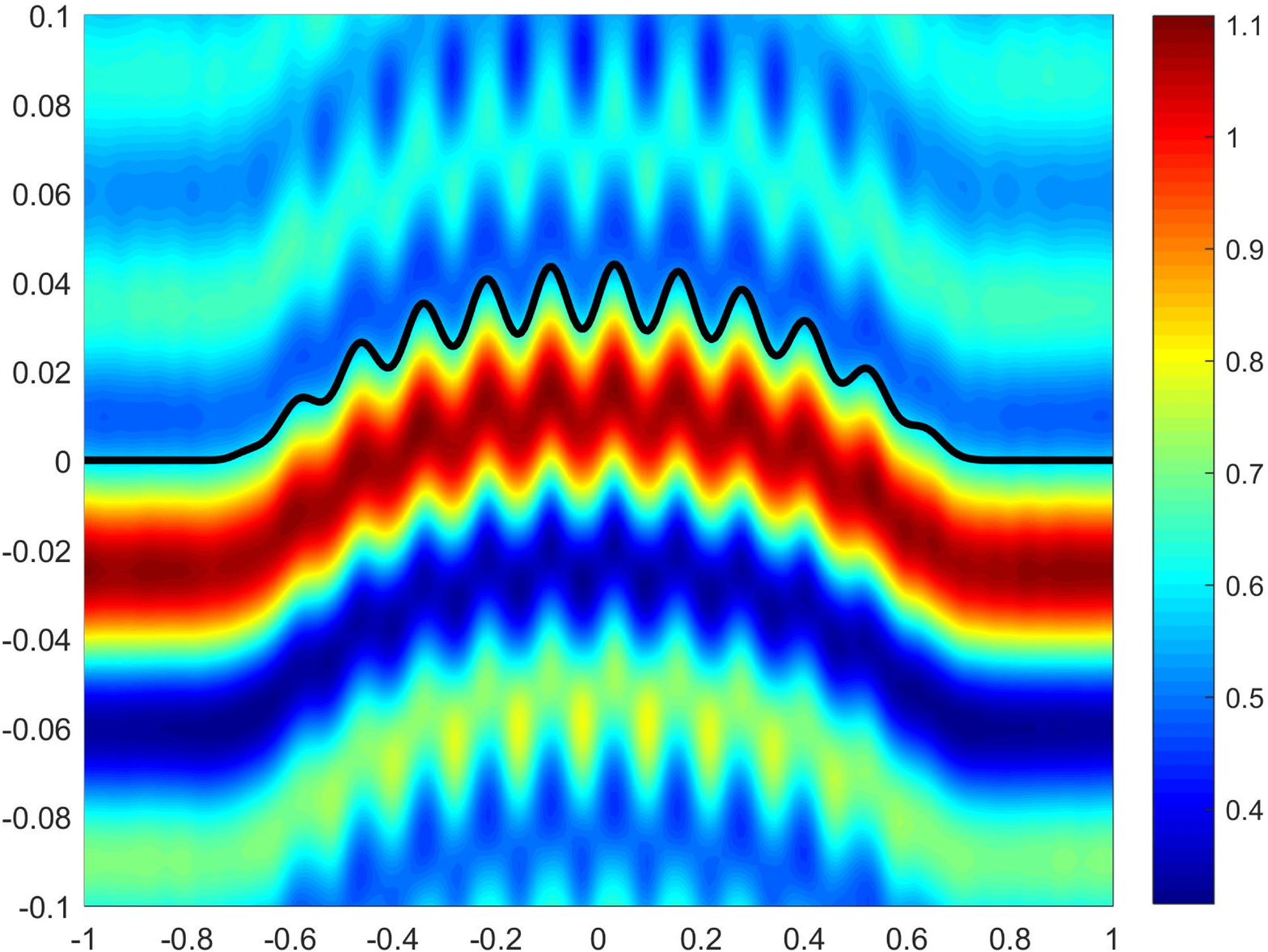}
\end{minipage}%
}%
\hfill
\subfigure[$10\%$ noise, $k_{+} = 90$, $k_{-} = 45$]
{
\begin{minipage}[t]{0.3\textwidth}
\centering
\includegraphics[width=\textwidth]{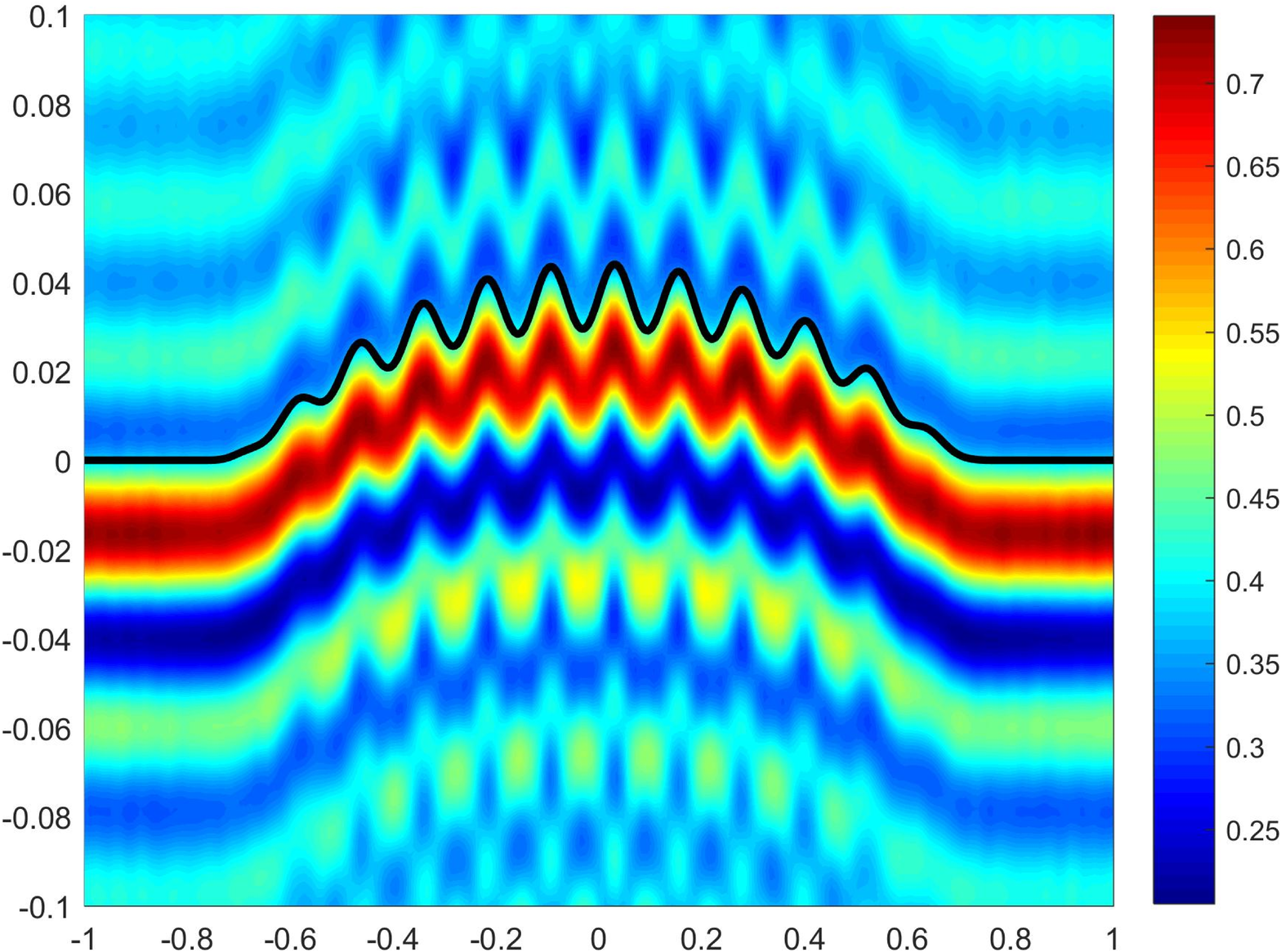}
\end{minipage}%
}
\hfill
\subfigure[$10\%$ noise, $k_{+} = 120$, $k_{-} = 60$]
{
\begin{minipage}[t]{0.3\textwidth}
\centering
\includegraphics[width=\textwidth]{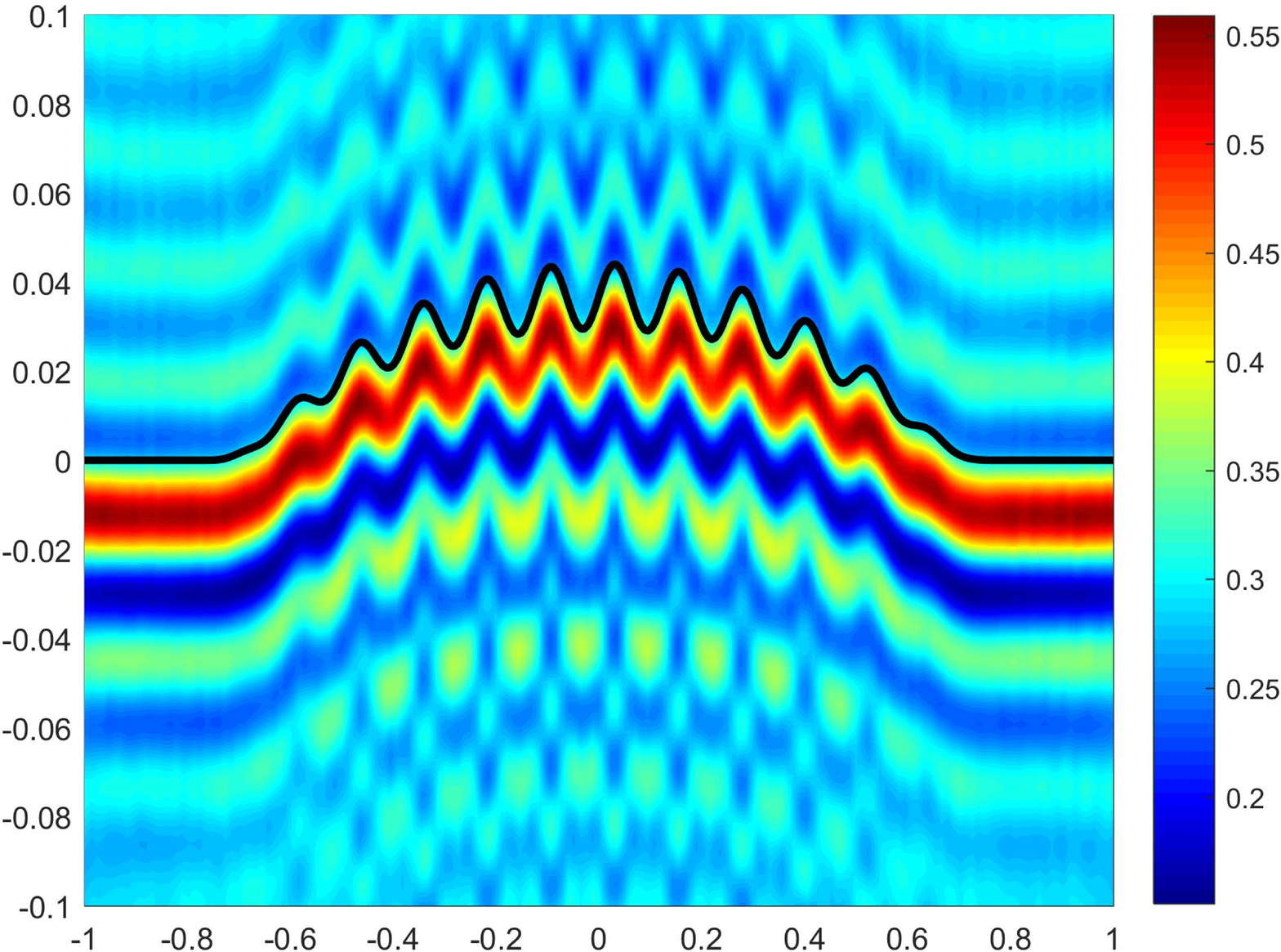}
\end{minipage}%
}%
\centering
\caption{Imaging results of $I_{F}(z)$ with the measured phased far-field data for different values of the wave numbers $k_+$ and $k_-$. The solid line represents the actual curve.}\label{fig6}
\end{figure}

\section{Conclusion}\label{con}

In this paper, we considered the problem of inverse scattering of time-harmonic acoustic plane waves
by a locally rough interface in a two-layered medium in 2D.
We have developed the direct imaging method with phaseless total-field data
and the direct imaging method with phased far-field data for reconstructing the penetrable locally rough interface.
We have also given the theoretical analysis of the proposed methods by studying the asymptotic
properties of relevant oscillatory integrals.
In doing so, an important role is played by the uniform far-field asymptotic properties of the scattered
wave for the acoustic scattering problem in the two-layered medium obtained in our recent work \cite{LYZZ3}.
Through various numerical experiments, it has been shown that our methods are effective for both
cases $k_{+}>k_{-}$ and $k_{+}<k_{-}$.
Moreover, for the considered scattering model,
it is interesting to study uniqueness of the inverse scattering problem in a two-layered medium
with a locally rough interface in 2D associated with phaseless total-field data and
with phased far-field data, which is still open and challenging.

\section*{Acknowledgments}

We thank Professor Wangtao Lu at the Zhejiang University for helpful and constructive
discussions on the perfectly matched layer-based boundary integral equation method proposed in \cite{LP_17}.
This work was partially supported by the National Key R \& D Program of China (2018YFA0702502),
China Postdoctoral Science Foundation Grant 2022M720158, Beijing Natural Science Foundation Z210001,
the NNSF of China grants 11961141007, 61520106004 and 12271515,
Microsoft Research of Asia, and Youth Innovation Promotion Association CAS.

\appendix

\renewcommand{\theequation}{\Alph{section}.\arabic{equation}}

\section{Proofs of Lemmas \ref{Le:3} and \ref{NLe:6}}\label{sec4}

\begin{proof}[Proof of Lemma \ref{Le:3}]
By a straightforward calculation, we have
\begin{align*}
I(\lambda) &= \int^{+\infty}_{-\infty} e^{-i\lambda\frac{\eta^2}2} d\eta-\int^{+\infty}_{b} e^{-i\lambda\frac{\eta^2}2}d\eta- \int^{a}_{-\infty} e^{-i\lambda\frac{\eta^2}2}d\eta\\
&=:I_1(\lambda)+I_2(\lambda)+I_3(\lambda).
\end{align*}
It follows from \cite[the last formula on page 98]{O97} that
\ben
I_1(\lambda) = 2 \int^{+\infty}_{0}\frac{e^{-i\lambda t}}{\sqrt{2t}}dt
=  \frac{e^{-i\frac{\pi}4}\sqrt{2\pi}}{\sqrt{\lambda}}.
\enn
Further, by the change of variable $\eta= \sqrt {2t}$ and an integration by parts, we have
\[
|I_2(\lambda)|=\left|\int^{+\infty}_{\frac{b^2}2} e^{-i\lambda t}\frac{1}{\sqrt{2t}}dt\right|
=\left|\frac{ie^{-i\lambda t}}{\lambda\sqrt{2t}}\bigg|^{+\infty}_{t=\frac{b^2}2}+ \frac{i}{\lambda}\int^{+\infty}_{\frac{b^2}2} e^{-i\lambda t}\frac{1}{2\sqrt 2t^{3/2}}dt\right|
\leq \frac{2}{\lambda b}.
\]
Similarly as the estimate of $I_2(\lambda)$, it can be deduced that $|I_3(\lambda)|\le {2}/({\lambda|a|})$.
Thus the statement of this lemma is obtained by the above discussions.
\end{proof}

\begin{proof}[Proof of Lemma \ref{NLe:6}]
Let $\eta\in[a,b]$ and define $g(\eta):= f{'}(\eta)\eta-(f(\eta)-f(0))$.
Now we claim that
\be\label{Fne20}
|g(\eta)|\le C(1+\|p\|_{C^3[a,b]})^3\|q\|_{C^3[a,b]}(1+t)^2\eta^2,\quad \eta\in[a,b].
\en
In fact, it is clear that $g(0)=g{'}(0)=0$. Thus for any $\eta\in[a,b]$, there exists
$\eta_1 = \theta \eta$ with some
$\theta \in (0,1)$ such that
\begin{align}\label{eq11}
g(\eta)= \frac{1}2g{''}(\eta_1)\eta^2= \frac{1}2\left[f{'''}(\eta_1)\theta\eta+f{''}(\eta_1)\right]\eta^2.
\end{align}
Hence, if $|\eta| \le ({1+t})^{-1}$, we have
\be\label{eq25}
|g(\eta)|\le C(1+\|p\|_{C^3[a,b]})^3\|q\|_{C^3[a,b]}(1+t)^2\eta^2.
\en
On the other hand, if $ |\eta| > ({1+t})^{-1}$, we have
\begin{align}\label{eq26}
\left|\frac{g(\eta)}{\eta^2}\right|&=\left|\frac{ f{'}(\eta)\eta-(f(\eta)-f(0))}{\eta^2}\right|
\le |f{'}(\eta)|(1+t)+|f(\eta)-f(0)|(1+t)^2\nonumber\\
&\le C(1+\|p\|_{C^1[a,b]})\|q\|_{C^1[a,b]}(1+t)^2.
\end{align}
Therefore, it follows from (\ref{eq25}) and (\ref{eq26}) that (\ref{Fne20}) holds.

Define the function $h(\eta):= [f(\eta)-f(0)]/\eta$ for $\eta \in[a,b]\ba\{0\}$. Since $g(\eta)=h'(\eta)\eta^2$
for $\eta \in[a,b]\ba\{0\}$, it easily follows from (\ref{Fne20}) and (\ref{eq11}) that
$h$ and its derivative can
be continuously extended from $[a,b]\ba\{0\}$ to $[a,b]$, and $h$ has the estimate
\begin{align}\label{eq32}
\|h\|_{C^1[a,b]}\le C(1+\|p\|_{C^3[a,b]})^3\|q\|_{C^3[a,b]}(1+t)^2.
\end{align}
On the other hand, an integration by parts gives that
\begin{align*}
I(\lambda)&= \frac{i}{\lambda}\left[e^{-i\lambda\frac{\eta^2}{2}}h(\eta)\bigg|^b_{\eta=a}-\int^{b}_a e^{-i\lambda\frac{\eta^2}{2}}h'(\eta)d\eta\right].
\end{align*}
This, together with (\ref{eq32}), implies that the statement of this lemma holds.
\end{proof}

\end{document}